\documentclass{amsart}
\usepackage{subfigure}
\usepackage{color, amsmath, amsfonts, amssymb, epsfig, amscd, graphicx, latexsym, latexsym, amsthm, etex, synttree, tikz,comment}
\usepackage[latin5]{inputenc}
\usepackage{young}

\newcommand{\Nn}{\mathbb N}

\newcommand{\fp}{\mbox{\rm {fp}}}

\newcommand{\Tt}{\mathcal T}

\newcommand{\scp}{\mathfrak{Sch}}
\newcommand{\crit}{\mathfrak{Crit}}
\newcommand{\ess}{\mathfrak{Ess}}

\newcommand{\fr}{\mathfrak{r}}
\newcommand{\fh}{\mathfrak{h}}

\newcommand{\slide}{\mbox{\rm sl}}
\newcommand{\sd}{\mbox{\rm sld}}

\newcommand{\fn}{\mbox{\rm fn}}
\newcommand{\hn}{\mbox{\rm hn}}
\newcommand{\gfn}{\mbox{\rm gfn}}
\definecolor{tur}{rgb}{1,0.5,0}

\font\co=lcircle10

\def\jr{\smash{	\raise2pt\hbox{\co \rlap{\rlap{\char'005} \char'007}}
		\raise6pt\hbox{\rlap{\vrule height2pt}}
		\raise2pt\hbox{\rlap{\hskip4pt \vrule height0.4pt depth0pt
                 width2.5pt}}
		\raise2pt\hbox{\rlap{\hskip-6pt \vrule height.4pt depth0pt
                 width2.2pt}}
		\lower4pt\hbox{\rlap{\vrule height2.5pt}}}}
\def\je{\smash{\raise2pt\hbox{\co \rlap{\rlap{\char'005}
                \phantom{\char'007}}}\raise6pt\hbox{\rlap{\vrule height2pt}}
	       \raise2pt\hbox{\rlap{\hskip-6pt \vrule height.4pt depth0pt
                width2.2pt}}}}
\def\+{\smash{\lower4pt\hbox{\rlap{\vrule height12pt}}
                \raise2pt\hbox{\rlap{\hskip-6pt \vrule height.4pt depth0pt
                width12.5pt}}}}
\def\er{\smash{	\raise2pt\hbox{\co \rlap{\phantom{\rlap{\char'005}} \char'007}}
		\raise6pt\hbox{\rlap{\phantom{\vrule height2pt}}}
		\raise2pt\hbox{\rlap{\hskip4pt \vrule height0.4pt depth0pt
                 width2.5pt}}
		\raise2pt\hbox{\rlap{\phantom{%
		 \hskip-6pt \vrule height.4pt depth0pt width2.2pt}}}
		\lower4pt\hbox{\rlap{\vrule height2.5pt}}}}
\def\hor{\smash{\lower2pt\hbox{\rlap{\phantom{\vrule height10pt}}}
                \raise2pt\hbox{\rlap{\hskip-6pt \vrule height.4pt depth0pt
                width12.5pt}}}}
\def\ver{\smash{\lower2pt\hbox{\rlap{\vrule height10pt}}
                \raise2pt\hbox{\rlap{\phantom{%
		\hskip-6pt \vrule height.4pt depth0pt width12.5pt}}}}}

 	{
	 \def\*{\makebox[0ex]{\footnotesize$+\,$}}%
	 \begin{array}{|*{#1}{@{\ \;\:}c|}}}
	{\end{array}}

	{
	 \begin{array}{c*{#1}{@{\ \ \;}c}}}
	{\end{array}}

	{
	 \ \ \begin{array}{@{}l|
			*{#1}{c@{\quad\ }}@{\hspace{-1.7ex}}|
			*{#2}{c@{\quad\ }}@{\hspace{-1.7ex}}|
			*{#3}{c@{\quad\ }}@{\hspace{-1.7ex}}|
			*{#4}{c@{\quad\ }}@{\hspace{-1.7ex}}|}
	}
	{\end{array}}

	{
	 \ \ \begin{array}{@{}l|
			*{#1}{c@{\quad\ }}@{\hspace{-1.7ex}}|}
	}
	{\end{array}}

   \newtheorem{thm}{Theorem}[section]
   \newtheorem{pro}{Proposition}[section]
   \newtheorem{lem}{Lemma}[section]
   \newtheorem{cor}{Corollary}[section]

\theoremstyle{definition}
   \newtheorem{defn}{Definition}[section]
   \newtheorem{ex}{Example}[section]
   \newtheorem{rem}{Remark}[section]

\newcommand{\cellsize}{7.142}
\newlength{\cellsz} 
\newsavebox{\cell}
\sbox{\cell}{\begin{picture}(\cellsize,\cellsize)
\put(0,0){\line(1,0){\cellsize}}
\put(0,0){\line(0,1){\cellsize}}
\put(\cellsize,0){\line(0,1){\cellsize}}
\put(0,\cellsize){\line(1,0){\cellsize}}
\end{picture}}
\newcommand\cellify[1]{\def\thearg{#1}\def\nothing{}%
\ifx\thearg\nothing
\vrule width0pt height\cellsz depth0pt\else
\hbox to 0pt{\usebox{\cell} \hss}\fi%
\vbox to \cellsz{
\vss
\hbox to \cellsz{\hss$#1$\hss}
\vss}}
\newcommand\tableau[1]{\vtop{\let\\\cr
\baselineskip -16000pt \lineskiplimit 16000pt \lineskip 0pt
\ialign{&\cellify{##}\cr#1\crcr}}}

\newcommand{\cellsizea}{14.142}
\newlength{\cellsza} 
\newsavebox{\cella}
\sbox{\cella}{\begin{picture}(\cellsizea,\cellsizea)
\put(0,0){\line(1,0){\cellsizea}}
\put(0,0){\line(0,1){\cellsizea}}
\put(\cellsizea,0){\line(0,1){\cellsizea}}
\put(0,\cellsizea){\line(1,0){\cellsizea}}
\end{picture}}
\newcommand\cellifya[1]{\def\thearg{#1}\def\nothing{}%
\ifx\thearg\nothing
\vrule width0pt height\cellsz1 depth0pt\else
\hbox to 0pt{\usebox{\cella} \hss}\fi%
\vbox to \cellsza{
\vss
\hbox to \cellsza{\hss$#1$\hss}
\vss}}
\newcommand\tableaux[1]{\vtop{\let\\\cr
\baselineskip -16000pt \lineskiplimit 16000pt \lineskip 0pt
\ialign{&\cellifya{##}\cr#1\crcr}}}

\title[]{Tower Diagrams and Pieri's Rule}
\author{Olcay Co\c{s}kun}
\author{M\"uge Ta\c{s}k\i n}
\address{Olcay Coşkun: Boğaziçi Üniversitesi, Matematik B\"ol\"um\"u 34342 Bebek,  Istanbul, Turkey}
\address{Müge Taşkın: Boğaziçi Üniversitesi, Matematik B\"ol\"um\"u 34342 Bebek,  Istanbul, Turkey }
        \thanks{Both of the authors are supported by T\"ubitak/1001/115F156).}

\begin{document}

\begin{abstract} We introduce an algorithm to describe Pieri's Rule for multiplication of Schubert polynomials. The algorithm uses tower diagrams introduced
by the authors and another new algorithm that describes Monk's Rule. Our result is different from the well-known descriptions (and proofs) of the rule by Bergeron-Billey
and Kogan-Kumar and uses Sottile's version of Pieri's Rule.

\end{abstract}
\maketitle
\section{Introduction}\label{sec:intro}

The origin of Schubert polynomials lies in the study  of the cohomology of flag manifolds by  Bernstein-Gelfand-Gelfand \cite{BGG} and Demazure \cite{D}. After the appearance of their  explicit description in the work  by Lascoux and Sch\"{u}tzenberger \cite{LS}, they become of great interest in combinatorics.  The works of Macdonald \cite{Ma}, Billey et al. \cite{BJS} and Fomin and Stanley \cite{FS} expose the rich combinatorial aspects of these polynomials. See also \cite{FGR}.

The basic problem regarding Schubert polynomials is to find a combinatorial description of the Littlewood-Richardson
coeffcients. Given a permutation $\omega$, we denote the corresponding Schubert polynomial by $\mathfrak S_\omega$. If
$\nu$ is another permutation, the coefficients $c_{\omega, \nu}^\mu$ in the expansion
\[
\mathfrak S_\omega \cdot \mathfrak S_\nu = \sum_{\mu} c_{\omega, \nu}^\mu \mathfrak S_\mu
\]
are called Littlewood-Richardson coefficients and known to be non-negative.

A general combinatorial description of these coefficients are not known but several special cases are known. The most basic
case is where one of the permutations is an adjacent transposition $s_k$ for some $k$. The Monk's rule states that the
Littlewood-Richardson coefficient $c_{\omega, s_k}^\mu$ is at most 1 and is equal to 1 exactly if $\mu$ is a cover of $\omega$
in $k$-Bruhat order. In \cite{NS}, Bergeron and Billey introduced an insertion algorithm to RC-graphs and give a new and
combinatorial proof of Monk's rule. In the same paper, they have two conjectural versions of Pieri's rule which describes two
other basic cases where one of the permutations is either a row permutation or a column permutation. See Section
\ref{sec:pieri} for definitions. Both of these conjectures are proved by Sottile \cite{F} in a more general way. More precisely,
Sottile proved that in these cases, the coefficients are again at most one and a coefficient $c_{\omega, r}^\mu$ non-zero if and
only if there is a saturated chain in the $k$-Bruhat order satisfying certain conditions. After his proof,  Kogan and 
Kumar \cite{KK}   give new combinatorial proofs of Pieri's rule and  later Kogan \cite{K} prove some other special cases.
There is also an insertion algorithm by Lenart \cite{len} for Pieri rule.

In this paper, we use tower diagrams to introduce two algorithms, Monk's and Pieri's algorithms. Given a permutation $\omega$
and a natural number $k$, Monk's algorithm produces the set of permutations that appear in the expansion of the product
$\mathfrak S_\omega\cdot \mathfrak S_k$ whereas Pieri's algorithm produces a similar set for Pieri's rule by determining a new
way of constructing the above mentioned $k$-Bruhat chains. Although we do not obtain a new special case in this paper, our techniques can be applied to more general cases and we are planning to cover these cases in an upcoming paper.

To summarize the results in this paper, recall that, in \cite{CT1} and \cite{CT2}, we introduced tower diagrams as a new approach to study reduced words of
permutations and Schubert polynomials. In \cite{CT1}, we have shown that a tower diagram can be attached
to any finite permutation and conversely, any tower diagram determines a unique finite permutation. Both of
these correspondences are given by explicit algorithms, called sliding and flight. We have also shown that
the tower diagram of a permutation regarded as a weak composition is the Lehmer code of the inverse
permutation. On the other hand, in \cite{CT2}, we have shown that it is possible to describe the Schubert polynomial of a
permutation using certain types of labellings of the corresponding tower diagram.

In this paper, we first improve our sliding and flight algorithms, and as an application of these new versions, we obtain
descriptions of the well-known Monk's and Pieri's rule for the products of Schubert polynomials. To be more precise, our first
version of the sliding algorithm only works with reduced expressions of permutations whereas the new version produces
the tower diagram of a given permutation starting from an arbitrary expression for the permutation. See Section \ref{sec:gen}
for details. Also, the reverse of the sliding, the flight algorithm, is introduced in \cite{CT1} as an algorithm that produces reduced
words from a given tower diagram. The new version, given in Section \ref{sec:flight} produces the associated permutation as a product
of (not necessarily adjacent) transpositions.

The new algorithms enables us to determine the tower diagram of the product of a permutation $\omega$ with a transposition
$t$ as a modification of tower diagram of $\omega$. The main observation is that when the length of the product $\omega t$ is
one more than that of $\omega$, then only one or only two towers are modified, and moreover, these towers can be
determined using the inverse permutation. In Section \ref{sec:slidetrans}, we explain basic steps of the general modification
process.

Our main results are contained in Section \ref{sec:monks} and Section \ref{sec:pieri} where we introduced Monk's and Pieri's
algorithms, mentioned above. We again note that our description of Monk's Rule is not a new proof of it, however, it is different
from the one given in Bergeron and Billey's milestone \cite{NS}. Indeed, in \cite{NS}, an insertion algorithm for RC-graphs is
used to determine the above set of permutations and hence it works out the Schubert polynomials that appear in the product
monomial-by-monomial. On the other hand, our algorithm directly determines the permutations and do not refer to monomials
forming its Schubert polynomial. For  Pieri's rule, note that Sottile's Theorem asserts that the product of a Schubert polynomial
with the Schubert polynomial of a row permutation is determined by the existence of certain chains in $k$-Bruhat order  and
Bergeron and Billey's conjecture gives an example of such a chain.  In Section \ref{sec:pieri}, we show how to obtain all such
chains and introduce an algorithm that chooses one for each permutation. Generically, these choices are different from the
ones given by Bergeron and Billey \cite{NS} and also from the ones given by Kogan and Kumar in \cite{KK}.

\section{From permutations to tower diagrams: Generalized sliding algorithm}\label{sec:gen}

We start by   recollecting necessary notation from
\cite{CT1} and \cite{CT2}. To begin with, a sequence $\mathcal
T=(\mathcal T_1,\mathcal T_2,\ldots)$ of non-negative integers such
that $\mathcal T_i=0$ for all $i> n$ for some $n$ is called a
\textbf{\textit{tower diagram}}. If $n$ is the largest integer for
which $\mathcal T_n\neq 0$, we write $\mathcal T=(\mathcal
T_1,\mathcal T_2,\ldots,\mathcal T_n)$ and call $\mathcal T_i$ the
\textbf{\textit{$i$-th tower}} of $\mathcal T$. We
identify each tower $\mathcal T_i$ with a vertical strip of height
$\mathcal T_i$ placed in the first quadrant of the plane over the
interval $[i-1,i]$. Therefore if $\mathcal T =(1,0,4,2,0,0,2)$, then
we can represent $\mathcal T$ as the following diagram.

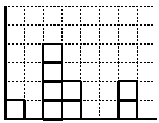
\begin{center}
\setlength{\unitlength}{0.25mm}
\begin{picture}(80,60)
\multiput(0,0)(0,0){1}{\line(1,0){80}}
\multiput(0,0)(0,0){1}{\line(0,1){60}}
\multiput(0,10)(2,0){40}{\line(0,1){.1}}
\multiput(0,20)(2,0){40}{\line(0,1){.1}}
\multiput(0,30)(2,0){40}{\line(0,1){.1}}
\multiput(0,40)(2,0){40}{\line(0,1){.1}}
\multiput(0,50)(2,0){40}{\line(0,1){.1}}
\multiput(0,60)(2,0){40}{\line(0,1){.1}}
\multiput(10,0)(0,2){30}{\line(1,0){.1}}
\multiput(20,0)(0,2){30}{\line(1,0){.1}}
\multiput(30,0)(0,2){30}{\line(1,0){.1}}
\multiput(40,0)(0,2){30}{\line(1,0){.1}}
\multiput(50,0)(0,2){30}{\line(1,0){.1}}
\multiput(60,0)(0,2){30}{\line(1,0){.1}}
\multiput(70,0)(0,2){30}{\line(1,0){.1}}
  \put(0,0){\tableau{{}}}
\put(20,30){\tableau{{}\\{}\\{}\\{}}} \put(30,10){\tableau{{}\\{}}}
\put(60,10){\tableau{{}\\{}}}
\end{picture}
\end{center}
Alternatively, a given tower diagram $\mathcal{T}$ can be represented by
the set consisting of the south-east corners of the cells in the
above representation. For the rest of the paper, we identify any cell with
its south-east corner. Any tower diagram $\mathcal T$ can be
filtered by the sequence $\mathcal T_{\ge 0}, \mathcal
T_{\ge 1}, \ldots$ where $\mathcal T_{\ge j}$ is the tower diagram
obtained from $\mathcal T$ by replacing all towers $\mathcal T_i$ with the
towers of height zero, for each $i< j$. Similarly, for $i\leq j$, we
denote by $\mathcal T_{[i,j]}$ the diagram obtained by replacing all
towers in $\mathcal T$ with index less than $i$ and greater than $j$
by towers of height zero.

Given two tower diagrams $\mathcal T$ and $\mathcal U$ such that there is an index $j$ with $\mathcal T_i = 0$ for
$i\ge j$ and $\mathcal U_k = 0$ for $k< j$. We define $\mathcal T \sqcup \mathcal U$ as the tower diagram with towers
$\mathcal T_1, \ldots, \mathcal T_{j-1}, \mathcal U_j, \mathcal U_{j+1}, \ldots$. With this notation, we always have
$\mathcal T = \mathcal T_{[0,j-1]} \sqcup \mathcal T_{\ge j}$.

Recall that the sliding algorithm in \cite{CT1} determines the rules
for sliding an adjacent transposition into a tower
diagram so that when applied on a sequence of such transpositions, only the
reduced expressions of permutations produce tower diagrams. Moreover this
property provides a bijection between the set of all permutations
and the set of all tower diagrams. Here we first provide a
generalization of this algorithm so that the sliding of any finite sequence of natural numbers
produces a tower diagram. In the next section, we shall show that the results of the new algorithm are compatible with the
results of the previous version.

Let $\mathcal T =(\mathcal T_1,\mathcal T_2, \ldots, \mathcal T_n)$ and let $c$ be a cell in $\mathcal T$. Write
$x+y = i$ for the line passing through the main diagonal of $c$. We call $i$ the \textit{slide} of $c$ and denote it
by $\slide(c)$. We also call the line $x+y=\slide(c)$ the \emph{slide} of $c$. If $e$ is another cell with $\slide(e)=j$, then
we define the \textit{slide distance} between $c$ and $e$ as the difference $i-j$ and denote it by $\sd(c,e) = i-j $.

To slide a natural number $i$ to $\mathcal T$, we first place a new cell $c$  whose east border is the interval
$[i,i+1]$ on the $y$-axis so that $\slide(c)=i$. Then, we let $c$  slide through its slide subject to the following conditions.
Let $\mathcal T_k$ be the first non-empty tower in $\mathcal T$ and $d$ be the top cell of $\mathcal T_k$.

\begin{enumerate}
\item[\bf{1.}] \textbf{Direct Pass:} If $\sd(c,d) \ge 2$, the sliding of $c$ continues on
$\mathcal T_{> k}$ subject to the conditions in this list.
\item[\bf{2.}] \textbf{Addition:} If $\sd(c,d) = 1$, then $c$ is placed on the top of $d$ and sliding of $c$ stops.
\item[\bf{3.}] \textbf{Deletion:} If $\sd(c,d) = 0$, then  sliding of $c$ stops by deleting   $d$.
\item[\bf{4.}] \textbf{Zigzag pass:} If $\sd(c,d)<0$, then we move the cell $c$ one level up so that its
new slide is $i+1$ and let it continue its sliding through $\mathcal T_{>k}$ on its new slide, subject to the conditions in this list.
\end{enumerate}

In the following Figure we illustrate the possible cases. Note that  in the third case the deleted cell is labeled by
$\varnothing$ and in the other cases the cells labeled by $\bullet$ represent the new positions of $c$.

\begin{figure}[h]
\begin{center}
\setlength{\unitlength}{0.25mm}
\begin{picture}(80,60)
\multiput(0,0)(0,0){1}{\line(1,0){60}}
\multiput(0,0)(0,0){1}{\line(0,1){60}}
\multiput(0,10)(2,0){30}{\line(0,1){.1}}
\multiput(0,20)(2,0){30}{\line(0,1){.1}}
\multiput(0,30)(2,0){30}{\line(0,1){.1}}
\multiput(0,40)(2,0){30}{\line(0,1){.1}}
\multiput(0,50)(2,0){30}{\line(0,1){.1}}
\multiput(0,60)(2,0){30}{\line(0,1){.1}}
\multiput(10,0)(0,2){30}{\line(1,0){.1}}
\multiput(20,0)(0,2){30}{\line(1,0){.1}}
\multiput(30,0)(0,2){30}{\line(1,0){.1}}
\multiput(40,0)(0,2){30}{\line(1,0){.1}}
\multiput(50,0)(0,2){30}{\line(1,0){.1}}
\multiput(60,0)(0,2){30}{\line(1,0){.1}}
\put(10,10){\tableau{{d}\\{}}} \put(20,0){\tableau{{}}}
 \put(-10,50){\tableau{{c}}}
\multiput(23,22)(0,0){1}{$\bullet$}
\multiput(0,50)(0,0){1}{\vector(1,-1){25}}
\multiput(5,-10)(0,0){1}{ }
\multiput(15,-10)(0,0){1}{}
\end{picture}
\hskip 0.05in
\begin{picture}(80,60)
\multiput(0,0)(0,0){1}{\line(1,0){60}}
\multiput(0,0)(0,0){1}{\line(0,1){60}}
\multiput(0,10)(2,0){30}{\line(0,1){.1}}
\multiput(0,20)(2,0){30}{\line(0,1){.1}}
\multiput(0,30)(2,0){30}{\line(0,1){.1}}
\multiput(0,40)(2,0){30}{\line(0,1){.1}}
\multiput(0,50)(2,0){30}{\line(0,1){.1}}
\multiput(0,60)(2,0){30}{\line(0,1){.1}}
\multiput(10,0)(0,2){30}{\line(1,0){.1}}
\multiput(20,0)(0,2){30}{\line(1,0){.1}}
\multiput(30,0)(0,2){30}{\line(1,0){.1}}
\multiput(40,0)(0,2){30}{\line(1,0){.1}}
\multiput(50,0)(0,2){30}{\line(1,0){.1}}
\multiput(60,0)(0,2){30}{\line(1,0){.1}}
\put(10,20){\tableau{{d}\\{}\\{}}} \put(20,0){\tableau{{}}}
\put(-10,50){\tableau{{c}}} \multiput(13,32)(0,0){1}{$\bullet$}
\multiput(0,50)(0,0){1}{\vector(1,-1){15}}
\multiput(5,-10)(0,0){1}{}
\multiput(15,-10)(0,0){1}{}
\end{picture}\hskip 0.05in
\hskip 0.05in
\begin{picture}(80,60)
\multiput(0,0)(0,0){1}{\line(1,0){60}}
\multiput(0,0)(0,0){1}{\line(0,1){60}}
\multiput(0,10)(2,0){30}{\line(0,1){.1}}
\multiput(0,20)(2,0){30}{\line(0,1){.1}}
\multiput(0,30)(2,0){30}{\line(0,1){.1}}
\multiput(0,40)(2,0){30}{\line(0,1){.1}}
\multiput(0,50)(2,0){30}{\line(0,1){.1}}
\multiput(0,60)(2,0){30}{\line(0,1){.1}}
\multiput(10,0)(0,2){30}{\line(1,0){.1}}
\multiput(20,0)(0,2){30}{\line(1,0){.1}}
\multiput(30,0)(0,2){30}{\line(1,0){.1}}
\multiput(40,0)(0,2){30}{\line(1,0){.1}}
\multiput(50,0)(0,2){30}{\line(1,0){.1}}
\multiput(60,0)(0,2){30}{\line(1,0){.1}}
\multiput(35,55)(0,0){1}{\vector(-1,-1){15}}
\put(35,55){$d$}
\put(10,30){\tableau{{\varnothing}\\{}\\{}\\{}}} \put(20,0){\tableau{{}}}
\put(-10,50){\tableau{{c}}}
\multiput(0,50)(0,0){1}{\vector(1,-1){15}}
\multiput(5,-10)(0,0){1}{}
\multiput(15,-10)(0,0){1}{}
\end{picture}\hskip 0.05in
\begin{picture}(80,60)
\multiput(0,0)(0,0){1}{\line(1,0){60}}
\multiput(0,0)(0,0){1}{\line(0,1){60}}
\multiput(0,10)(2,0){30}{\line(0,1){.1}}
\multiput(0,20)(2,0){30}{\line(0,1){.1}}
\multiput(0,30)(2,0){30}{\line(0,1){.1}}
\multiput(0,40)(2,0){30}{\line(0,1){.1}}
\multiput(0,50)(2,0){30}{\line(0,1){.1}}
\multiput(0,60)(2,0){30}{\line(0,1){.1}}
\multiput(10,0)(0,2){30}{\line(1,0){.1}}
\multiput(20,0)(0,2){30}{\line(1,0){.1}}
\multiput(30,0)(0,2){30}{\line(1,0){.1}}
\multiput(40,0)(0,2){30}{\line(1,0){.1}}
\multiput(50,0)(0,2){30}{\line(1,0){.1}}
\multiput(60,0)(0,2){30}{\line(1,0){.1}}
\put(10,40){\tableau{{d}\\{}\\{}\\{}\\{}}}
\put(20,0){\tableau{{}}}\put(30,10){\tableau{{}\\{}}}
\put(-10,50){\tableau{{c}}} \multiput(5,-10)(0,0){1}{}
\multiput(15,-10)(0,0){1}{}
\multiput(0,50)(0,0){1}{\line(1,-1){15}}
\multiput(15,45)(0,0){1}{\line(0,-1){10}}
\multiput(15,45)(0,0){1}{\vector(1,-1){15}}
\multiput(24,31)(0,0){1}{$\bullet$}
\end{picture}
\end{center}
\caption{Direct pass, Addition, Deletion and Zigzag pass of a single cell. }\label{fig:cellsleding}
\end{figure}
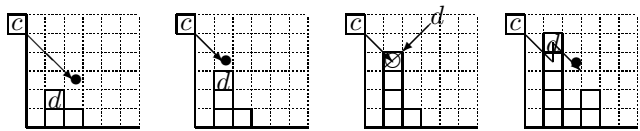

We remark that the only difference between the generalized sliding and the one in \cite{CT1} appears in Condition 3.
 In case $\sd(c,d) = 0$, the algorithm \emph{terminates without a result }(that is, no tower diagram is produced) in the
previous version, whereas in the new one, a tower diagram is produced by removing $d$.  Since the new algorithm is
an extension of the previous one, we still call it the sliding algorithm. We shall see that with this improvement, the new
algorithm has several new applications and allows us to construct new operations.

We denote the tower diagram obtained as the result of sliding the number $i$ (resp. the cell $c$) to $\mathcal T$ by
$i\searrow \mathcal T$ (resp. $c\searrow \mathcal T$). In fact any finite word on  natural numbers $\alpha=\alpha_1\ldots \alpha_k$ can be slid to $\mathcal T$ by the following rule:
$$\alpha \searrow \mathcal T:= \alpha_2\ldots \alpha_k \searrow (\alpha_1 \searrow \mathcal T)=  \alpha_k \searrow \ldots \alpha_2 \searrow \alpha_1 \searrow \mathcal T.
$$
If $\mathcal T$ is the empty diagram then we denote the resulting diagram by $\Tt_\alpha.$

For example, sliding of  $\alpha = 43413$ in to the empty diagram produces $\Tt_{\alpha}$ through the following  sequence of diagrams  displayed in Figure \ref{fig:wordsliding}.
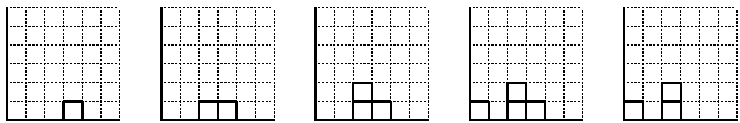
\begin{figure}[h]
\setlength{\unitlength}{0.25mm}
\begin{center}
\begin{picture}(80,60)
\multiput(0,0)(0,0){1}{\line(1,0){60}}
\multiput(0,0)(0,0){1}{\line(0,1){60}}
\multiput(0,10)(2,0){30}{\line(0,1){.1}}
\multiput(0,20)(2,0){30}{\line(0,1){.1}}
\multiput(0,30)(2,0){30}{\line(0,1){.1}}
\multiput(0,40)(2,0){30}{\line(0,1){.1}}
\multiput(0,50)(2,0){30}{\line(0,1){.1}}
\multiput(0,60)(2,0){30}{\line(0,1){.1}}
\multiput(10,0)(0,2){30}{\line(1,0){.1}}
\multiput(20,0)(0,2){30}{\line(1,0){.1}}
\multiput(30,0)(0,2){30}{\line(1,0){.1}}
\multiput(40,0)(0,2){30}{\line(1,0){.1}}
\multiput(50,0)(0,2){30}{\line(1,0){.1}}
\multiput(60,0)(0,2){30}{\line(1,0){.1}} \put(30,0){\tableau{{}}}
\end{picture}\hskip 0.02in
\begin{picture}(80,60)
\multiput(0,0)(0,0){1}{\line(1,0){60}}
\multiput(0,0)(0,0){1}{\line(0,1){60}}
\multiput(0,10)(2,0){30}{\line(0,1){.1}}
\multiput(0,20)(2,0){30}{\line(0,1){.1}}
\multiput(0,30)(2,0){30}{\line(0,1){.1}}
\multiput(0,40)(2,0){30}{\line(0,1){.1}}
\multiput(0,50)(2,0){30}{\line(0,1){.1}}
\multiput(0,60)(2,0){30}{\line(0,1){.1}}
\multiput(10,0)(0,2){30}{\line(1,0){.1}}
\multiput(20,0)(0,2){30}{\line(1,0){.1}}
\multiput(30,0)(0,2){30}{\line(1,0){.1}}
\multiput(40,0)(0,2){30}{\line(1,0){.1}}
\multiput(50,0)(0,2){30}{\line(1,0){.1}}
\multiput(60,0)(0,2){30}{\line(1,0){.1}} \put(30,0){\tableau{{}}}
\put(20,0){\tableau{{}}}
\end{picture}\hskip 0.02in
\begin{picture}(80,60)
\multiput(0,0)(0,0){1}{\line(1,0){60}}
\multiput(0,0)(0,0){1}{\line(0,1){60}}
\multiput(0,10)(2,0){30}{\line(0,1){.1}}
\multiput(0,20)(2,0){30}{\line(0,1){.1}}
\multiput(0,30)(2,0){30}{\line(0,1){.1}}
\multiput(0,40)(2,0){30}{\line(0,1){.1}}
\multiput(0,50)(2,0){30}{\line(0,1){.1}}
\multiput(0,60)(2,0){30}{\line(0,1){.1}}
\multiput(10,0)(0,2){30}{\line(1,0){.1}}
\multiput(20,0)(0,2){30}{\line(1,0){.1}}
\multiput(30,0)(0,2){30}{\line(1,0){.1}}
\multiput(40,0)(0,2){30}{\line(1,0){.1}}
\multiput(50,0)(0,2){30}{\line(1,0){.1}}
\multiput(60,0)(0,2){30}{\line(1,0){.1}}
 \put(30,0){\tableau{{}}}\put(20,0){\tableau{{}}} \put(20,10){\tableau{{}}}
\end{picture}\hskip 0.02in
\begin{picture}(80,60)
\multiput(0,0)(0,0){1}{\line(1,0){60}}
\multiput(0,0)(0,0){1}{\line(0,1){60}}
\multiput(0,10)(2,0){30}{\line(0,1){.1}}
\multiput(0,20)(2,0){30}{\line(0,1){.1}}
\multiput(0,30)(2,0){30}{\line(0,1){.1}}
\multiput(0,40)(2,0){30}{\line(0,1){.1}}
\multiput(0,50)(2,0){30}{\line(0,1){.1}}
\multiput(0,60)(2,0){30}{\line(0,1){.1}}
\multiput(10,0)(0,2){30}{\line(1,0){.1}}
\multiput(20,0)(0,2){30}{\line(1,0){.1}}
\multiput(30,0)(0,2){30}{\line(1,0){.1}}
\multiput(40,0)(0,2){30}{\line(1,0){.1}}
\multiput(50,0)(0,2){30}{\line(1,0){.1}}
\multiput(60,0)(0,2){30}{\line(1,0){.1}}  \put(30,0){\tableau{{}}}\put(20,0){\tableau{{}}} \put(20,10){\tableau{{}}}
\put(0,0){\tableau{{}}}
\end{picture}\hskip 0.02in
\begin{picture}(80,60)
\multiput(0,0)(0,0){1}{\line(1,0){60}}
\multiput(0,0)(0,0){1}{\line(0,1){60}}
\multiput(0,10)(2,0){30}{\line(0,1){.1}}
\multiput(0,20)(2,0){30}{\line(0,1){.1}}
\multiput(0,30)(2,0){30}{\line(0,1){.1}}
\multiput(0,40)(2,0){30}{\line(0,1){.1}}
\multiput(0,50)(2,0){30}{\line(0,1){.1}}
\multiput(0,60)(2,0){30}{\line(0,1){.1}}
\multiput(10,0)(0,2){30}{\line(1,0){.1}}
\multiput(20,0)(0,2){30}{\line(1,0){.1}}
\multiput(30,0)(0,2){30}{\line(1,0){.1}}
\multiput(40,0)(0,2){30}{\line(1,0){.1}}
\multiput(50,0)(0,2){30}{\line(1,0){.1}}
\multiput(60,0)(0,2){30}{\line(1,0){.1}} \put(20,0){\tableau{{}}}
\put(0,0){\tableau{{}}}\put(20,10){\tableau{{}}}
\end{picture}
\caption{Sliding and the tower diagram of the word $\alpha = 43413$. }\label{fig:wordsliding}
\end{center}
\end{figure}

As an immediate application, given tower
diagrams $\mathcal T $ and $\mathcal U= (\mathcal U_1,\mathcal U_2,
\ldots,\mathcal U_m)$, it is now possible to define the product $\mathcal U\searrow\mathcal T$ as  follows. First move
the diagram $\mathcal U$ in the plane so that the right most bottom cell has its east border on  the interval
$[m,m+1]$ on the $y$-axis and has its slide  $m$. Then we slide the cells in $\mathcal U$ to $\mathcal T$, starting from
this cell and continue from bottom to top and right to left.

In the following we explain how finite permutations can be associated to tower diagrams.  A permutation $\omega$ can
be written as a product $\omega=s_{\alpha_1}\ldots s_{\alpha_l}$ of adjacent transpositions $s_{\alpha_1},\ldots,
s_{\alpha_l}$. There are infinitely many different such expressions and  among these, the ones having the
minimum length $l(\omega)$ are called reduced expressions of $\omega$. We say that the word
$\alpha=\alpha_1\ldots\alpha_l$ is associated to $\omega$ if
$$s_{[\alpha]}:=s_{\alpha_1}\ldots s_{\alpha_l}=\omega.$$
If $s_{[\alpha]}$ is a reduced expression, then the word $\alpha$ is called a reduced word.

With respect to the sliding algorithm defined in  \cite{CT1}, we show that the sliding of $\alpha$ terminates with a  result
if and only if $s_{[\alpha]}$ is a reduced word of a permutation. We also showed that for two words $\alpha$ and
$\beta$, we have $s_{[\alpha]}$ and $s_{[\beta]}$ are reduced expressions of the same permutation if and only if
$\Tt_{\alpha}=\Tt_{\beta}$. This result enables us to define the \emph{tower diagram} of a permutation $\omega$ by
 $$ \Tt_{\omega}:=  \Tt_{\alpha}
 $$
when $s_{[\alpha]}$ is a reduced expression of $\omega$.

With the new version of the sliding algorithm, the first result above  can now be restated as:  a deletion never occurs in
the  sliding of $\alpha$ into the empty tower diagram if and only if $s_{[\alpha]}$ is a reduced expression of a
 permutation. Extending the second result above we have the following theorem.

\begin{thm}\label{thm:wordshape}
Let $\alpha$ and $\beta$ be two words on natural numbers. Then
$\Tt_{\alpha}= \Tt_{\beta} $ if and only if $s_{[\alpha]}
= s_{[\beta]}$.
\end{thm}

\begin{proof}
The theorem is true if both of the words $\alpha$ and $\beta$ are
reduced. Thus to prove the
theorem, it is sufficient to prove that for any word $\alpha$ satisfying $s_{[\alpha]}$,  $\Tt_{\alpha} = \Tt_{\omega}$.

Write $\alpha =\alpha_1\alpha_2\ldots\alpha_n$.
Without loss of generality, assume that $n$ is the first index that a deletion occurs in the sliding of $\alpha$, hence the length of $\omega$ is $n-2$.  (Otherwise we change $\alpha$
with the sub word where the first termination occurs, and hence
change $\omega$ with the corresponding permutation.) Thus
the word $\alpha_1\alpha_2\ldots\alpha_{n-1}$ is a reduced word of another permutation, say $\omega'$ and    there is another reduced word of $\omega'$ of the form $$\omega'=s_{[\gamma_1\gamma_2\ldots\gamma_{n-2}\alpha_n]}=s_{[\alpha_1\alpha_2\ldots\alpha_{n-1}]}.$$
Hence $\gamma_1\gamma_2\ldots\gamma_{n-2}\alpha_n$ and $\alpha_1\alpha_2\ldots\alpha_{n-1}$ are braid related and
$\Tt_{\gamma_1\gamma_2\ldots\gamma_{n-2}\alpha_n}=\Tt_{\alpha_1\alpha_2\ldots\alpha_{n-1}}$ since both words
are reduced. Now sliding   one more   $\alpha_n$ gives the same diagrams that is
$$\Tt_{\alpha}=\Tt_{\alpha_1\alpha_2\ldots\alpha_{n-1}\alpha_n}=\Tt_{\gamma_1\gamma_2\ldots\gamma_{n-2}\alpha_n\alpha_n}=\Tt_{\gamma_1\gamma_2\ldots\gamma_{n-2}}.$$

Also note that  $\gamma_1\gamma_2\ldots\gamma_{n-2}$ is a reduced word of $\omega$, since
$\alpha=\alpha_1\alpha_2\ldots\alpha_{n-1}\alpha_n$ and    $\gamma_1\gamma_2\ldots\gamma_{n-2}=\gamma_1\gamma_2\ldots\gamma_{n-2}\alpha_n\alpha_n$ are also braid related.  Therefore
$$\Tt_{\alpha}=\Tt_{\gamma_1\gamma_2\ldots\gamma_{n-2}}=\Tt_{\omega}.
$$

\end{proof}

\begin{cor} For two permutations $\omega$ and $\tau$, $\Tt_\omega \searrow \Tt_\tau = \Tt_{\tau\cdot\omega }$.
\end{cor}

\section{From tower diagrams  to permutations: generalized flight algorithm}\label{sec:flight}

As seen above, the sliding algorithm constructs the tower diagram of the given permutation. We now explain how we obtain the corresponding  permutation from the given tower diagram. We achieve this through the flight algorithm given below in a generalized version.

\begin{defn}
Let $\mathcal T$ be a tower diagram, $c$ be a cell not necessarily contained in $\mathcal T$. We define the flight path,  $\fp(\Tt,c)$,  of $c$ in $\Tt$ recursively as follows: Let $d$ be the west and $e$ be the northwest neighbors of $c$.
$$   \fp(\Tt,c)=\left\{ \begin{aligned}  \fp(\Tt,d)\cup \{c\}, & \,\,  \mbox{  if } d\in \Tt\\
                                             \fp (\Tt,e)\cup \{c\}, & \,\,  \mbox{  if } d\not \in \Tt
                                         \end{aligned} \right.$$
\end{defn}

\begin{ex} Consider the cells $c$, $c^\prime$ and $c^{\prime\prime}$ in Figure \ref{fig:fn}. Their  flight paths are shown by bullets   and circles so that the ones labeled by a bullet always lie in $\Tt$ and the ones labeled by a circle always lie outside of $\Tt$.

\begin{figure}[h]\setlength{\unitlength}{0.25mm}
\begin{center}
\begin{picture}(80,60)
\multiput(0,0)(0,0){1}{\line(1,0){80}}
\multiput(0,0)(0,0){1}{\line(0,1){60}}
\multiput(0,10)(2,0){40}{\line(0,1){.1}}
\multiput(0,20)(2,0){40}{\line(0,1){.1}}
\multiput(0,30)(2,0){40}{\line(0,1){.1}}
\multiput(0,40)(2,0){40}{\line(0,1){.1}}
\multiput(0,50)(2,0){40}{\line(0,1){.1}}
\multiput(0,60)(2,0){40}{\line(0,1){.1}}
\multiput(10,0)(0,2){30}{\line(1,0){.1}}
\multiput(20,0)(0,2){30}{\line(1,0){.1}}
\multiput(30,0)(0,2){30}{\line(1,0){.1}}
\multiput(40,0)(0,2){30}{\line(1,0){.1}}
\multiput(50,0)(0,2){30}{\line(1,0){.1}}
\multiput(60,0)(0,2){30}{\line(1,0){.1}}
\multiput(70,0)(0,2){30}{\line(1,0){.1}}
\multiput(80,0)(0,2){30}{\line(1,0){.1}} \put(3,53){$\circ$}
 \put(10,40){\tableau{{\bullet}\\{}\\{}\\{}\\{}}}
\put(20,40){\tableau{{\bullet}\\{}\\{}\\{}\\{}}}
\put(30,0){\tableau{{}}} \put(40,0){\tableau{{}}}
\put(53,23){$\circ$} \put(43,33){$\circ$}\put(33,43){$\circ$}
 \put(60,30){\tableau{{}\\{}\\{\bullet}\\{}}}
\put(70,10){\tableau{{c}\\{}}}
\end{picture}\hskip.2in
\begin{picture}(80,60)
\multiput(0,0)(0,0){1}{\line(1,0){80}}
\multiput(0,0)(0,0){1}{\line(0,1){60}}
\multiput(0,10)(2,0){40}{\line(0,1){.1}}
\multiput(0,20)(2,0){40}{\line(0,1){.1}}
\multiput(0,30)(2,0){40}{\line(0,1){.1}}
\multiput(0,40)(2,0){40}{\line(0,1){.1}}
\multiput(0,50)(2,0){40}{\line(0,1){.1}}
\multiput(0,60)(2,0){40}{\line(0,1){.1}}
\multiput(10,0)(0,2){30}{\line(1,0){.1}}
\multiput(20,0)(0,2){30}{\line(1,0){.1}}
\multiput(30,0)(0,2){30}{\line(1,0){.1}}
\multiput(40,0)(0,2){30}{\line(1,0){.1}}
\multiput(50,0)(0,2){30}{\line(1,0){.1}}
\multiput(60,0)(0,2){30}{\line(1,0){.1}}
\multiput(70,0)(0,2){30}{\line(1,0){.1}}
\multiput(80,0)(0,2){30}{\line(1,0){.1}} \put(3,53){}
 \put(10,40){\tableau{{}\\{}\\{}\\{\bullet}\\{}}}
\put(20,40){\tableau{{}\\{}\\{}\\{\bullet}\\{}}}
\put(30,0){\tableau{{}}} \put(40,0){\tableau{{}}}
\put(53,23){} \put(43,33){}\put(2,23){$\circ$} \put(32,11){$c\prime$}
 \put(60,30){\tableau{{}\\{}\\{}\\{}}}
\put(70,10){\tableau{{}\\{}}}
\end{picture}
\hskip.2in
\begin{picture}(80,60)
\multiput(0,0)(0,0){1}{\line(1,0){80}}
\multiput(0,0)(0,0){1}{\line(0,1){60}}
\multiput(0,10)(2,0){40}{\line(0,1){.1}}
\multiput(0,20)(2,0){40}{\line(0,1){.1}}
\multiput(0,30)(2,0){40}{\line(0,1){.1}}
\multiput(0,40)(2,0){40}{\line(0,1){.1}}
\multiput(0,50)(2,0){40}{\line(0,1){.1}}
\multiput(0,60)(2,0){40}{\line(0,1){.1}}
\multiput(10,0)(0,2){30}{\line(1,0){.1}}
\multiput(20,0)(0,2){30}{\line(1,0){.1}}
\multiput(30,0)(0,2){30}{\line(1,0){.1}}
\multiput(40,0)(0,2){30}{\line(1,0){.1}}
\multiput(50,0)(0,2){30}{\line(1,0){.1}}
\multiput(60,0)(0,2){30}{\line(1,0){.1}}
\multiput(70,0)(0,2){30}{\line(1,0){.1}}
\multiput(80,0)(0,2){30}{\line(1,0){.1}} \put(3,53){}
 \put(10,40){\tableau{{}\\{}\\{}\\{}\\{\bullet}}}
\put(20,40){\tableau{{}\\{}\\{}\\{}\\{\bullet}}}
\put(30,0){\tableau{c^{\prime \prime}}} \put(40,0){\tableau{{}}}
\put(53,23){} \put(43,33){}\put(2,13){$\circ$}
 \put(60,30){\tableau{{}\\{}\\{}\\{}}}
\put(70,10){\tableau{{}\\{}}}
\end{picture}
\end{center}
\caption{}\label{fig:fn}
\end{figure}
\end{ex}

In the following, we explain how to assign a \emph{ generalized flight number} to any top cell of $\Tt$.
Let $c$ be a top cell and   $c_1,\ldots, c_k=c$ be the cells in the flight path of $c$ ordered from left to right. Now $c_1=(1,j)$    for some  $j\geq 0$ and we define the \emph{flight number} of $c$ to be the slide of $c$, that is
 $$\fn(\Tt,c):=1+j.$$
We now define the \emph{hook number} of $c$ by assigning a sequence  $n_1,\ldots,n_k$  of nonnegative  numbers to   $c_1,\ldots, c_k=c$ in the following manner:  Let $n_k=0$ and for $i=k, \ldots,2$,
$$ n_{i-1}=\left \{ \begin{aligned}
    n_{i}  \,& \,\,\mbox{if $c_{i-1} \not \in\Tt $} \\
    n_{i}  \,& \,\,\mbox{if $c_{i-1} \in\Tt $ and the number of cells above $c_{i-1}$ is greater  than $n_i$.}    \\
  n_i +1 \,& \,\, \mbox{if $c_{i-1} \in\Tt $ and the number of cells above $c_{i-1}$ is less  than or equal to $n_i$.}
   \end{aligned} \right.
   $$
Then the hook number of $c$  is defined to be $\hn(\Tt,c):= n_1+\mathrm{sl}(c_1)$.  Hence the \emph{ generalized flight number} of $c$ in $\Tt$ is defined to be the pair
$$\gfn(c,\Tt):=(\fn(\Tt,c),\hn(\Tt,c)).
$$
Also a top cell in $\Tt$ is called a \emph{ corner cell},  if its hook number and flight number  are equal.

\begin{ex}
For instance in the tower diagrams illustrated in Figure \ref{fig:fn},  $\gfn(\Tt,c)=(6,8)$,  $\gfn(\Tt,c')=(3,3)$ and  $\gfn(\Tt,c'')=(2,2)$, but among them just $c''$ is a corner cell of $\Tt$.
\end{ex}

\begin{rem}
1) Contrary to the one given  in \cite{CT1}, the above definition assigns a flight path to any cell in the first  quadrant, not necessarily
contained in $\mathcal T$.  Also we generalized the notion of flight number to any top cell  (not necessarily a corner cell) by adding also a hook number.  Note that the flight number of the corner cell in the new and the old versions are the same and its hook number is zero.  The reason why we call the second number  as hook number will be clear once we describe hook sliding algorithm.

2) In \cite{CT1}, we have shown how to obtain the Rothe diagram of a permutation from its tower diagram and vice
versa. It is easy to see that under this correspondence, the flight number of a cell in a tower diagram is the row index of
the corresponding cell in the corresponding Rothe diagram. In particular, given a tower diagram $\mathcal T$, the flight numbers of cells in or over a tower in $\mathcal T$ strictly increases from bottom to top.

3)  If $c$ is a corner cell of $\Tt$ then every bullet in the flight path of $c$ must have nonempty tops. Moreover, denoting by
 $\Tt-c$ the diagram obtained by erasing $c$ from $\Tt$, we see that
 $$ \fn(\Tt,c) \searrow (\Tt-c) =\Tt
 $$
as illustrated in the following figure.
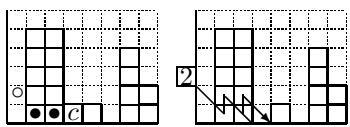
\begin{figure}[h]\setlength{\unitlength}{0.25mm}
\begin{center}
\begin{picture}(80,50)
\multiput(0,0)(0,0){1}{\line(1,0){80}}
\multiput(0,0)(0,0){1}{\line(0,1){60}}
\multiput(0,10)(2,0){40}{\line(0,1){.1}}
\multiput(0,20)(2,0){40}{\line(0,1){.1}}
\multiput(0,30)(2,0){40}{\line(0,1){.1}}
\multiput(0,40)(2,0){40}{\line(0,1){.1}}
\multiput(0,50)(2,0){40}{\line(0,1){.1}}
\multiput(0,60)(2,0){40}{\line(0,1){.1}}
\multiput(10,0)(0,2){30}{\line(1,0){.1}}
\multiput(20,0)(0,2){30}{\line(1,0){.1}}
\multiput(30,0)(0,2){30}{\line(1,0){.1}}
\multiput(40,0)(0,2){30}{\line(1,0){.1}}
\multiput(50,0)(0,2){30}{\line(1,0){.1}}
\multiput(60,0)(0,2){30}{\line(1,0){.1}}
\multiput(70,0)(0,2){30}{\line(1,0){.1}}
\multiput(80,0)(0,2){30}{\line(1,0){.1}} \put(3,53){}
 \put(10,40){\tableau{{}\\{}\\{}\\{}\\{\bullet}}}
\put(20,40){\tableau{{}\\{}\\{}\\{}\\{\bullet}}}
\put(30,0){\tableau{{c}}} \put(40,0){\tableau{{}}}
\put(53,23){} \put(43,33){}\put(2,13){$\circ$}
 \put(60,30){\tableau{{}\\{}\\{}\\{}}}
\put(70,10){\tableau{{}\\{}}}
\end{picture}\hskip.2in
\begin{picture}(80,50)
\multiput(0,0)(0,0){1}{\line(1,0){80}}
\multiput(0,0)(0,0){1}{\line(0,1){60}}
\multiput(0,10)(2,0){40}{\line(0,1){.1}}
\multiput(0,20)(2,0){40}{\line(0,1){.1}}
\multiput(0,30)(2,0){40}{\line(0,1){.1}}
\multiput(0,40)(2,0){40}{\line(0,1){.1}}
\multiput(0,50)(2,0){40}{\line(0,1){.1}}
\multiput(0,60)(2,0){40}{\line(0,1){.1}}
\multiput(10,0)(0,2){30}{\line(1,0){.1}}
\multiput(20,0)(0,2){30}{\line(1,0){.1}}
\multiput(30,0)(0,2){30}{\line(1,0){.1}}
\multiput(40,0)(0,2){30}{\line(1,0){.1}}
\multiput(50,0)(0,2){30}{\line(1,0){.1}}
\multiput(60,0)(0,2){30}{\line(1,0){.1}}
\multiput(70,0)(0,2){30}{\line(1,0){.1}}
\multiput(80,0)(0,2){30}{\line(1,0){.1}} \put(3,53){}
 \put(10,40){\tableau{{}\\{}\\{}\\{}\\{}}}
\put(20,40){\tableau{{}\\{}\\{}\\{}\\{}}}
 \put(40,0){\tableau{{}}}
\put(53,23){} \put(43,33){}
 \put(60,30){\tableau{{}\\{}\\{}\\{}}}
\put(70,10){\tableau{{}\\{}}}
\put(-10,20){\tableau{{2}}}
\multiput(0,20)(0,0){1}{\line(1,-1){15}}
\multiput(15,15)(0,0){1}{\line(1,-1){15}}
\multiput(15,15)(0,0){1}{\line(0,-1){10}}\multiput(25,15)(0,0){1}{\line(0,-1){10}}
\multiput(25,15)(0,0){1}{\vector(1,-1){15}}
\end{picture}
\end{center}
\caption{ Here $c$ is a corner cell in $\Tt$ with  flight number $2$. In fact sliding $2$ into $\Tt-c$ creates $c$ at the end. }
\end{figure}

Hence  starting from a tower diagram $\Tt$ of size $n$ and deleting a corner cell at each time, one can  get a sequence of  tower diagrams
$\Tt=\Tt^l,\ldots, \Tt^0=\varnothing$ of decreasing sizes and a sequence
$\alpha_l,\ldots,\alpha_1$ of positive integers such that $\alpha_i$ is the flight number of the  corner cell chosen  in
$\Tt^i$ and  $\Tt^{i+1}$ is obtained from $\Tt^i$ by erasing that corner cell. Now  we see that the permutation
$$\omega=s_{[\alpha_1,\ldots,\alpha_l]}= s_{\alpha_1}\ldots s_{\alpha_l}
$$
has the tower diagram $\Tt$.
\end{rem}

The above algorithm finds a reduced word of the permutation whose tower diagram is the given one. Below we introduce a new algorithm which determines the one line notation of the permutation directly.

To start with let $n-1$ be the maximum of all the slides of the top cells in $\Tt=(\Tt_1,\ldots,\Tt_n, \ldots, )$.  This means that  all the towers  to the right of
$(n-1)$-th tower  in $\Tt$ are always empty.  Let $x_i$ be the lowest empty cell in the tower $\Tt_i$ for each $i=1,\ldots ,n$  and let
$$f_i=\fn(\Tt, x_i)
$$
that is $f_i$ is the flight number of the empty cell lying on top of $\Tt_i$.   Now we define a function, $\pi_{_\Tt}:\{f_1,\ldots,f_n\} \mapsto \{1,\ldots,n\}$, by the rule
that  $$\pi_{_\Tt}(f_i)=i$$ and we also define $\pi_{_\Tt}$-index of $\Tt$ to be the sequence $(f_1,\ldots,f_n)$. Note that
by the next proposition, this is a well defined permutation in $S_n$.

As an example consider the following tower diagram $\Tt=(0,3,3,1,1,0,1,0)$ with largest slide being $7$.  The flight numbers of the empty cells $x_1,\ldots,
x_8$ in $\Tt$ are respectively   $f_1,\ldots, f_8 =1,6,7,3,4,2,8,5$.
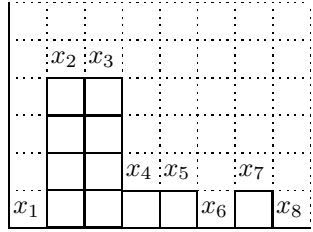
\begin{figure}[h]\setlength{\unitlength}{0.5mm}
\begin{center}
\begin{picture}(80,60)
\multiput(0,0)(0,0){1}{\line(1,0){80}}
\multiput(0,0)(0,0){1}{\line(0,1){60}}
\multiput(0,10)(2,0){40}{\line(0,1){.1}}
\multiput(0,20)(2,0){40}{\line(0,1){.1}}
\multiput(0,30)(2,0){40}{\line(0,1){.1}}
\multiput(0,40)(2,0){40}{\line(0,1){.1}}
\multiput(0,50)(2,0){40}{\line(0,1){.1}}
\multiput(0,60)(2,0){40}{\line(0,1){.1}}
\multiput(10,0)(0,2){30}{\line(1,0){.1}}
\multiput(20,0)(0,2){30}{\line(1,0){.1}}
\multiput(30,0)(0,2){30}{\line(1,0){.1}}
\multiput(40,0)(0,2){30}{\line(1,0){.1}}
\multiput(50,0)(0,2){30}{\line(1,0){.1}}
\multiput(60,0)(0,2){30}{\line(1,0){.1}}
\multiput(70,0)(0,2){30}{\line(1,0){.1}}
\multiput(80,0)(0,2){30}{\line(1,0){.1}} \put(3,53){}
 \put(10,30){\tableaux{{}\\{}\\{}\\{}}}
\put(20,30){\tableaux{{}\\{}\\{}\\{}}}
\put(30,0){\tableaux{{}}} \put(40,0){\tableaux{{}}}
\put(53,23){} \put(43,33){}
\put(60,0){\tableaux{{}}}
\put(1,4){${x_1}$}
\put(11,44){${x_2}$}
\put(21,44){${x_3}$}
\put(31,14){${x_4}$}
\put(41,14){${x_5}$}
\put(51,4){${x_6}$}
\put(61,14){${x_7}$}
\put(71,4){${x_8}$}
\end{picture}
\end{center}
\end{figure}

Hence  $\pi_{_\Tt}= 16458237$ and  $\pi_{_\Tt}$-index of $\Tt$ is $(1,6,7,3,4,2,8,5)$.

\begin{pro} \label{line.notation} Let $\Tt$ be a tower diagram of the permutation  $\omega$. Then $\pi_{_\Tt}$ is a permutation and moreover $\omega=\pi_\Tt$.

\end{pro}

\begin{proof} It is easy to check that the statement of the theorem is true when $\Tt$ is the empty diagram or it has just one cell. By induction assume that for all diagram of size $k-1$ the statement is true. Let $\Tt$ be a tower diagram of size $k$, $\Tt_l$ be the first nonempty tower and  $d$ be the top cell of  $\Tt_l$, with flight number  $i$. Observe that, in this case, $d$ is a corner cell.

Now  let $\mathcal U$ be a tower diagram  obtained by  erasing $d$ from $\Tt$ and $u$ be the corresponding permutation. By induction we can assume that $\pi_{_\mathcal U}=u$.
On the other hand we have
 $$\omega=u\cdot s_i=\pi_{_\mathcal U}\cdot s_i.$$
  Hence the theorem will be proved once we show that  $\pi_{_\Tt} = \pi_{_\mathcal U}\cdot s_i$.

Consider the following figure which illustrates the  first  $l$ towers of $\mathcal U$ and   $\mathcal T$. In the diagram of  $\mathcal U$,   the stars represent the cells in the flight path of $x_l$. Since the  flight number of $x_l$ is $i$ and since $\pi_{_\mathcal U}=u$ is a permutation we see that  $\pi_{_\mathcal U}$-index of  $\mathcal U$ is of the form $(1,\ldots,l-1,i,\ldots,i+1,\ldots,)$. Therefore, for some $r>l$, there must be an empty cell $x_r$ lying on top of its $r$-th tower, with flight number $i+1$.  Moreover, as illustrated by circles,  the flight path of $x_r$ passes through the cell on the right of $x_l$.
\begin{figure}[h]\setlength{\unitlength}{0.25mm}
	\begin{center}
		\begin{picture}(80,50)
		\multiput(0,0)(0,0){1}{\line(1,0){80}}
		\multiput(0,0)(0,0){1}{\line(0,1){60}}
		\multiput(0,10)(2,0){40}{\line(0,1){.1}}
		\multiput(0,20)(2,0){40}{\line(0,1){.1}}
		\multiput(0,30)(2,0){40}{\line(0,1){.1}}
		\multiput(0,40)(2,0){40}{\line(0,1){.1}}
		\multiput(0,50)(2,0){40}{\line(0,1){.1}}
		\multiput(0,60)(2,0){40}{\line(0,1){.1}}
		\multiput(10,0)(0,2){30}{\line(1,0){.1}}
		\multiput(20,0)(0,2){30}{\line(1,0){.1}}
		\multiput(30,0)(0,2){30}{\line(1,0){.1}}
		\multiput(40,0)(0,2){30}{\line(1,0){.1}}
		\multiput(50,0)(0,2){30}{\line(1,0){.1}}
		\multiput(60,0)(0,2){30}{\line(1,0){.1}}
		\multiput(70,0)(0,2){30}{\line(1,0){.1}}
		\multiput(80,0)(0,2){30}{\line(1,0){.1}} \put(-25,23){$\mathcal U =$}
		\put(30,0){\tableau{{}}}
		\put(1,4){$_{x_1}$}
		\put(11,4){$\ldots$}
		\put(31,14){$_{x_l}$}
		  \put(2,52){$\circ$} \put(12,42){$\circ$}\put(22,32){$\circ$} \put(32,22){$\circ$}\put(42,12){$\circ$}	
		  \put(1,42){$\star$} \put(11,32){$\star$}\put(21,22){$\star$}
		\end{picture}\hskip.5in
			\begin{picture}(80,50)
			\multiput(0,0)(0,0){1}{\line(1,0){80}}
			\multiput(0,0)(0,0){1}{\line(0,1){60}}
			\multiput(0,10)(2,0){40}{\line(0,1){.1}}
			\multiput(0,20)(2,0){40}{\line(0,1){.1}}
			\multiput(0,30)(2,0){40}{\line(0,1){.1}}
			\multiput(0,40)(2,0){40}{\line(0,1){.1}}
			\multiput(0,50)(2,0){40}{\line(0,1){.1}}
			\multiput(0,60)(2,0){40}{\line(0,1){.1}}
			\multiput(10,0)(0,2){30}{\line(1,0){.1}}
			\multiput(20,0)(0,2){30}{\line(1,0){.1}}
			\multiput(30,0)(0,2){30}{\line(1,0){.1}}
			\multiput(40,0)(0,2){30}{\line(1,0){.1}}
			\multiput(50,0)(0,2){30}{\line(1,0){.1}}
			\multiput(60,0)(0,2){30}{\line(1,0){.1}}
			\multiput(70,0)(0,2){30}{\line(1,0){.1}}
			\multiput(80,0)(0,2){30}{\line(1,0){.1}} \put(-25,23){$\mathcal T =$}
			\put(30,10){\tableau{{}\\{}}}
			\put(1,4){$_{x_1}$}
			\put(11,4){$\ldots$}
			\put(31,24){$_{x_l}$} \put(1,52){$\star$} \put(11,42){$\star$}\put(21,32){$\star$}
			\put(32,12){$\bullet$}
			\put(22,22){$\circ$}\put(12,32){$\circ$}\put(2,42){$\circ$}
			\put(42,12){$\circ$}
			\multiput(60,40)(0,0){1}{\vector(-1,-1){23}} \put(62,42){$d$}
			\end{picture}
	\end{center}
\end{figure}
Recall that the diagram of $\Tt$ is obtained by adding the cell $d$ to the $l$-th tower of  $\mathcal U$, so that $x_l$ is shifted up and all other empty cells lying on top of other towers remains same. In this case  the flight number of $x_l$ in $\Tt$ becomes  $i+1$. On the other hand, as illustrated in the above figure,  the flight path of $x_r$ in $\Tt$  contains  $d$ and hence  its new  flight number becomes $i$. It is easy to observe that flight numbers of  the remaining empty cell lying on top of towers do not change. Hence   $\pi_{_\mathcal T}$-index of $\Tt$ is obtained by interchanging $i$ and $i+1$ in  $\pi_{_\mathcal U}$-index of $\mathcal U$. This shows that $\pi_{_\mathcal T}= \pi_{_\mathcal U}\cdot s_i$ as required.

\end{proof}

 \section{Sliding  a transposition}\label{sec:slidetrans}

In this section, we describe the basic operations that could appear in the sliding of a transposition into a given tower
diagram. As remarked earlier, if $\omega$ is any permutation and
$t$ is any transposition, then these operations determine the tower diagram of the product $\omega\cdot t$
as a modification of the tower diagram of $\omega$.

Let $j\geq i>0$ be integers. We write $t_{i, j+1}$ for the transposition interchanging $i$ and $j+1$. A reduced expression of
 the transposition $t_{i,j+1}$ can be given as $s_{j}\ldots s_{i+1} s_{i}s_{i+1}\ldots s_{j}$.  Therefore the tower diagram of a transposition $t_{i,j+1}$ has  a
\textbf{\textit{hook shape}} that we denote by $\fh_{i,j}$ whose \textbf{\emph{heel}} has slide  $i$ and whose cells in the
\textbf{\emph{foot}} and in the \textbf{\emph{leg}} have slides from $i+1$ to $j$.

It turns out that there are five basic operations in the sliding of a transposition on an arbitrary tower diagram $\Tt$. To determine these basic operations, let $\mathcal T_l$ be the first nonempty tower  in $\Tt$ and let $t$ be the slide of its top cell.

 \textbf{1. Direct-pass:}  If $ t+1<i\leq j$ then  the hook   $\fh_{i,j}$  has no intersection with $\Tt_l$, and it continues its sliding
 on $\Tt_{>l}$  with the hook $\fh_{i,j}$ subject to the  conditions given in this list.

\textbf{2. Broken(+) pass:} If $t+1=i\leq j$, then the foot passes $\mathcal T_l$ directly, and the heel and the leg
sits on it, creating a new tower denoted by $\mathcal T_l'$ and its foot continues its sliding on $\Tt_{>l}$.  See the following  figure.
Observe that, sliding of the foot on $\Tt_{>l}$
 can erase at most $j-i$ cells. Therefore, a broken(+) pass always adds at least one cell to the given tower diagram.

 
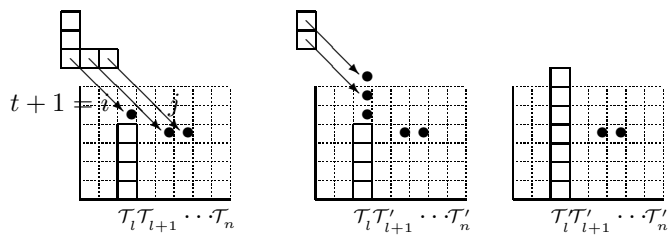
\begin{figure}[h]\setlength{\unitlength}{0.25mm}
\begin{center}
\begin{picture}(80,100)
\multiput(0,0)(0,0){1}{\line(1,0){80}}
\multiput(0,0)(0,0){1}{\line(0,1){60}}
\multiput(0,10)(2,0){40}{\line(0,1){.1}}
\multiput(0,20)(2,0){40}{\line(0,1){.1}}
\multiput(0,30)(2,0){40}{\line(0,1){.1}}
\multiput(0,40)(2,0){40}{\line(0,1){.1}}
\multiput(0,50)(2,0){40}{\line(0,1){.1}}
\multiput(0,60)(2,0){40}{\line(0,1){.1}}
\multiput(10,0)(0,2){30}{\line(1,0){.1}}
\multiput(20,0)(0,2){30}{\line(1,0){.1}}
\multiput(30,0)(0,2){30}{\line(1,0){.1}}
\multiput(40,0)(0,2){30}{\line(1,0){.1}}
\multiput(50,0)(0,2){30}{\line(1,0){.1}}
\multiput(60,0)(0,2){30}{\line(1,0){.1}}
\multiput(70,0)(0,2){30}{\line(1,0){.1}}
\multiput(80,0)(0,2){30}{\line(1,0){.1}}
\put(20,30){\tableau{{}\\{}\\{}\\{}}}
\put(20,-10){$_{\mathcal{T}_{_{l}}}$}\put(30,-10){$_{\mathcal{T}_{_{l+1}}}$}
\put(55,-10){$\ldots$} \put(70,-10){$_{\mathcal{T}_{_{n}}}$}
 \put(-10,90){\tableau{{}\\{}\\{}}}
\put(0,70){\tableau{{}}} \put(10,70){\tableau{{}}}
\multiput(-5,75)(0,0){1}{\vector(1,-1){28}} \put(-37,47){$t+1=i$}
\multiput(5,75)(0,0){1}{\vector(1,-1){38}}
\multiput(15,75)(0,0){1}{\vector(1,-1){38}}\put(47,47){$j$}
\put(24,42){$\bullet$}
 \put(44,32){$\bullet$}
\put(54,32){$\bullet$}
\end{picture}
\hskip.4in
\begin{picture}(80,100)
\multiput(0,0)(0,0){1}{\line(1,0){80}}
\multiput(0,0)(0,0){1}{\line(0,1){60}}
\multiput(0,10)(2,0){40}{\line(0,1){.1}}
\multiput(0,20)(2,0){40}{\line(0,1){.1}}
\multiput(0,30)(2,0){40}{\line(0,1){.1}}
\multiput(0,40)(2,0){40}{\line(0,1){.1}}
\multiput(0,50)(2,0){40}{\line(0,1){.1}}
\multiput(0,60)(2,0){40}{\line(0,1){.1}}
\multiput(10,0)(0,2){30}{\line(1,0){.1}}
\multiput(20,0)(0,2){30}{\line(1,0){.1}}
\multiput(30,0)(0,2){30}{\line(1,0){.1}}
\multiput(40,0)(0,2){30}{\line(1,0){.1}}
\multiput(50,0)(0,2){30}{\line(1,0){.1}}
\multiput(60,0)(0,2){30}{\line(1,0){.1}}
\multiput(70,0)(0,2){30}{\line(1,0){.1}}
\multiput(80,0)(0,2){30}{\line(1,0){.1}}
\put(20,30){\tableau{{}\\{}\\{}\\{}}}
 \put(20,-10){$_{\mathcal{T}_{_{l}}}$}
 \put(-10,90){\tableau{{}\\{}}}
\multiput(-5,85)(0,0){1}{\vector(1,-1){28}}
\multiput(-5,95)(0,0){1}{\vector(1,-1){28}}
\put(30,-10){$_{\mathcal{T}'_{_{l+1}}}$} \put(55,-10){$\ldots$}
\put(70,-10){$_{\mathcal{T}'_{_{n}}}$}
\put(44,32){$\bullet$}\put(54,32){$\bullet$}
\put(24,62){$\bullet$}\put(24,52){$\bullet$}\put(24,42){$\bullet$}
\end{picture}
\hskip.2in
\begin{picture}(80,100)
\multiput(0,0)(0,0){1}{\line(1,0){80}}
\multiput(0,0)(0,0){1}{\line(0,1){60}}
\multiput(0,10)(2,0){40}{\line(0,1){.1}}
\multiput(0,20)(2,0){40}{\line(0,1){.1}}
\multiput(0,30)(2,0){40}{\line(0,1){.1}}
\multiput(0,40)(2,0){40}{\line(0,1){.1}}
\multiput(0,50)(2,0){40}{\line(0,1){.1}}
\multiput(0,60)(2,0){40}{\line(0,1){.1}}
\multiput(10,0)(0,2){30}{\line(1,0){.1}}
\multiput(20,0)(0,2){30}{\line(1,0){.1}}
\multiput(30,0)(0,2){30}{\line(1,0){.1}}
\multiput(40,0)(0,2){30}{\line(1,0){.1}}
\multiput(50,0)(0,2){30}{\line(1,0){.1}}
\multiput(60,0)(0,2){30}{\line(1,0){.1}}
\multiput(70,0)(0,2){30}{\line(1,0){.1}}
\multiput(80,0)(0,2){30}{\line(1,0){.1}}
\put(20,60){\tableau{{}\\{}\\{}\\{}\\{}\\{}\\{}}}
\put(20,-10){$_{\mathcal{T}'_{_{l}}}$}
\put(30,-10){$_{\mathcal{T}'_{_{l+1}}}$} \put(55,-10){$\ldots$}
\put(70,-10){$_{\mathcal{T}'_{_{n}}}$}
\put(44,32){$\bullet$}\put(54,32){$\bullet$}
\end{picture}
\end{center}
\label{fig:b+p-t}
\end{figure}

\textbf{3.} \textbf{Shrunken pass:} If $i\leq t < j$, then as illustrated
in the following figure, there is a cell $c$ in the foot (and another one $c'$ in the leg) with slide equal to $t+1$.  As
illustrated in the figure, while the tower $\Tt_{l}$ remains unchanged, the hook $\fh_{i,j}$ passes it in a way that it
loses  $c$ and $c^\prime $ and  its heel has new slide $i+1$. Therefore sliding continues on $\Tt_{>l}$
with the new hook $\fh_{i+1,j}$.

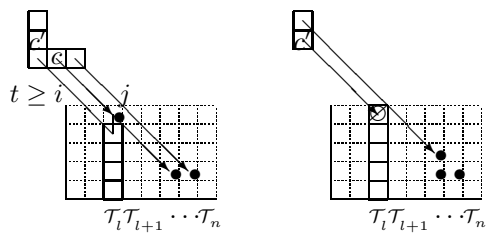
\begin{figure}[h]\setlength{\unitlength}{0.25mm}
\begin{center}
\begin{picture}(80,100)
\multiput(0,0)(0,0){1}{\line(1,0){80}}
\multiput(0,0)(0,0){1}{\line(0,1){50}}
\multiput(0,10)(2,0){40}{\line(0,1){.1}}
\multiput(0,20)(2,0){40}{\line(0,1){.1}}
\multiput(0,30)(2,0){40}{\line(0,1){.1}}
\multiput(0,40)(2,0){40}{\line(0,1){.1}}
\multiput(0,50)(2,0){40}{\line(0,1){.1}}
\multiput(10,0)(0,2){25}{\line(1,0){.1}}
\multiput(20,0)(0,2){25}{\line(1,0){.1}}
\multiput(30,0)(0,2){25}{\line(1,0){.1}}
\multiput(40,0)(0,2){25}{\line(1,0){.1}}
\multiput(50,0)(0,2){25}{\line(1,0){.1}}
\multiput(60,0)(0,2){25}{\line(1,0){.1}}
\multiput(70,0)(0,2){25}{\line(1,0){.1}}
\multiput(80,0)(0,2){25}{\line(1,0){.1}}
\put(20,30){\tableau{{}\\{}\\{}\\{}}}
 \put(20,-10){$_{\mathcal{T}_{_{l}}}$}
 \put(-20,90){\tableau{{}\\{c'}}}
 \put(-20,70){\tableau{{}}}
\put(-10,70){\tableau{{c}}}
\put(0,70){\tableau{{}}}
\multiput(-5,75)(0,0){1}{\vector(1,-1){30}}
\multiput(5,75)(0,0){1}{\vector(1,-1){60}}
 \put(25,40){$\bullet$}
\put(65,10){$\bullet$}
\multiput(-15,75)(0,0){1}{\line(1,-1){40}} \put(-30,52){$t\geq i$}
\multiput(25,35)(0,0){1}{\line(0,1){10}}
\multiput(25,45)(0,0){1}{\vector(1,-1){30}}\put(28,52){$j$}
\put(55,10){$\bullet$}
\put(30,-10){$_{\mathcal{T}_{_{l+1}}}$} \put(55,-10){$\ldots$}
\put(70,-10){$_{\mathcal{T}_{_{n}}}$}
\end{picture}\hskip.6in
\begin{picture}(80,100)
\multiput(0,0)(0,0){1}{\line(1,0){80}}
\multiput(0,0)(0,0){1}{\line(0,1){50}}
\multiput(0,10)(2,0){40}{\line(0,1){.1}}
\multiput(0,20)(2,0){40}{\line(0,1){.1}}
\multiput(0,30)(2,0){40}{\line(0,1){.1}}
\multiput(0,40)(2,0){40}{\line(0,1){.1}}
\multiput(0,50)(2,0){40}{\line(0,1){.1}}
\multiput(10,0)(0,2){25}{\line(1,0){.1}}
\multiput(20,0)(0,2){25}{\line(1,0){.1}}
\multiput(30,0)(0,2){25}{\line(1,0){.1}}
\multiput(40,0)(0,2){25}{\line(1,0){.1}}
\multiput(50,0)(0,2){25}{\line(1,0){.1}}
\multiput(60,0)(0,2){25}{\line(1,0){.1}}
\multiput(70,0)(0,2){25}{\line(1,0){.1}}
\multiput(80,0)(0,2){25}{\line(1,0){.1}}
\put(20,40){\tableau{{\varnothing}\\{}\\{}\\{}\\{}}}
 \put(20,-10){$_{\mathcal{T}_{_{l}}}$}
 \put(-20,90){\tableau{{}\\{c'}}}
\multiput(-15,85)(0,0){1}{\vector(1,-1){40}}
\multiput(-15,95)(0,0){1}{\vector(1,-1){70}}
\put(55,10){$\bullet$}
\put(65,10){$\bullet$}
\put(55,20){$\bullet$}
\put(30,-10){$_{\mathcal{T}_{_{l+1}}}$} \put(55,-10){$\ldots$}
\put(70,-10){$_{\mathcal{T}_{_{n}}}$}
\end{picture}\hskip.6in
\end{center}
\label{fig:sp-1}
\end{figure}

 \textbf{4. Broken(-)pass:} If $i\leq j=t$,  then the foot and the heel of $\fh_{i,j}$  deletes top $j-i+1 $ cells of
$\mathcal T_l$ and then its leg passes directly and  continues it sliding on $\mathcal T_{>l}$ as
illustrated in the following figure.

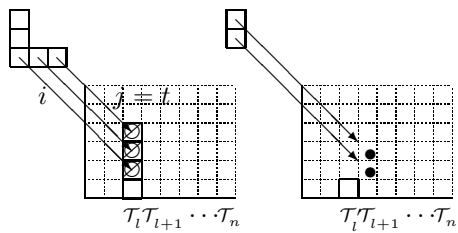
\begin{figure}[h]\setlength{\unitlength}{0.25mm}
\begin{center}
\begin{picture}(80,100)
\multiput(0,0)(0,0){1}{\line(1,0){80}}
\multiput(0,0)(0,0){1}{\line(0,1){60}}
\multiput(0,10)(2,0){40}{\line(0,1){.1}}
\multiput(0,20)(2,0){40}{\line(0,1){.1}}
\multiput(0,30)(2,0){40}{\line(0,1){.1}}
\multiput(0,40)(2,0){40}{\line(0,1){.1}}
\multiput(0,50)(2,0){40}{\line(0,1){.1}}
\multiput(0,60)(2,0){40}{\line(0,1){.1}}
\multiput(10,0)(0,2){30}{\line(1,0){.1}}
\multiput(20,0)(0,2){30}{\line(1,0){.1}}
\multiput(30,0)(0,2){30}{\line(1,0){.1}}
\multiput(40,0)(0,2){30}{\line(1,0){.1}}
\multiput(50,0)(0,2){30}{\line(1,0){.1}}
\multiput(60,0)(0,2){30}{\line(1,0){.1}}
\multiput(70,0)(0,2){30}{\line(1,0){.1}}
\multiput(80,0)(0,2){30}{\line(1,0){.1}}
\put(20,30){\tableau{{\varnothing
}\\{\varnothing}\\{\varnothing}\\{}}}
\put(20,-10){$_{\mathcal{T}_{_{l}}}$}\put(30,-10){$_{\mathcal{T}_{_{l+1}}}$}
\put(55,-10){$\ldots$} \put(70,-10){$_{\mathcal{T}_{_{n}}}$}
 \put(-40,90){\tableau{{}\\{}\\{}}}
\put(-30,70){\tableau{{}}} \put(-20,70){\tableau{{}}}
\multiput(-35,75)(0,0){1}{\vector(1,-1){60}}\put(-25,50){$i$}
\multiput(-25,75)(0,0){1}{\vector(1,-1){50}}
\multiput(-15,75)(0,0){1}{\vector(1,-1){40}}\put(15,50){$j=t$}
\end{picture}
\hskip.3in
\begin{picture}(80,100)
\multiput(0,0)(0,0){1}{\line(1,0){80}}
\multiput(0,0)(0,0){1}{\line(0,1){60}}
\multiput(0,10)(2,0){40}{\line(0,1){.1}}
\multiput(0,20)(2,0){40}{\line(0,1){.1}}
\multiput(0,30)(2,0){40}{\line(0,1){.1}}
\multiput(0,40)(2,0){40}{\line(0,1){.1}}
\multiput(0,50)(2,0){40}{\line(0,1){.1}}
\multiput(0,60)(2,0){40}{\line(0,1){.1}}
\multiput(10,0)(0,2){30}{\line(1,0){.1}}
\multiput(20,0)(0,2){30}{\line(1,0){.1}}
\multiput(30,0)(0,2){30}{\line(1,0){.1}}
\multiput(40,0)(0,2){30}{\line(1,0){.1}}
\multiput(50,0)(0,2){30}{\line(1,0){.1}}
\multiput(60,0)(0,2){30}{\line(1,0){.1}}
\multiput(70,0)(0,2){30}{\line(1,0){.1}}
\multiput(80,0)(0,2){30}{\line(1,0){.1}} \put(20,0){\tableau{{}}}
 \put(20,-10){$_{\mathcal{T}'_{_{l}}}$}
 \put(-40,90){\tableau{{}\\{}}}
\multiput(-35,85)(0,0){1}{\vector(1,-1){65}}
\multiput(-35,95)(0,0){1}{\vector(1,-1){65}}
 \put(30,-10){$_{\mathcal{T}_{_{l+1}}}$}
\put(55,-10){$\ldots$} \put(70,-10){$_{\mathcal{T}_{_{n}}}$}
\put(33,10){$\bullet$}
\put(33,20){$\bullet$}
\end{picture}
\end{center}
\label{fig:b-pl}
\end{figure}

\textbf{5.} \textbf{Zigzag pass:} If $i\leq j< t$, then each cell of the hook passes the tower $\mathcal T_l$ with a zigzag.
Therefore the sliding continues on $\mathcal T_{>l}$ with a new hook $\fh_{i+1,j+1}$. See the following figure.

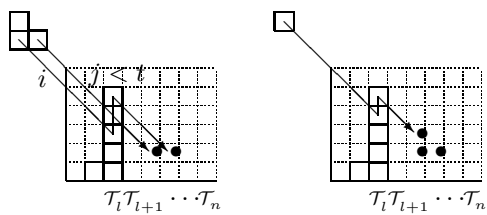
\begin{figure}[h]\setlength{\unitlength}{0.25mm}
\begin{center}
\begin{picture}(80,100)
\multiput(0,0)(0,0){1}{\line(1,0){80}}
\multiput(0,0)(0,0){1}{\line(0,1){60}}
\multiput(0,10)(2,0){40}{\line(0,1){.1}}
\multiput(0,20)(2,0){40}{\line(0,1){.1}}
\multiput(0,30)(2,0){40}{\line(0,1){.1}}
\multiput(0,40)(2,0){40}{\line(0,1){.1}}
\multiput(0,50)(2,0){40}{\line(0,1){.1}}
\multiput(0,60)(2,0){40}{\line(0,1){.1}}
\multiput(10,0)(0,2){30}{\line(1,0){.1}}
\multiput(20,0)(0,2){30}{\line(1,0){.1}}
\multiput(30,0)(0,2){30}{\line(1,0){.1}}
\multiput(40,0)(0,2){30}{\line(1,0){.1}}
\multiput(50,0)(0,2){30}{\line(1,0){.1}}
\multiput(60,0)(0,2){30}{\line(1,0){.1}}
\multiput(70,0)(0,2){30}{\line(1,0){.1}}
\multiput(80,0)(0,2){30}{\line(1,0){.1}}
\put(20,40){\tableau{{}\\{}\\{}\\{}\\{}}} \put(10,0){\tableau{{}}}
\put(20,-10){$_{\mathcal{T}_{_{l}}}$}\put(30,-10){$_{\mathcal{T}_{_{l+1}}}$}
\put(55,-10){$\ldots$} \put(70,-10){$_{\mathcal{T}_{_{n}}}$}
 \put(-30,80){\tableau{{}\\{}}}
\put(-20,70){\tableau{{}}}
\multiput(-25,75)(0,0){1}{\line(1,-1){50}}\put(-15,50){$i$}
\multiput(25,25)(0,0){1}{\line(0,1){10}}
\multiput(25,35)(0,0){1}{\vector(1,-1){19}}
\multiput(-15,75)(0,0){1}{\line(1,-1){40}}\put(12,53){$j<t$}
\multiput(25,35)(0,0){1}{\line(0,1){10}}
\multiput(25,45)(0,0){1}{\vector(1,-1){29}}\put(45,12){$\bullet$}\put(55,12){$\bullet$}
\end{picture}\hskip.6in
\begin{picture}(80,100)
\multiput(0,0)(0,0){1}{\line(1,0){80}}
\multiput(0,0)(0,0){1}{\line(0,1){60}}
\multiput(0,10)(2,0){40}{\line(0,1){.1}}
\multiput(0,20)(2,0){40}{\line(0,1){.1}}
\multiput(0,30)(2,0){40}{\line(0,1){.1}}
\multiput(0,40)(2,0){40}{\line(0,1){.1}}
\multiput(0,50)(2,0){40}{\line(0,1){.1}}
\multiput(0,60)(2,0){40}{\line(0,1){.1}}
\multiput(10,0)(0,2){30}{\line(1,0){.1}}
\multiput(20,0)(0,2){30}{\line(1,0){.1}}
\multiput(30,0)(0,2){30}{\line(1,0){.1}}
\multiput(40,0)(0,2){30}{\line(1,0){.1}}
\multiput(50,0)(0,2){30}{\line(1,0){.1}}
\multiput(60,0)(0,2){30}{\line(1,0){.1}}
\multiput(70,0)(0,2){30}{\line(1,0){.1}}
\multiput(80,0)(0,2){30}{\line(1,0){.1}}
\put(20,40){\tableau{{}\\{}\\{}\\{}\\{}}} \put(10,0){\tableau{{}}}
\put(20,-10){$_{\mathcal{T}_{_{l}}}$}\put(30,-10){$_{\mathcal{T}_{_{l+1}}}$}
\put(55,-10){$\ldots$} \put(70,-10){$_{\mathcal{T}_{_{n}}}$}
 \put(-30,80){\tableau{{}}}
\multiput(-25,85)(0,0){1}{\line(1,-1){50}}
\multiput(25,35)(0,0){1}{\line(0,1){10}}
\multiput(25,45)(0,0){1}{\vector(1,-1){19}}
\put(45,12){$\bullet$}\put(55,12){$\bullet$}\put(45,22){$\bullet$}
\end{picture}
\end{center}
\label{fig:zzhook}
\end{figure}

Now, we are ready to unearth the relation between  the generalized  flight number and the sliding of transpositions.

\begin{pro}\label{pro:slideflight}
Let $c$ be a top cell in  a tower diagram $\mathcal T$  with the generalized  flight number  $(i,j)$.  Then
$$ \fh_{i,j} \searrow (\mathcal T - c)= \mathcal T.
$$
where $\fh_{i,j}$ is the hook corresponding to the transposition $t_{i,j+1}$ and $\mathcal T - c$ be the tower
diagram obtained by erasing $c$ from $\Tt$.
Moreover if $\omega$ is a permutation with $\Tt_\omega=\mathcal T- c$ then  $\Tt_{\omega\cdot t_{i,j+1}}=\Tt$.
\end{pro}
\begin{proof} Let $c_1,\ldots, c_l=c$ be the cells in the flight path of $c$ ordered from left to right lying in or over the
towers $\Tt_1,\ldots,\Tt_l$ of $\Tt$, respectively.  We start by sliding  $\fh_{i,j}$ on the tower $\Tt_1$. There are several cases:

\textbf{Case 1.} If $c_1$ is an empty cell lying above $\Tt_1$ then the cell below $c_1$ must also be empty and the
slide of $c_1$ must be $i$. Then $\fh_{i,j}$ pass directly  $\Tt_1$ and
$$ \fh_{i,j} \searrow \mathcal (T- c)= \Tt_{1}  \sqcup  \fh_{i,j} \searrow (\mathcal T_{>1} - c )
 $$
while   the generalized  flight number  of $c$ in $\Tt_{>1}$ is still $(i,j)$ as illustrated in the following figure.
\begin{figure}[h]\setlength{\unitlength}{0.25mm}
\begin{center}
\begin{picture}(80,70)
\multiput(0,0)(0,0){1}{\line(1,0){80}}
\multiput(0,0)(0,0){1}{\line(0,1){60}}
\multiput(0,10)(2,0){40}{\line(0,1){.1}}
\multiput(0,20)(2,0){40}{\line(0,1){.1}}
\multiput(0,30)(2,0){40}{\line(0,1){.1}}
\multiput(0,40)(2,0){40}{\line(0,1){.1}}
\multiput(0,50)(2,0){40}{\line(0,1){.1}}
\multiput(0,60)(2,0){40}{\line(0,1){.1}}
\multiput(10,0)(0,2){30}{\line(1,0){.1}}
\multiput(20,0)(0,2){30}{\line(1,0){.1}}
\multiput(30,0)(0,2){30}{\line(1,0){.1}}
\multiput(40,0)(0,2){30}{\line(1,0){.1}}
\multiput(50,0)(0,2){30}{\line(1,0){.1}}
\multiput(60,0)(0,2){30}{\line(1,0){.1}}
\multiput(70,0)(0,2){30}{\line(1,0){.1}}
\multiput(80,0)(0,2){30}{\line(1,0){.1}}
\put(0,10){\tableau{{}\\{}}}
 \put(10,40){\tableau{{}\\{}\\{\bullet}\\{}\\{}}}
\put(20,20){\tableau{{\bullet}\\{}\\{}}}
\put(30,0){\tableau{{}}} \put(40,0){\tableau{{}}}
\put(53,23){} \put(43,33){}\put(2,33){$\circ$} \put(42,11){c}\put(32,21){$\circ$}
 \put(60,30){\tableau{{}\\{}\\{}\\{}}}
\put(70,10){\tableau{{}\\{}}}
\put(0,-10){$_{\mathcal{T}_{_{1}}}$}\put(10,-10){$_{\mathcal{T}_{_{2}}}$}
\put(25,-10){$\ldots$} \put(40,-10){$_{\mathcal{T}_{_{l}}}$}  \put(-30,70){\tableau{{}\\{}}}
\put(-20,60){\tableau{{}}}
\multiput(-25,65)(0,0){1}{\vector(1,-1){35}}
\end{picture}
\end{center}
\caption{Sliding of $\fh_{4,6}$ in to the tower diagram $\Tt -c$ where the cell $c$ has the generalized flight number $(4,5)$}
\end{figure}

\textbf{Case 2.} If $c_1$ lies in   $\Tt_1$ and the number of nonempty cells lying above $c_1$ is less than or equal to $r$ then  $\fh_{i,j}$ makes a shrunken pass through $\Tt_1$.
That is it continues its sliding on $\Tt_{>1}$ with the new hook $\fh_{i+1,j}$ and
$$ \fh_{i,j} \searrow \mathcal (T- c)= \Tt_{1}  \sqcup  \fh_{i+1,j} \searrow (\mathcal T_{>1} - c ).
 $$
 Observe that in this case  the generalized  flight number  of $c$ in $\Tt_{>1}$ is in fact  $(i+1,j)$ as illustrated in the following figure.
\begin{figure}[h]\setlength{\unitlength}{0.25mm}
\begin{center}
\begin{picture}(80,70)
\multiput(0,0)(0,0){1}{\line(1,0){80}}
\multiput(0,0)(0,0){1}{\line(0,1){60}}
\multiput(0,10)(2,0){40}{\line(0,1){.1}}
\multiput(0,20)(2,0){40}{\line(0,1){.1}}
\multiput(0,30)(2,0){40}{\line(0,1){.1}}
\multiput(0,40)(2,0){40}{\line(0,1){.1}}
\multiput(0,50)(2,0){40}{\line(0,1){.1}}
\multiput(0,60)(2,0){40}{\line(0,1){.1}}
\multiput(10,0)(0,2){30}{\line(1,0){.1}}
\multiput(20,0)(0,2){30}{\line(1,0){.1}}
\multiput(30,0)(0,2){30}{\line(1,0){.1}}
\multiput(40,0)(0,2){30}{\line(1,0){.1}}
\multiput(50,0)(0,2){30}{\line(1,0){.1}}
\multiput(60,0)(0,2){30}{\line(1,0){.1}}
\multiput(70,0)(0,2){30}{\line(1,0){.1}}
\multiput(80,0)(0,2){30}{\line(1,0){.1}}
\put(10,10){\tableau{{}\\{}}}
 \put(0,30){\tableau{{\bullet}\\{}\\{}\\{}}}
\put(20,40){\tableau{{}\\{}\\{\bullet}\\{}\\{}}}
\put(30,0){\tableau{{}}} \put(40,0){\tableau{{}}}
\put(53,23){} \put(43,33){}\put(12,33){$\circ$} \put(42,11){c}\put(32,21){$\circ$}
 \put(60,30){\tableau{{}\\{}\\{}\\{}}}
\put(70,10){\tableau{{}\\{}}}
\put(20,40){\tableau{{}\\{}\\{}\\{}\\{}}} \put(10,0){\tableau{{}}}
\put(0,-10){$_{\mathcal{T}_{_{1}}}$}\put(10,-10){$_{\mathcal{T}_{_{2}}}$}
\put(25,-10){$\ldots$} \put(40,-10){$_{\mathcal{T}_{_{l}}}$}
 \put(-30,70){\tableau{{}\\{}}}
\put(-20,60){\tableau{{}}}
\multiput(-25,65)(0,0){1}{\vector(1,-1){30}}
\end{picture}
\end{center}
\caption{Sliding of $\fh_{4,6}$ in to the tower diagram  $\Tt -c$ where the cell $c$ has the generalized flight number $(4,5)$}
\end{figure}

\textbf{Case 3.} If $c_1$ lies in   $\Tt_1$ and the number of nonempty cell lying above $c_1$ is greater than $r$ then  $\fh_{i,j}$ makes a zigzag pass through $\Tt_1$.
That is it continues its sliding on $\Tt_{>1}$ with the new hook $\fh_{i+1,j+1}$ and
$$ \fh_{i,j} \searrow \mathcal (T- c)= \Tt_{1} \cup \fh_{i+1,j+1} \searrow (\mathcal T_{>1} - c ).
 $$
 Observe that in this case  the generalized flight number  of $c$ in $\Tt_{>1}$ is in fact  $(i+1,j+1)$ as illustrated in the following figure.
\begin{figure}[h]\setlength{\unitlength}{0.25mm}
\begin{center}
\begin{picture}(80,70)
\multiput(0,0)(0,0){1}{\line(1,0){80}}
\multiput(0,0)(0,0){1}{\line(0,1){60}}
\multiput(0,10)(2,0){40}{\line(0,1){.1}}
\multiput(0,20)(2,0){40}{\line(0,1){.1}}
\multiput(0,30)(2,0){40}{\line(0,1){.1}}
\multiput(0,40)(2,0){40}{\line(0,1){.1}}
\multiput(0,50)(2,0){40}{\line(0,1){.1}}
\multiput(0,60)(2,0){40}{\line(0,1){.1}}
\multiput(10,0)(0,2){30}{\line(1,0){.1}}
\multiput(20,0)(0,2){30}{\line(1,0){.1}}
\multiput(30,0)(0,2){30}{\line(1,0){.1}}
\multiput(40,0)(0,2){30}{\line(1,0){.1}}
\multiput(50,0)(0,2){30}{\line(1,0){.1}}
\multiput(60,0)(0,2){30}{\line(1,0){.1}}
\multiput(70,0)(0,2){30}{\line(1,0){.1}}
\multiput(80,0)(0,2){30}{\line(1,0){.1}}
\put(10,10){\tableau{{}\\{}}}
 \put(0,50){\tableau{{}\\{}\\{\bullet}\\{}\\{}\\{}}}
\put(20,20){\tableau{{\bullet}\\{}\\{}}}
\put(30,0){\tableau{{}}} \put(40,0){\tableau{{}}}
\put(53,23){} \put(43,33){}\put(12,33){$\circ$} \put(42,11){c}\put(32,21){$\circ$}
 \put(60,30){\tableau{{}\\{}\\{}\\{}}}
\put(70,10){\tableau{{}\\{}}}
\put(10,0){\tableau{{}}}
\put(0,-10){$_{\mathcal{T}_{_{1}}}$}\put(10,-10){$_{\mathcal{T}_{_{2}}}$}
\put(25,-10){$\ldots$} \put(40,-10){$_{\mathcal{T}_{_{l}}}$}
 \put(-30,70){\tableau{{}\\{}}}
\put(-20,60){\tableau{{}}}
\multiput(-25,65)(0,0){1}{\vector(1,-1){30}}
\end{picture}
\end{center}
\caption{Sliding of $\fh_{4,6}$ in to the tower diagram  $\Tt -c$ where the cell $c$ has the generalized flight number $(4,1)$}
\end{figure}
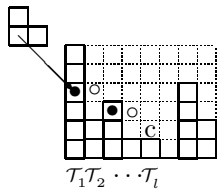

Iterating the above procedure,  we see that in  the sliding of $\fh_{i,j}$, hooks pass through $\Tt_1,\ldots,\Tt_{l-1}$ with a direct or a shrunken or a zigzag pass  and a shrunken pass occurs exactly on such towers at which the hook number of $c$ increases by one. Hence when the sliding arrives to $\Tt_l$ it already shrunk to a single cell with slide equal to the one of $c$. Hence the sliding terminates by filling the cell $c$.
\end{proof}

\begin{rem}\label{rem1}  It is clear from its definition that the sliding of a hook $\fh_{i,j}=s_{j}\ldots s_{i+1}s_is_{i+1}\ldots s_{j}$ can be done tower by tower, obeying
one of the above basic steps at each tower. Among the five basic steps, the only one which increases the size of the given tower
diagram is a broken(+) pass. Recall that, if $\fh_{i,j}$ passes a tower say $\Tt_l$ with a  broken(+) pass,  its  foot (corresponding to $s_{j}\ldots s_{i+1}$)
continues to its sliding on $\Tt_{>l}$ while  the heel and the feet (corresponding to $s_is_{i+1}\ldots s_{j}$) are placed on top $\Tt_l$.
Below, we show that if the sliding of the foot erases $j-i$ cells from $\Tt_{>l}$, then these cells must be the top $j-i
$ cells of a unique tower.

Suppose that in the  sliding $ s_{j}\ldots s_{i+1}$ into $\mathcal{T}_{>l}$,   $s_{j}$ deletes a top cell of a  $\mathcal{T}_s$ for some $l<s$. Denote by $\Tt_s^\prime$ the tower obtained by removing this top cell of $\Tt_s$.

Now, if the sliding of $s_{j-1}$  adds a cell on top of a tower, say $\mathcal{T}_{t}$, then  $t\leq s$  since the sliding path of $s_{j-1}$ always lies below the sliding path of $s_{j}$, as illustrated in Figure \ref{fig:remove-rest222}.

\begin{figure}[h]\setlength{\unitlength}{0.25mm}
\begin{center}
\begin{picture}(120,95)
\multiput(0,0)(0,0){1}{\line(1,0){120}}
\multiput(0,0)(0,0){1}{\line(0,1){60}}
\multiput(0,10)(2,0){60}{\line(0,1){.1}}
\multiput(0,20)(2,0){60}{\line(0,1){.1}}
\multiput(0,30)(2,0){60}{\line(0,1){.1}}
\multiput(0,40)(2,0){60}{\line(0,1){.1}}
\multiput(0,50)(2,0){60}{\line(0,1){.1}}
\multiput(0,60)(2,0){60}{\line(0,1){.1}}
\multiput(10,0)(0,2){30}{\line(1,0){.1}}
\multiput(20,0)(0,2){30}{\line(1,0){.1}}
\multiput(30,0)(0,2){30}{\line(1,0){.1}}
\multiput(40,0)(0,2){30}{\line(1,0){.1}}
\multiput(50,0)(0,2){30}{\line(1,0){.1}}
\multiput(60,0)(0,2){30}{\line(1,0){.1}}
\multiput(70,0)(0,2){30}{\line(1,0){.1}}
\multiput(80,0)(0,2){30}{\line(1,0){.1}}
\multiput(90,0)(0,2){30}{\line(1,0){.1}}
\multiput(100,0)(0,2){30}{\line(1,0){.1}}
\multiput(110,0)(0,2){30}{\line(1,0){.1}}
\multiput(120,0)(0,2){30}{\line(1,0){.1}}
\put(20,40){\tableau{{}\\{}\\{}\\{}\\{}}} \put(10,0){\tableau{{}}}
\put(30,0){\tableau{{}}} \put(50,20){\tableau{{}\\{}\\{}}}
\put(60,40){\tableau{{}\\{}\\{}\\{}\\{}}} \put(80,0){\tableau{{}}}
\put(90,10){\tableau{{\varnothing}\\{}}}
  \put(-10,100){\tableau{{}\\{}\\{}}}
\put(0,80){\tableau{{}}} \put(10,80){\tableau{{}}}
\multiput(15,85)(0,0){1}{\line(1,-1){50}}
\multiput(5,85)(0,0){1}{\vector(1,-1){50}}\put(53,33){$\bullet$}
\multiput(65,35)(0,0){1}{\line(0,1){10}}
\multiput(65,45)(0,0){1}{\vector(1,-1){28}} \put(27,75){$s_{j}$} \put(0,65){$s_{j-1}$}
\put(20,-7){$_{\mathcal{T}_{_{l}}}$}\put(50,-7){$_{\mathcal{T}_{_{t}}}$}
\put(60,-7){$\ldots$}\put(90,-7){$_{\mathcal{T}_{_{s}}}$}\put(100,-7){$\ldots$}
\put(110,-7){$_{\mathcal{T}_{_{n}}}$}
\end{picture}
\hskip.6in
\begin{picture}(120,95)
\multiput(0,0)(0,0){1}{\line(1,0){120}}
\multiput(0,0)(0,0){1}{\line(0,1){60}}
\multiput(0,10)(2,0){60}{\line(0,1){.1}}
\multiput(0,20)(2,0){60}{\line(0,1){.1}}
\multiput(0,30)(2,0){60}{\line(0,1){.1}}
\multiput(0,40)(2,0){60}{\line(0,1){.1}}
\multiput(0,50)(2,0){60}{\line(0,1){.1}}
\multiput(0,60)(2,0){60}{\line(0,1){.1}}
\multiput(10,0)(0,2){30}{\line(1,0){.1}}
\multiput(20,0)(0,2){30}{\line(1,0){.1}}
\multiput(30,0)(0,2){30}{\line(1,0){.1}}
\multiput(40,0)(0,2){30}{\line(1,0){.1}}
\multiput(50,0)(0,2){30}{\line(1,0){.1}}
\multiput(60,0)(0,2){30}{\line(1,0){.1}}
\multiput(70,0)(0,2){30}{\line(1,0){.1}}
\multiput(80,0)(0,2){30}{\line(1,0){.1}}
\multiput(90,0)(0,2){30}{\line(1,0){.1}}
\multiput(100,0)(0,2){30}{\line(1,0){.1}}
\multiput(110,0)(0,2){30}{\line(1,0){.1}}
\multiput(120,0)(0,2){30}{\line(1,0){.1}}
\put(20,40){\tableau{{}\\{}\\{}\\{}\\{}}} \put(10,0){\tableau{{}}}
\put(30,0){\tableau{{}}} \put(50,10){\tableau{{}\\{}}}
\put(60,40){\tableau{{}\\{}\\{}\\{}\\{}}} \put(80,0){\tableau{{}}}
\put(90,10){\tableau{{\varnothing}\\{}}}
  \put(-10,100){\tableau{{}\\{}\\{}}}
\put(0,80){\tableau{{}}} \put(10,80){\tableau{{}}}
\multiput(15,85)(0,0){1}{\line(1,-1){50}}
\multiput(65,35)(0,0){1}{\line(0,1){10}}
\multiput(65,45)(0,0){1}{\vector(1,-1){28}}
\multiput(5,85)(0,0){1}{\line(1,-1){60}}
\multiput(65,25)(0,0){1}{\line(0,1){10}}
\multiput(65,35)(0,0){1}{\vector(1,-1){18}} \put(83,13){$\bullet$}
\put(27,75){$s_{j}$} \put(0,65){$s_{j-1}$}
\put(20,-7){$_{\mathcal{T}_{_{l}}}$}\put(50,-7){$_{\mathcal{T}_{_{t}}}$}
\put(60,-7){$\ldots$}\put(90,-7){$_{\mathcal{T}_{_{s}}}$}\put(100,-7){$\ldots$}
\put(110,-7){$_{\mathcal{T}_{_{n}}}$}
\end{picture}
\end{center}
\caption{}\label{fig:remove-rest222}
\end{figure}

In fact, if  the sliding of $s_{j-1}$ arrives to the tower $\Tt_s^\prime$, as illustrated in Figure \ref{fig:remove-rest2},  then it must delete the top cell of $\Tt_s^\prime$. Hence $t< s$ in this case.

\begin{figure}[h]\setlength{\unitlength}{0.25mm}
\begin{center}
\begin{picture}(120,95)
\multiput(0,0)(0,0){1}{\line(1,0){120}}
\multiput(0,0)(0,0){1}{\line(0,1){60}}
\multiput(0,10)(2,0){60}{\line(0,1){.1}}
\multiput(0,20)(2,0){60}{\line(0,1){.1}}
\multiput(0,30)(2,0){60}{\line(0,1){.1}}
\multiput(0,40)(2,0){60}{\line(0,1){.1}}
\multiput(0,50)(2,0){60}{\line(0,1){.1}}
\multiput(0,60)(2,0){60}{\line(0,1){.1}}
\multiput(10,0)(0,2){30}{\line(1,0){.1}}
\multiput(20,0)(0,2){30}{\line(1,0){.1}}
\multiput(30,0)(0,2){30}{\line(1,0){.1}}
\multiput(40,0)(0,2){30}{\line(1,0){.1}}
\multiput(50,0)(0,2){30}{\line(1,0){.1}}
\multiput(60,0)(0,2){30}{\line(1,0){.1}}
\multiput(70,0)(0,2){30}{\line(1,0){.1}}
\multiput(80,0)(0,2){30}{\line(1,0){.1}}
\multiput(90,0)(0,2){30}{\line(1,0){.1}}
\multiput(100,0)(0,2){30}{\line(1,0){.1}}
\multiput(110,0)(0,2){30}{\line(1,0){.1}}
\multiput(120,0)(0,2){30}{\line(1,0){.1}}
\put(20,40){\tableau{{}\\{}\\{}\\{}\\{}}} \put(10,0){\tableau{{}}}
\put(30,0){\tableau{{}}} \put(50,10){\tableau{{}\\{}}}
\put(60,40){\tableau{{}\\{}\\{}\\{}\\{}}}
\put(90,10){\tableau{{\varnothing}\\{\varnothing}}}
  \put(-10,100){\tableau{{}\\{}\\{}}}
\put(0,80){\tableau{{}}} \put(10,80){\tableau{{}}}
\multiput(15,85)(0,0){1}{\line(1,-1){50}}
\multiput(65,35)(0,0){1}{\line(0,1){10}}
\multiput(65,45)(0,0){1}{\vector(1,-1){28}}
\multiput(5,85)(0,0){1}{\line(1,-1){60}}
\multiput(65,25)(0,0){1}{\line(0,1){10}}
\multiput(65,35)(0,0){1}{\vector(1,-1){28}}
\put(27,75){$s_{j}$} \put(0,65){$s_{j-1}$}
\put(20,-7){$_{\mathcal{T}_{_{l}}}$}\put(50,-7){$_{\mathcal{T}_{_{t}}}$}
\put(60,-7){$\ldots$}\put(90,-7){$_{\mathcal{T}_{_{s}}}$}\put(100,-7){$\ldots$}
\put(110,-7){$_{\mathcal{T}_{_{n}}}$}
\end{picture}
\end{center}
\caption{}\label{fig:remove-rest2}
\end{figure}
On the other hand if the sliding of $s_{j-1}$  removes the top cell of some  tower, say $\mathcal{T}_{t}$ for some $t<s$,  then sliding of $s_{j}$ would fill the empty  cell on  top of $\Tt_t$ in the first place, since this empty cell lies on the sliding path of $s_{j}$. Hence $t=s$ in this case.

Iterating the argument above we can conclude that if  $\fh_{i,j}$ passes a tower say $\Tt_l$ with a  broken(+) pass and its  if its foot erases $r$ cell from $\Tt_{>l}$
then  these cells must be the top $r$ cells of a unique tower.
\end{rem}

Hence we have proved the following theorem which has crucial importance in the following sections.
\begin{thm}\label{thm:plus1}
Let $\omega$ be a permutation with tower diagram $\mathcal T_\omega$ and
let $t_{i,j+1}$ be a transposition corresponding to hook  $\fh_{i,j}$. Then
$$l(\omega\cdot t_{i,j+1}) = l(\omega) + 1$$
if and only if one of the following condition holds in $\fh_{i,j}\searrow \Tt_\omega $.
\begin{enumerate}
\item Being reduced to a single cell following a sequence of direct or shrunken or zigzag passes, the  sliding   $\fh_{i,j}$
into $\Tt_\omega $ fills the empty cell, say $e_0$  on top of some tower $\Tt_l$ with generalized flight number $(i,j)$.
\item By following a sequence of direct or shrunken or zigzag passes which is ended with a broken(+) pass the  sliding
$\fh_{i,j}$ into $\Tt_\omega $ puts $r+1$ cells, say $e_0,e_1,\ldots, e_{r}$, on top of some tower
$\Tt_l$
and erases top $r$ cells of some tower $\Tt_s$ with $s>l$ and $0<r\leq j-i$.
 \end{enumerate}
\end{thm}

\begin{rem} With the notation of the theorem, consider Figure \ref{fig:remove-rest} for an illustration of the second case of the theorem.  Let $\omega' = \omega\cdot t_{i,j+1}$ and $\Tt'$ be the tower diagram corresponding to $\omega'$. In the following we show that $\omega'$-index of $\Tt'$ is obtained by interchanging $i$ and $j+1$ in the $\omega$- index of $\Tt$.

\begin{figure}[h]\setlength{\unitlength}{0.25mm}
\begin{center}
\begin{picture}(120,95)
\multiput(0,0)(0,0){1}{\line(1,0){120}}
\multiput(0,0)(0,0){1}{\line(0,1){60}}
\multiput(0,10)(2,0){60}{\line(0,1){.1}}
\multiput(0,20)(2,0){60}{\line(0,1){.1}}
\multiput(0,30)(2,0){60}{\line(0,1){.1}}
\multiput(0,40)(2,0){60}{\line(0,1){.1}}
\multiput(0,50)(2,0){60}{\line(0,1){.1}}
\multiput(0,60)(2,0){60}{\line(0,1){.1}}
\multiput(10,0)(0,2){30}{\line(1,0){.1}}
\multiput(20,0)(0,2){30}{\line(1,0){.1}}
\multiput(30,0)(0,2){30}{\line(1,0){.1}}
\multiput(40,0)(0,2){30}{\line(1,0){.1}}
\multiput(50,0)(0,2){30}{\line(1,0){.1}}
\multiput(60,0)(0,2){30}{\line(1,0){.1}}
\multiput(70,0)(0,2){30}{\line(1,0){.1}}
\multiput(80,0)(0,2){30}{\line(1,0){.1}}
\multiput(90,0)(0,2){30}{\line(1,0){.1}}
\multiput(100,0)(0,2){30}{\line(1,0){.1}}
\multiput(110,0)(0,2){30}{\line(1,0){.1}}
\multiput(120,0)(0,2){30}{\line(1,0){.1}}
\put(20,30){\tableau{{}\\{}\\{}\\{}}} \put(10,0){\tableau{{}}}
\put(30,0){\tableau{{}}} \put(50,10){\tableau{{}\\{}}} \put(50,25){$_{x_l}$}\put(80,35){$_{x_s}$}
\put(60,40){\tableau{{}\\{}\\{}\\{}\\{}}}
\put(80,20){\tableau{{\varnothing}\\{\varnothing}\\{}}}\put(80,35){$_{x_s}$}
  \put(-10,100){\tableau{{}\\{}\\{}}}
\put(0,80){\tableau{{}}} \put(10,80){\tableau{{}}}
\multiput(15,85)(0,0){1}{\line(1,-1){50}}
\multiput(65,35)(0,0){1}{\line(0,1){10}}
\multiput(65,45)(0,0){1}{\vector(1,-1){18}}
\multiput(5,85)(0,0){1}{\line(1,-1){60}}
\multiput(65,25)(0,0){1}{\line(0,1){10}}
\multiput(65,35)(0,0){1}{\vector(1,-1){18}}
\put(27,75){$s_{j}$} \put(0,65){$s_{i+1}$}
\put(20,-7){$_{\mathcal{T}_{_{l}}}$}\put(50,-7){$_{\mathcal{T}_{_{l}}}$}
\put(60,-7){$\ldots$}\put(80,-7){$_{\mathcal{T}_{_{s}}}$}\put(100,-7){$\ldots$}
\put(110,-7){$_{\mathcal{T}_{_{n}}}$}
\put(-35,45){$\Tt=$}
\end{picture}
\hskip .2in
\begin{picture}(120,95)
\multiput(0,0)(0,0){1}{\line(1,0){120}}
\multiput(0,0)(0,0){1}{\line(0,1){60}}
\multiput(0,10)(2,0){60}{\line(0,1){.1}}
\multiput(0,20)(2,0){60}{\line(0,1){.1}}
\multiput(0,30)(2,0){60}{\line(0,1){.1}}
\multiput(0,40)(2,0){60}{\line(0,1){.1}}
\multiput(0,50)(2,0){60}{\line(0,1){.1}}
\multiput(0,60)(2,0){60}{\line(0,1){.1}}
\multiput(10,0)(0,2){30}{\line(1,0){.1}}
\multiput(20,0)(0,2){30}{\line(1,0){.1}}
\multiput(30,0)(0,2){30}{\line(1,0){.1}}
\multiput(40,0)(0,2){30}{\line(1,0){.1}}
\multiput(50,0)(0,2){30}{\line(1,0){.1}}
\multiput(60,0)(0,2){30}{\line(1,0){.1}}
\multiput(70,0)(0,2){30}{\line(1,0){.1}}
\multiput(80,0)(0,2){30}{\line(1,0){.1}}
\multiput(90,0)(0,2){30}{\line(1,0){.1}}
\multiput(100,0)(0,2){30}{\line(1,0){.1}}
\multiput(110,0)(0,2){30}{\line(1,0){.1}}
\multiput(120,0)(0,2){30}{\line(1,0){.1}}
\put(20,30){\tableau{{}\\{}\\{}\\{}}} \put(10,0){\tableau{{}}}
\put(30,0){\tableau{{}}} \put(50,40){\tableau{{}\\{}\\{}\\{}\\{}}}
\put(50,25){$_{e_0}$}\put(50,45){$_{e_r}$}\put(50,30){$\cdot$}\put(50,35){$\cdot$}\put(50,40){$\cdot$}
 \put(50,55){$_{x'_l}$}\put(80,15){$_{x'_s}$}
\put(60,40){\tableau{{}\\{}\\{}\\{}\\{}}}
\put(80,0){\tableau{{}}}
  \put(-10,100){\tableau{{}\\{}\\{}}}
\multiput(-5,105)(0,0){1}{\vector(1,-1){55}}
\multiput(-5,85)(0,0){1}{\vector(1,-1){55}}
\put(27,75){$s_{j}$} \put(0,65){$s_{i}$}
\put(20,-7){$_{\mathcal{T}_{_{l}}}$}\put(50,-7){$_{\mathcal{T}'_{_{l}}}$}
\put(60,-7){$\ldots$}\put(80,-7){$_{\mathcal{T}'_{_{s}}}$}\put(100,-7){$\ldots$}
\put(110,-7){$_{\mathcal{T}_{_{n}}}$}
\put(125,45){$=\Tt'$}
\end{picture}
\end{center}
\caption{}\label{fig:remove-rest}
\end{figure}

Observe that the empty cell $x_l$  can be filled if we slide   $\fh_{i,j-r}$ into $\Tt$. Hence the flight number of $x_l$ in $\Tt$ is  $i$.  On the other hand
the top cell of $\Tt_s$ is a corner cell with flight number $j$ since it is erased by the sliding of $s_j$.  Now it is easy to deduce that the flight number of the empty cell $x_s$ is $j+1$, since it lies on top of a corner cell with flight number $j$. Hence in the $\omega$-index of $\Tt$, the $l$-th index is $i$, whereas the $s$-th index is $j+1$.

Similarly the top cell $e_r$ in $\Tt'_l$ is produced by sliding of $s_j$, hence it is a corner cell with flight number $j$. Therefore by being the empty cell on top this corner cell $x'_l$ has flight  number $j+1$. Here observe that if we  slid $\fh_{i,j-r}$ into $\Tt'$,  it would be reduced to a single cell when it arrived $\Tt'_l$.  Moreover this single cell would fill $x'_s$ after making a zigzag pass at this tower. This argument shows that the flight number of  $x'_s$ is $i$. Hence  in the $\omega'$- index of $\Tt'$, the $l$-th index is $j+1$ whereas the $s$-index is $i$.

For the first case the same result can be shown easily by following the argument  used in the proof of the Proposition \ref{line.notation}.
\end{rem}

\begin{cor}\label{cor: inc-dec}
Given a tower diagram $\mathcal T$, and a transposition $t_{i,j}\in S_n$ with $i<j$. Let $\omega = \omega_\mathcal T$ and
$\mathcal T_s$ be the tower with $\omega$-index $i$ and $\mathcal T_l$ be the tower with $\omega$-index $j$ in $\mathcal T$. Then we have
$l(\omega t_{i, j}) = l(\omega) + 1$
only if $s < l$. Moreover,
if $l(\omega t_{i, j}) = l(\omega) + 1$, then
\begin{enumerate}
\item the tower diagram of $\omega t_{i, j}$ is obtained from $\mathcal T_\omega$ by modifying only the towers $\mathcal T_s$ and $\mathcal T_l$. In
particular, a new cell $c_s$ is added to the tower $\mathcal T_s$ and possibly some cells are moved from $\mathcal T_l$ to $\mathcal T_s$.
\item $j$ is the $\omega t_{i, j}$-index of the tower whose height is increased.
\item $i$ is the $\omega t_{i, j}$-index of the other modified tower whose height may or may not decreased.
\item The new cell $c_s$ has flight number $i$.
\end{enumerate}
\end{cor}


\section{Monk's rule}\label{sec:monks}

In this section, we use the above versions of the flight and the sliding algorithms to give a constructive description of the
Monk's rule \cite{M}. We refer to \cite{LS} for the theory of Schubert polynomials.

\begin{thm}[Monk's Rule]
Let $\omega$ be a permutation and let $\mathfrak S_\omega$ denote the Schubert polynomial of $\omega$. Also let
$s_k$ be an adjacent transposition. Then
\[
\mathfrak{S}_\omega \cdot \mathfrak S_{s_k} = \sum_{\omega^\prime\in \omega\wedge s_k}
\mathfrak S_{\omega^\prime}
\]
where $\omega\wedge s_k$ is the set of all permutations  of the form $\omega\cdot t_{i,j}$ of length $l(\omega)
+1$, where $i\le k<j$.
 \end{thm}

Recall that for a given tower diagram $\mathcal T$, the associated permutation is denoted by $\omega_\mathcal T$.
Our aim is to describe Monk's rule using this correspondence. For this aim, first we define, for a positive integer $k$, the
set
$$ k\cdot\mathcal T= \{ \mathcal U \mid \omega_\mathcal U\in \omega_\mathcal T \wedge s_k\}$$
i.e. $k\cdot\mathcal T$ is the set of all tower diagrams of the  permutations in $\omega_\mathcal T\wedge s_k$.
Note that, since the length of the permutations in $\omega_\mathcal T \wedge s_k$ are just one more than that of
$\omega_\mathcal T$, the sizes of the diagrams in $k\cdot\mathcal T$  must be one bigger than
that of $\mathcal T$. As discussed in Remark~\ref{rem1}, this
happens if and only if the sliding of a hook adds $r+1$ cells to a certain tower and deletes $r$ cells from a
tower lying to the right of it for some $r\le 0$. Therefore, we only need an algorithm to choose these special towers.

Next, we determine a path for $k$ in $\Tt$ which is used in this algorithm. This is a certain extension of the path which
consists of cells that the sliding of $k$ to $\Tt$ would pass through. We denote by $\mathcal V_i$ the vertical strip over the interval $[i-1, i]$.
\begin{defn}
Let $\Tt $ be a tower diagram and $k$ be a natural number. The \textbf{\textit{Schubert path}} $\mathcal P =\scp(k,\Tt)$
of $k$ in $\mathcal T$ is the set of all cells in the first quadrant such that
\begin{enumerate}
\item $|\mathcal P\cap \mathcal V_i| \le 1$ for each $i$ and
\item If $\mathcal P\cap \mathcal V_i = {(i, j)}$, then $\fn(\mathcal T, (i,j))\le k$ and for any $j'>j$, we have $\fn(\mathcal T, (i,j'))>k$.
\end{enumerate}
\end{defn}

\begin{rem}
The Schubert path $\mathcal P =\scp(k,\Tt)$ of $k$ in $\Tt$ can be determined recursively as follows. The path
$\mathcal P$ starts with the cell $(1, k-1)$ and
extends to the right and downwards in such a way that if a cell $c$ is in $\mathcal P$ then so is  the one of the following cells.
\begin{enumerate}
\item If $c\in \mathcal T$, then  the one on the right of $c$ is  also in $\mathcal P$
\item If $c\not \in \Tt$ and $c$ is above the $x$-axis, then the one on the down-right of $c$  is also in $\mathcal P$.
\end{enumerate}
Note that this characterization always appends a cell under the $x$-axis to the Schubert path. We append it to make the following definition easier. Note also
that, as seen in Figure \ref{fig:sp8}, we always label the cells in $\mathcal P$ by one of $\bullet, \circ$ or $\ast$. The distribution of these symbols is justified as
follows.
\end{rem}

\begin{defn}
A cell $c$ in $\mathcal P$ and $i$ be the number of empty cells lying below $c$.  Then the cell $c$ is called a
\textit{critical cell} if either
\begin{enumerate}
\item $i=0$ or
\item $i>0$ and the first cell $d\in \mathcal P\cap \mathcal T$ on the right hand side of $c$ satisfies that the number of
cells on and below it is greater than or equal to  $i$.
\end{enumerate}
We denote the set of all critical cells of the Schubert path $\mathcal P$ by $\crit[k,\Tt]$. Moreover a critical cell which lies on top of  (possibly empty) tower is called an \textit{essential cell}. We denote the set of all essential cells by   $\ess[k,\Tt]$.
\end{defn}

Now we label the critical cells by $\ast$, the cells in $\mathcal P\cap \mathcal T$ by $\bullet$, and the others by $\circ$.
To illustrate these definitions, see the Schubert paths of $8$ in two distinct tower diagrams shown in Figure \ref{fig:sp8}.
Note that in the first diagram only the fourth and fifth asterisks from the left are essential cells whereas in the second
one the first, third and the fifth asterisks are essential.
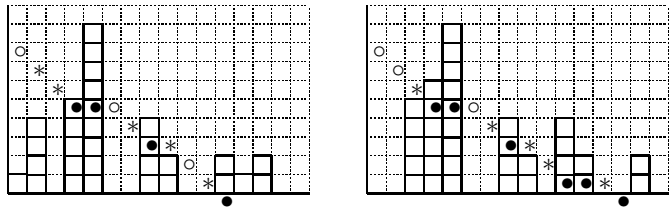
\begin{figure}[h]
\setlength{\unitlength}{0.25mm}
\begin{center}
\begin{picture}(80,60)
\multiput(0,0)(0,0){1}{\line(0,1){100}}
\multiput(0,0)(0,0){1}{\line(1,0){160}}
\multiput(0,10)(2,0){80}{\line(0,1){.1}}
\multiput(0,20)(2,0){80}{\line(0,1){.1}}
\multiput(0,30)(2,0){80}{\line(0,1){.1}}
\multiput(0,40)(2,0){80}{\line(0,1){.1}}
\multiput(0,50)(2,0){80}{\line(0,1){.1}}
\multiput(0,60)(2,0){80}{\line(0,1){.1}}
\multiput(0,70)(2,0){80}{\line(0,1){.1}}
\multiput(0,80)(2,0){80}{\line(0,1){.1}}
\multiput(0,90)(2,0){80}{\line(0,1){.1}}
\multiput(0,100)(2,0){80}{\line(0,1){.1}}
\multiput(10,0)(0,2){50}{\line(1,0){.1}}
\multiput(20,0)(0,2){50}{\line(1,0){.1}}
\multiput(30,0)(0,2){50}{\line(1,0){.1}}
\multiput(40,0)(0,2){50}{\line(1,0){.1}}
\multiput(50,0)(0,2){50}{\line(1,0){.1}}
\multiput(60,0)(0,2){50}{\line(1,0){.1}}
\multiput(70,0)(0,2){50}{\line(1,0){.1}}
\multiput(80,0)(0,2){50}{\line(1,0){.1}}
\multiput(90,0)(0,2){50}{\line(1,0){.1}}
\multiput(100,0)(0,2){50}{\line(1,0){.1}}
\multiput(110,0)(0,2){50}{\line(1,0){.1}}
\multiput(120,0)(0,2){50}{\line(1,0){.1}}
\multiput(130,0)(0,2){50}{\line(1,0){.1}}
\multiput(140,0)(0,2){50}{\line(1,0){.1}}
\multiput(150,0)(0,2){50}{\line(1,0){.1}}
\multiput(160,0)(0,2){50}{\line(1,0){.1}}
\put(0,0){\tableau{{}}}
\put(10,30){\tableau{{}\\{}\\{}\\{}}}
\put(30,40){\tableau{{}\\{}\\{}\\{}\\{}}}
\put(40,80){\tableau{{}\\{}\\{}\\{}\\{}\\{}\\{}\\{}\\{}}}
\put(70,30){\tableau{{}\\{}\\{}\\ {}}}
\put(80,20){\tableau{\\{}\\{}}} \put(110,10){\tableau{{}\\{}}}
\put(120,0){\tableau{{}}}\put(130,10){\tableau{{}\\{}}}
\multiput(3,72)(0,0){1}{$\circ$}
\multiput(13,62)(0,0){1}{$\ast$}
\multiput(23,52)(0,0){1}{$\ast$}
 \multiput(33,42)(0,0){1}{$\bullet$}
\multiput(43,42)(0,0){1}{$\bullet$}
\multiput(53,42)(0,0){1}{$\circ$}
\multiput(63,32)(0,0){1}{$\ast$}
\multiput(73,22)(0,0){1}{$\bullet$}
\multiput(83,22)(0,0){1}{$\ast$}
\multiput(93,12)(0,0){1}{$\circ$}
\multiput(103,2)(0,0){1}{$\ast$}
\multiput(113,-8)(0,0){1}{$\bullet$}
\end{picture}
\hspace{1in}
\begin{picture}(100,100)
\multiput(0,0)(0,0){1}{\line(0,1){100}}
\multiput(0,0)(0,0){1}{\line(1,0){160}}
\multiput(0,10)(2,0){80}{\line(0,1){.1}}
\multiput(0,20)(2,0){80}{\line(0,1){.1}}
\multiput(0,30)(2,0){80}{\line(0,1){.1}}
\multiput(0,40)(2,0){80}{\line(0,1){.1}}
\multiput(0,50)(2,0){80}{\line(0,1){.1}}
\multiput(0,60)(2,0){80}{\line(0,1){.1}}
\multiput(0,70)(2,0){80}{\line(0,1){.1}}
\multiput(0,80)(2,0){80}{\line(0,1){.1}}
\multiput(0,90)(2,0){80}{\line(0,1){.1}}
\multiput(0,100)(2,0){80}{\line(0,1){.1}}
\multiput(10,0)(0,2){50}{\line(1,0){.1}}
\multiput(20,0)(0,2){50}{\line(1,0){.1}}
\multiput(30,0)(0,2){50}{\line(1,0){.1}}
\multiput(40,0)(0,2){50}{\line(1,0){.1}}
\multiput(50,0)(0,2){50}{\line(1,0){.1}}
\multiput(60,0)(0,2){50}{\line(1,0){.1}}
\multiput(70,0)(0,2){50}{\line(1,0){.1}}
\multiput(80,0)(0,2){50}{\line(1,0){.1}}
\multiput(90,0)(0,2){50}{\line(1,0){.1}}
\multiput(100,0)(0,2){50}{\line(1,0){.1}}
\multiput(110,0)(0,2){50}{\line(1,0){.1}}
\multiput(120,0)(0,2){50}{\line(1,0){.1}}
\multiput(130,0)(0,2){50}{\line(1,0){.1}}
\multiput(140,0)(0,2){50}{\line(1,0){.1}}
\multiput(150,0)(0,2){50}{\line(1,0){.1}}
\multiput(160,0)(0,2){50}{\line(1,0){.1}}
\put(20,40){\tableau{{}\\{}\\{}\\{}\\{}}}
\put(30,50){\tableau{{}\\{}\\{}\\{}\\{}\\{}}}
\put(40,80){\tableau{{}\\{}\\{}\\{}\\{}\\{}\\{}\\{}\\{}}}
\put(70,30){\tableau{{}\\{}\\{}\\ {}}}
\put(80,20){\tableau{\\{}\\{}}}
\put(100,30){\tableau{{}\\{}\\{}\\{}}}
\put(110,10){\tableau{{}\\{}}} \put(140,10){\tableau{{}\\{}}}
\multiput(3,72)(0,0){1}{$\circ$}
\multiput(13,62)(0,0){1}{$\circ$}
 \multiput(23,52)(0,0){1}{$\ast$}
\multiput(33,42)(0,0){1}{$\bullet$}
\multiput(43,42)(0,0){1}{$\bullet$}
\multiput(53,42)(0,0){1}{$\circ$}
\multiput(63,32)(0,0){1}{$\ast$}
\multiput(73,22)(0,0){1}{$\bullet$}
 \multiput(83,22)(0,0){1}{$\ast$}
\multiput(93,12)(0,0){1}{$\ast$}
\multiput(103,2)(0,0){1}{$\bullet$}
\multiput(113,2)(0,0){1}{$\bullet$}
\multiput(123,2)(0,0){1}{$\ast$}
\multiput(133,-8)(0,0){1}{$\bullet$}
\end{picture}
\end{center}
\caption{Schubert paths of $8$ in two different tower diagrams}\label{fig:sp8}
\end{figure}

The next result is the fundamental property of the Schubert path of $k$ in $\mathcal T$. It relates sliding of a
transposition and the $k$-Bruhat order.

\begin{lem}\label{lem:funda}
Let $\mathcal T$ be a tower diagram with $\omega = \omega_\mathcal T$, and let $t_{i,j}$ be a transposition such that
$l(\omega t_{i,j}) = l(\omega) +1$, and let $e_0, \ldots, e_r$ be the new cells added to the tower with $\omega$-index
$i$. Then the following statements are equivalent.
\begin{enumerate}
\item $i\le k< j$.
\item One of the cells $e_0, \ldots, e_r$ lies in the Schubert path of $k$ in $\mathcal T$.
\end{enumerate}
Moreover, in this case, the cell from the Schubert path is critical.
\end{lem}

\begin{proof}
Let $\mathcal U = \mathcal T_{\omega t_{i,j}}$. Notice that by Corollary \ref{cor: inc-dec}, in $\mathcal U$, the cell
$e_0$ has flight number $i$ and the bottom empty cell $c$ over $e_r$ has flight number $j$. Now, first, suppose that
$e_t$ lies in the Schubert path of $k$ in $\mathcal T$ for some $1\le t\le r$. By the definition of Schubert path,
$\fn(\mathcal U, e_t)\le k$ and any other cell over $e_t$ has flight number at least $k+1$. Also by the definition of
flight numbers, we have
$\fn(\mathcal U, e_0) < \fn(\mathcal U, e_t)<\fn(\mathcal U, c)$. Thus, we get $i\le k<j$, as required.

Conversely, suppose $i\le k<j$. Since the length of the permutation is increased by one, we have that the
size of $\Tt$ is increased by one after the sliding $\fh_{i,j-1}\searrow \Tt$.

Since $i\le k<j$, let $c^\prime$ be the cell in the foot of the hook $\fh_{i,j}$ having slide $k$, when the sliding begins.
Then, it is clear from the definition of the Schubert path that, initially, the sliding of the cell $c^\prime$ traces the
Schubert path of $k$. Now it is sufficient to show that the Schubert path is traced by one of the cells of the foot.
By Theorem \ref{thm:plus1}, this case no broken(-) pass occurs. Thus, if the heel $h$ of the hook passes a tower
directly, then $c^\prime$ passes it directly, too, and after such a pass, the leg still has a cell from the Schubert path,
namely $c^\prime$. Or else, if the hook makes a zigzag pass through a tower, then the diagonal that the heel is sliding
is shifted up by one cell, but from its definition, there is also a shift in the Schubert path of $k$. Thus the leg still contains
a cell from it. Finally, if the hook makes a shrunken pass at a tower, that is, the heel $h$ makes a zigzag pass, but one
of the cells in the foot does not, and hence is deleted from the foot and the leg. If the cell deleted from the foot is on the
right of $c^\prime$, then we are in the previous case, since $c^\prime$ makes a zigzag, too. If the deleted cell is on the
left of $c^\prime$, then we are in the first case, since $c^\prime$ makes a direct pass. Finally, if the deleted
cell is the cell $c^\prime$, then from the definition of the shrunken pass, the cell just on the left of $c^\prime$ is shifted
up to replace $c^\prime$. In this case, the Schubert path is traced by this cell, as seen from its definition. Hence, we
again have a cell from the Schubert path which lies in the heel or in the foot.

Thus, in any case, the foot contains a cell from the Schubert path, and since the hook is symmetric, the leg contains a
cell from the Schubert path as well. Therefore, when the leg lands on a tower, forming the cells $e_0, \ldots, e_r$, one of
these cells must be contained in the Schubert path of $k$ in $\mathcal T$.

Now suppose $i\le k<j$ and $e_0, \ldots, e_r$ be as above. Suppose $e_t$ is contained in the Schubert path of $k$
in $\mathcal T$. We need to show that $e_t$ is a critical cell. By Theorem \ref{thm:plus1}, there are two cases.

In the first case, sliding of the hook must contain enough number of shrunken passes so that the hook is
reduced to a single cell which adds, at the end, a new cell to the diagram. Hence it must be located on top of a tower
of $\mathcal T$ and hence it must be an essential cell, and hence critical.

In the second case, that is, if the sliding sequence contains  a broken(+) pass at some tower $\mathcal T_\alpha$ of
$\Tt$, the heel and the leg  of the hook, which arrived to that tower through direct, zigzag or shrunken pass, is placed at
top of a tower, say $\mathcal T_\alpha$. Moreover its foot being either a single cell or a row permutation deletes as
many cells as its size from the top of a tower on the right of $\mathcal T_\alpha$. Let $c$ be the cell in $\Tt_\alpha$
contained in the Schubert path of $k$ on $\Tt$. We need to justify that it is critical.

To prove this claim, we should show that the next  bullet $d$ on the right of $c$ has enough number of cells below.
But, this is trivial, since we know that the number of cells in $\mathcal T$ should be increased by one, and hence the
foot should delete as many cells as its length. Thus, either the foot deletes cells from the tower of $d$ or it passes that
tower with zigzag pass. Hence in both cases $d$ must have enough cells, and the claim is proved.
\end{proof}

Now we introduce an algorithm which determines the set $k\cdot\mathcal T$. By the above lemma, this set can be
divided into disjoint subsets according to critical cells. More precisely, to construct the set $k\cdot\mathcal T$, we can consider
each
critical cell separately and determine all possible additions to the tower containing this cell.  To introduce the algorithm,
for each critical cell $c\in\crit[k,\mathcal T]$, define $c\cdot\mathcal T$ as the set of tower diagrams in
$k\cdot\mathcal T$ where the new cell is added to the tower containing $c$ so that we have
$$k\cdot \Tt = \bigcup_{c\in \crit[k,\Tt]} c\cdot \Tt. $$
Notice that there might exists critical cells where the set $c\cdot\mathcal T$ is empty. Now the set $c\cdot \mathcal T$
can be determined via the following algorithm. Suppose $c$ lies above the tower $\Tt_l$ of $\Tt$ and $t\geq 0$ is the
number of empty cells lying below $c$. Denote by $e_0,\ldots, e_t, e_{t+1},\ldots$ those empty cells lying above $\Tt_l$
from bottom to top, where $e_t=c$.  Also let $s\geq 0$ be the number of cells lying in $\Tt$ and above the first bullet
which is right next to $c$.

\smallskip
\noindent\textsc{Monk's Algorithm}.

\noindent\textsc{Step 1}. Put $r=0$ and let $\widetilde{c\cdot\mathcal T}$ be the empty set.

\noindent\textsc{Step 2}. Adjoin the cells  $e_0, e_1, \ldots, e_i,\ldots,e_{i+r}$ to $\Tt$ and denote the new tower
diagram obtained in this way by $\mathcal U$.

\noindent\textsc{Step 3}.  Slide the cells $e_{i+r}, e_{i+r-1}, \ldots, e_1$  to $\Tt_{>l}=\mathcal U_{>l}$, in this order.
If each of these cells deletes a cell, then all these deleted cells must be the top $i$ cells of a unique tower on the right
hand side of $\Tt_l$. In this case, denote by $\mathcal U'$  the diagram obtained from $\mathcal U$ by removing
these cells and append $\mathcal U'$ to $\widetilde{c\cdot \Tt}$.

\noindent\textsc{Step 4}. Put $r=r+1$. If $r\le s+1$, back to Step 2. Otherwise terminate the algorithm.

This algorithm produces a set $\widetilde{c\cdot \Tt}$ and the equality $\widetilde{c\cdot \Tt} = c\cdot \Tt$ follows from
Lemma \ref{lem:funda}. Indeed, let $\mathcal U \in\Tt\wedge k$. Then $\omega_\mathcal U=\omega_\mathcal T \cdot
t_{i,j}$ for some $i\leq k<j$. Since the length of the permutation is increased by one, we have that the size of $\Tt$ is
increased by one after the sliding $\fh_{i, j-1}\searrow\Tt$. Thus by Lemma \ref{lem:funda}, $\mathcal U\in
\widetilde{c\cdot \Tt}$. Conversely, if $\mathcal U \in \widetilde{k\cdot\Tt}$, then Lemma \ref{lem:funda} implies that
$\omega_\mathcal U=\omega_\mathcal T \cdot t_{i,j}$ for some $i\leq k<j$ and hence $\mathcal U\in k\cdot
\mathcal T$. Hence we have completed the proof of the following theorem.

\begin{thm}[Monk's Rule]\label{mainthm}
Let $\omega$ be a permutation and let $\mathfrak S_\omega$ denote the Schubert polynomial of $\omega$. Also let
$s_k$ be an adjacent transposition. Then
\[
\mathfrak{S}_\omega \cdot \mathfrak S_{s_k} = \sum_{\mathcal T\in k\cdot\mathcal T_\omega}
\mathfrak S_{\omega_{\mathcal T}}
\]
\end{thm}

\begin{ex}\label{ex:ess2} Let $w=[1, 2, 5, 6, 4, 10, 3, 8, 7, 11, 9]$ be a permutation in one line notation. The Schubert
path of $5$ together with critical cells $b,c$ and $d$ are shown below.  Here the first bullet to
the right of $b$ has only one cell above. Therefore the algorithm works for $r=0,1$ but  in case of $r=0$, sliding the cells
$e_3=b, e_2$ and $e_1$ actually produce three new cells in the diagram.
Therefore only for $r=1$ we have a diagram which is obtained by adding $e_4,e_3=b, e_2,e_1,e_0$ and deleting the
four cells from the right next tower.
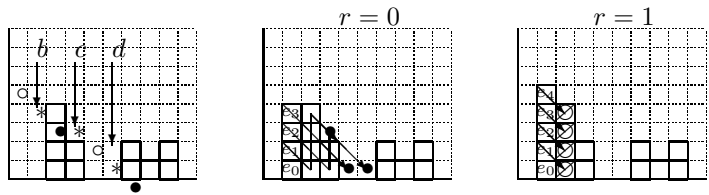
\begin{figure}[h]
\vspace{.3in}
\setlength{\unitlength}{0.25mm}
\begin{picture}(80,60)
\multiput(0,0)(0,0){1}{\line(0,1){80}}
\multiput(0,0)(0,0){1}{\line(1,0){100}}
\multiput(0,10)(2,0){50}{\line(0,1){.1}}
\multiput(0,20)(2,0){50}{\line(0,1){.1}}
\multiput(0,30)(2,0){50}{\line(0,1){.1}}
\multiput(0,40)(2,0){50}{\line(0,1){.1}}
\multiput(0,50)(2,0){50}{\line(0,1){.1}}
\multiput(0,60)(2,0){50}{\line(0,1){.1}}
\multiput(0,70)(2,0){50}{\line(0,1){.1}}
\multiput(0,80)(2,0){50}{\line(0,1){.1}}
\multiput(10,0)(0,2){40}{\line(1,0){.1}}
\multiput(20,0)(0,2){40}{\line(1,0){.1}}
\multiput(30,0)(0,2){40}{\line(1,0){.1}}
\multiput(40,0)(0,2){40}{\line(1,0){.1}}
\multiput(50,0)(0,2){40}{\line(1,0){.1}}
\multiput(60,0)(0,2){40}{\line(1,0){.1}}
\multiput(70,0)(0,2){40}{\line(1,0){.1}}
\multiput(80,0)(0,2){40}{\line(1,0){.1}}
\multiput(90,0)(0,2){40}{\line(1,0){.1}}
\multiput(100,0)(0,2){40}{\line(1,0){.1}}
\put(20,30){\tableau{{}\\{}\\{}\\ {}}}
\put(30,20){\tableau{\\{}\\{}}} \put(60,10){\tableau{{}\\{}}}
\put(70,0){\tableau{{}}} \put(80,10){\tableau{{}\\{}}}
\multiput(4,42)(0,0){1}{$\circ$}
\multiput(14,32)(0,0){1}{$\ast$}
\multiput(24,22)(0,0){1}{$\bullet$}
\multiput(34,22)(0,0){1}{$\ast$}
\multiput(44,12)(0,0){1}{$\circ$}
\multiput(54,2)(0,0){1}{$\ast$}
\multiput(64,-8)(0,0){1}{$\bullet$}
\multiput(15,62)(0,0){1}{\vector(0,-1){24}}
\multiput(35,62)(0,0){1}{\vector(0,-1){34}}
\multiput(55,62)(0,0){1}{\vector(0,-1){44}}
\multiput(15,65)(0,0){1}{$b$}
\multiput(35,65)(0,0){1}{$c$}
\multiput(55,65)(0,0){1}{$d$}
\end{picture}\hspace{.5in}
\begin{picture}(80,60)
 \multiput(0,0)(0,0){1}{\line(0,1){80}}
 \multiput(0,0)(0,0){1}{\line(1,0){100}}
 \multiput(0,10)(2,0){50}{\line(0,1){.1}}
 \multiput(0,20)(2,0){50}{\line(0,1){.1}}
 \multiput(0,30)(2,0){50}{\line(0,1){.1}}
 \multiput(0,40)(2,0){50}{\line(0,1){.1}}
 \multiput(0,50)(2,0){50}{\line(0,1){.1}}
 \multiput(0,60)(2,0){50}{\line(0,1){.1}}
 \multiput(0,70)(2,0){50}{\line(0,1){.1}}
 \multiput(0,80)(2,0){50}{\line(0,1){.1}}
 \multiput(10,0)(0,2){40}{\line(1,0){.1}}
 \multiput(20,0)(0,2){40}{\line(1,0){.1}}
 \multiput(30,0)(0,2){40}{\line(1,0){.1}}
 \multiput(40,0)(0,2){40}{\line(1,0){.1}}
 \multiput(50,0)(0,2){40}{\line(1,0){.1}}
 \multiput(60,0)(0,2){40}{\line(1,0){.1}}
 \multiput(70,0)(0,2){40}{\line(1,0){.1}}
 \multiput(80,0)(0,2){40}{\line(1,0){.1}}
 \multiput(90,0)(0,2){40}{\line(1,0){.1}}
 \multiput(100,0)(0,2){40}{\line(1,0){.1}}
\put(10,30){\tableau{{_{e_3}}\\{_{e_2}}\\ {_{e_1}}\\{_{e_0}}}}
 \put(20,30){\tableau{{}\\{}\\{}\\ {}}}
 \put(30,10){\tableau{{}\\{}}} \put(60,10){\tableau{{}\\{}}}
 \put(70,0){\tableau{{}}} \put(80,10){\tableau{{}\\{}}}
  \multiput(10,40)(0,0){1}{\line(1,-1){25}}
   \multiput(10,30)(0,0){1}{\line(1,-1){25}}
    \multiput(10,20)(0,0){1}{\line(1,-1){15}}
    \multiput(25,5)(0,0){1}{\line(0,1){30}}
    \multiput(35,5)(0,0){1}{\line(0,1){20}}
     \multiput(25,35)(0,0){1}{\vector(1,-1){15}}\put(32,22){$\bullet$}\put(42,2){$\bullet$}\put(52,2){$\bullet$}
    \multiput(30,30)(0,0){1}{\vector(1,-1){25}}
    \multiput(30,20)(0,0){1}{\vector(1,-1){15}}
\put(40,82){$r=0$}
\end{picture}\hspace{.5in}
\begin{picture}(80,60)
 \multiput(0,0)(0,0){1}{\line(0,1){80}}
 \multiput(0,0)(0,0){1}{\line(1,0){100}}
 \multiput(0,10)(2,0){50}{\line(0,1){.1}}
 \multiput(0,20)(2,0){50}{\line(0,1){.1}}
 \multiput(0,30)(2,0){50}{\line(0,1){.1}}
 \multiput(0,40)(2,0){50}{\line(0,1){.1}}
 \multiput(0,50)(2,0){50}{\line(0,1){.1}}
 \multiput(0,60)(2,0){50}{\line(0,1){.1}}
 \multiput(0,70)(2,0){50}{\line(0,1){.1}}
 \multiput(0,80)(2,0){50}{\line(0,1){.1}}
 \multiput(10,0)(0,2){40}{\line(1,0){.1}}
 \multiput(20,0)(0,2){40}{\line(1,0){.1}}
 \multiput(30,0)(0,2){40}{\line(1,0){.1}}
 \multiput(40,0)(0,2){40}{\line(1,0){.1}}
 \multiput(50,0)(0,2){40}{\line(1,0){.1}}
 \multiput(60,0)(0,2){40}{\line(1,0){.1}}
 \multiput(70,0)(0,2){40}{\line(1,0){.1}}
 \multiput(80,0)(0,2){40}{\line(1,0){.1}}
 \multiput(90,0)(0,2){40}{\line(1,0){.1}}
 \multiput(100,0)(0,2){40}{\line(1,0){.1}}
 \put(10,40){\tableau{{_{e_4}}\\{_{e_3}}\\{_{e_2}}\\ {_{e_1}}\\{_{e_0}}}}
  \put(20,30){\tableau{{\varnothing}\\{\varnothing}\\{\varnothing}\\ {\varnothing}}}
 \put(30,10){\tableau{{}\\{}}} \put(60,10){\tableau{{}\\{}}}
 \put(70,0){\tableau{{}}} \put(80,10){\tableau{{}\\{}}}
 \multiput(10,50)(0,0){1}{\vector(1,-1){15}}
  \multiput(10,40)(0,0){1}{\vector(1,-1){15}}
   \multiput(10,30)(0,0){1}{\vector(1,-1){15}}
    \multiput(10,20)(0,0){1}{\vector(1,-1){15}}
\put(40,82){$r=1$}
\end{picture}
\end{figure}

For the critical cell $c$, observe that the first bullet to the right of $c$ is below the $x$-axis and $\Tt$ has two cells lying above this bullet. Therefore the algorithm works for $r=0,1,2$ and produces a new diagram in each case, as explained below.
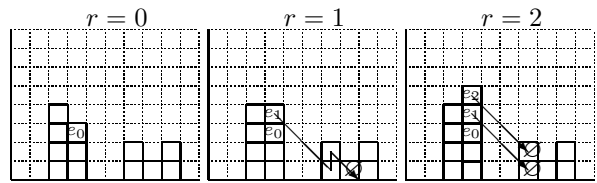
\begin{figure}[h]
\vspace{.3in}
\setlength{\unitlength}{0.25mm}
 \begin{picture}(80,60)
 \multiput(0,0)(0,0){1}{\line(0,1){80}}
 \multiput(0,0)(0,0){1}{\line(1,0){100}}
 \multiput(0,10)(2,0){50}{\line(0,1){.1}}
 \multiput(0,20)(2,0){50}{\line(0,1){.1}}
 \multiput(0,30)(2,0){50}{\line(0,1){.1}}
 \multiput(0,40)(2,0){50}{\line(0,1){.1}}
 \multiput(0,50)(2,0){50}{\line(0,1){.1}}
 \multiput(0,60)(2,0){50}{\line(0,1){.1}}
 \multiput(0,70)(2,0){50}{\line(0,1){.1}}
 \multiput(0,80)(2,0){50}{\line(0,1){.1}}
 \multiput(10,0)(0,2){40}{\line(1,0){.1}}
 \multiput(20,0)(0,2){40}{\line(1,0){.1}}
 \multiput(30,0)(0,2){40}{\line(1,0){.1}}
 \multiput(40,0)(0,2){40}{\line(1,0){.1}}
 \multiput(50,0)(0,2){40}{\line(1,0){.1}}
 \multiput(60,0)(0,2){40}{\line(1,0){.1}}
 \multiput(70,0)(0,2){40}{\line(1,0){.1}}
 \multiput(80,0)(0,2){40}{\line(1,0){.1}}
 \multiput(90,0)(0,2){40}{\line(1,0){.1}}
 \multiput(100,0)(0,2){40}{\line(1,0){.1}}
 \put(20,30){\tableau{{}\\{}\\{}\\ {}}}
 \put(30,20){\tableau{{_{e_0}}\\{}\\{}}} \put(60,10){\tableau{{}\\{}}}
 \put(70,0){\tableau{{}}} \put(80,10){\tableau{{}\\{}}}
\put(40,82){$r=0$}
\end{picture}\hspace{.2in}
\begin{picture}(80,60)
\multiput(0,0)(0,0){1}{\line(0,1){80}}
\multiput(0,0)(0,0){1}{\line(1,0){100}}
\multiput(0,10)(2,0){50}{\line(0,1){.1}}
\multiput(0,20)(2,0){50}{\line(0,1){.1}}
\multiput(0,30)(2,0){50}{\line(0,1){.1}}
\multiput(0,40)(2,0){50}{\line(0,1){.1}}
\multiput(0,50)(2,0){50}{\line(0,1){.1}}
\multiput(0,60)(2,0){50}{\line(0,1){.1}}
\multiput(0,70)(2,0){50}{\line(0,1){.1}}
\multiput(0,80)(2,0){50}{\line(0,1){.1}}
\multiput(10,0)(0,2){40}{\line(1,0){.1}}
\multiput(20,0)(0,2){40}{\line(1,0){.1}}
\multiput(30,0)(0,2){40}{\line(1,0){.1}}
\multiput(40,0)(0,2){40}{\line(1,0){.1}}
\multiput(50,0)(0,2){40}{\line(1,0){.1}}
\multiput(60,0)(0,2){40}{\line(1,0){.1}}
\multiput(70,0)(0,2){40}{\line(1,0){.1}}
\multiput(80,0)(0,2){40}{\line(1,0){.1}}
\multiput(90,0)(0,2){40}{\line(1,0){.1}}
\multiput(100,0)(0,2){40}{\line(1,0){.1}}
\put(20,30){\tableau{{}\\{}\\{}\\ {}}}
\put(30,30){\tableau{{_{e_1}}\\{_{e_0}}\\{}\\{}}}
\put(60,10){\tableau{{}\\{}}}\put(70,0){\tableau{{}}}
\put(72,2){$\varnothing$}\put(80,10){\tableau{{}\\{}}}
\multiput(35,35)(0,0){1}{\line(1,-1){30}}
\multiput(65,5)(0,0){1}{\line(0,1){10}}
\multiput(65,15)(0,0){1}{\vector(1,-1){15}}
\put(40,82){$r=1$}
\end{picture}\hspace{.2in}
\begin{picture}(80,60,120)
\multiput(0,0)(0,0){1}{\line(0,1){80}}
\multiput(0,0)(0,0){1}{\line(1,0){100}}
\multiput(0,10)(2,0){50}{\line(0,1){.1}}
\multiput(0,20)(2,0){50}{\line(0,1){.1}}
\multiput(0,30)(2,0){50}{\line(0,1){.1}}
\multiput(0,40)(2,0){50}{\line(0,1){.1}}
\multiput(0,50)(2,0){50}{\line(0,1){.1}}
\multiput(0,60)(2,0){50}{\line(0,1){.1}}
\multiput(0,70)(2,0){50}{\line(0,1){.1}}
\multiput(0,80)(2,0){50}{\line(0,1){.1}}
\multiput(10,0)(0,2){40}{\line(1,0){.1}}
\multiput(20,0)(0,2){40}{\line(1,0){.1}}
\multiput(30,0)(0,2){40}{\line(1,0){.1}}
\multiput(40,0)(0,2){40}{\line(1,0){.1}}
\multiput(50,0)(0,2){40}{\line(1,0){.1}}
\multiput(60,0)(0,2){40}{\line(1,0){.1}}
\multiput(70,0)(0,2){40}{\line(1,0){.1}}
\multiput(80,0)(0,2){40}{\line(1,0){.1}}
\multiput(90,0)(0,2){40}{\line(1,0){.1}}
\multiput(100,0)(0,2){40}{\line(1,0){.1}}
\put(20,30){\tableau{{}\\{}\\{}\\ {}}}
\put(30,40){\tableau{{_{e_2}}\\{_{e_1}}\\{_{e_0}}\\{}\\{}}}\put(60,10){\tableau{{}\\{}}}
\put(62,2){$\varnothing$}  \put(62,12){$\varnothing$}
\put(70,0){\tableau{{}}} \put(80,10){\tableau{{}\\{}}}
\multiput(35,45)(0,0){1}{\vector(1,-1){30}}
\multiput(35,35)(0,0){1}{\vector(1,-1){30}}
\put(40,82){$r=2$}
\end{picture}
\end{figure}

Now for the critical cell $d$ still the algorithm works for $r=0,1,2$ and again it produces a new diagram in each case.

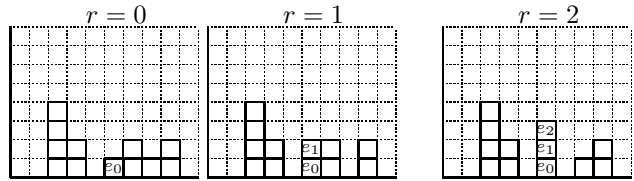
\begin{figure}[h]
\setlength{\unitlength}{0.25mm}
\begin{picture}(80,60)
 \multiput(0,0)(0,0){1}{\line(0,1){80}}
 \multiput(0,0)(0,0){1}{\line(1,0){100}}
 \multiput(0,10)(2,0){50}{\line(0,1){.1}}
 \multiput(0,20)(2,0){50}{\line(0,1){.1}}
 \multiput(0,30)(2,0){50}{\line(0,1){.1}}
 \multiput(0,40)(2,0){50}{\line(0,1){.1}}
 \multiput(0,50)(2,0){50}{\line(0,1){.1}}
 \multiput(0,60)(2,0){50}{\line(0,1){.1}}
 \multiput(0,70)(2,0){50}{\line(0,1){.1}}
 \multiput(0,80)(2,0){50}{\line(0,1){.1}}
 \multiput(10,0)(0,2){40}{\line(1,0){.1}}
 \multiput(20,0)(0,2){40}{\line(1,0){.1}}
 \multiput(30,0)(0,2){40}{\line(1,0){.1}}
 \multiput(40,0)(0,2){40}{\line(1,0){.1}}
 \multiput(50,0)(0,2){40}{\line(1,0){.1}}
 \multiput(60,0)(0,2){40}{\line(1,0){.1}}
 \multiput(70,0)(0,2){40}{\line(1,0){.1}}
 \multiput(80,0)(0,2){40}{\line(1,0){.1}}
 \multiput(90,0)(0,2){40}{\line(1,0){.1}}
 \multiput(100,0)(0,2){40}{\line(1,0){.1}}
 \put(20,30){\tableau{{}\\{}\\{}\\ {}}}
 \put(30,10){\tableau{{}\\{}}} \put(60,10){\tableau{{}\\{}}}\put(50,0){\tableau{{_{e_0}}}}
 \put(70,0){\tableau{{}}} \put(80,10){\tableau{{}\\{}}}
\put(40,82){$r=0$}
\end{picture}\hspace{.2in}
\begin{picture}(100,100)
 \multiput(0,0)(0,0){1}{\line(0,1){80}}
 \multiput(0,0)(0,0){1}{\line(1,0){100}}
 \multiput(0,10)(2,0){50}{\line(0,1){.1}}
 \multiput(0,20)(2,0){50}{\line(0,1){.1}}
 \multiput(0,30)(2,0){50}{\line(0,1){.1}}
 \multiput(0,40)(2,0){50}{\line(0,1){.1}}
 \multiput(0,50)(2,0){50}{\line(0,1){.1}}
 \multiput(0,60)(2,0){50}{\line(0,1){.1}}
 \multiput(0,70)(2,0){50}{\line(0,1){.1}}
 \multiput(0,80)(2,0){50}{\line(0,1){.1}}
 \multiput(10,0)(0,2){40}{\line(1,0){.1}}
 \multiput(20,0)(0,2){40}{\line(1,0){.1}}
 \multiput(30,0)(0,2){40}{\line(1,0){.1}}
 \multiput(40,0)(0,2){40}{\line(1,0){.1}}
 \multiput(50,0)(0,2){40}{\line(1,0){.1}}
 \multiput(60,0)(0,2){40}{\line(1,0){.1}}
 \multiput(70,0)(0,2){40}{\line(1,0){.1}}
 \multiput(80,0)(0,2){40}{\line(1,0){.1}}
 \multiput(90,0)(0,2){40}{\line(1,0){.1}}
 \multiput(100,0)(0,2){40}{\line(1,0){.1}}
 \put(20,30){\tableau{{}\\{}\\{}\\ {}}}
 \put(30,10){\tableau{{}\\{}}} \put(60,10){\tableau{{}\\{}}}\put(50,10){\tableau{{_{e_1}}\\{_{e_0}}}}
 \put(80,10){\tableau{{}\\{}}}
\put(40,82){$r=1$}
\end{picture}\hspace{.2in}
\begin{picture}(100,100)
 \multiput(0,0)(0,0){1}{\line(0,1){80}}
 \multiput(0,0)(0,0){1}{\line(1,0){100}}
 \multiput(0,10)(2,0){50}{\line(0,1){.1}}
 \multiput(0,20)(2,0){50}{\line(0,1){.1}}
 \multiput(0,30)(2,0){50}{\line(0,1){.1}}
 \multiput(0,40)(2,0){50}{\line(0,1){.1}}
 \multiput(0,50)(2,0){50}{\line(0,1){.1}}
 \multiput(0,60)(2,0){50}{\line(0,1){.1}}
 \multiput(0,70)(2,0){50}{\line(0,1){.1}}
 \multiput(0,80)(2,0){50}{\line(0,1){.1}}
 \multiput(10,0)(0,2){40}{\line(1,0){.1}}
 \multiput(20,0)(0,2){40}{\line(1,0){.1}}
 \multiput(30,0)(0,2){40}{\line(1,0){.1}}
 \multiput(40,0)(0,2){40}{\line(1,0){.1}}
 \multiput(50,0)(0,2){40}{\line(1,0){.1}}
 \multiput(60,0)(0,2){40}{\line(1,0){.1}}
 \multiput(70,0)(0,2){40}{\line(1,0){.1}}
 \multiput(80,0)(0,2){40}{\line(1,0){.1}}
 \multiput(90,0)(0,2){40}{\line(1,0){.1}}
 \multiput(100,0)(0,2){40}{\line(1,0){.1}}
 \put(20,30){\tableau{{}\\{}\\{}\\ {}}}
 \put(30,10){\tableau{{}\\{}}} \put(50,20){\tableau{{_{e_2}}\\{_{e_1}}\\{_{e_0}}}}
 \put(70,0){\tableau{{}}} \put(80,10){\tableau{{}\\{}}}
\put(40,82){$r=2$}
\end{picture}
\end{figure}

Hence reading the permutations (in one line notation) from the above pictures, we get
\begin{eqnarray*}
5\cdot \Tt_{[1, 2, 5, 6, 4, 10, 3, 8, 7, 11, 9]}&=&\{ \Tt_{[1, 2, 5, 6, 10, 4, 3, 8, 7, 11, 9]}
, \Tt_{[1, 2, 5, 6, 8, 10, 3, 4, 7, 11, 9]}, \Tt_{[1, 2, 5, 6, 7, 10, 3, 8, 4, 11, 9]}, \\&& \Tt_{
[1, 3, 5, 6, 4, 10, 2, 8, 7, 11, 9]}, \Tt_{[1, 2, 5, 10, 4, 6, 3, 8, 7, 11, 9]}, \Tt_{[1, 2, 5, 8, 4, 10, 3, 6,
7, 11, 9]}, \\&&\Tt_{[1, 2, 5, 7, 4, 10, 3, 8, 6, 11, 9]}\}
\end{eqnarray*}

\end{ex}

\section{Pieri's Rule}\label{sec:pieri}

In this section, we introduce another algorithm which describes Pieri's rule in terms of tower diagrams. To begin with, fix $k\in \Nn$. Let $\omega\in S_n$ be a permutation and $a_1,b_1, a_2, b_2, \ldots, a_m, b_m$ be a sequence of
positive integers such that for each $1\le i \le m$,

\begin{enumerate}
\item[(1)] $a_i\le k < b_i$,
\item[(2)] $l(\omega t_{a_1,b_1}\ldots t_{a_i, b_i}) = l(\omega t_{a_1, b_1}\ldots t_{a_{i-1}, b_{i-1}}) +1$
\end{enumerate}

For each $1\le i\le m$, set $\omega^{(i)} = \omega t_{a_1,b_1}\ldots t_{a_i, b_i}$. We sometimes write
$\omega' = \omega^{(m)}$. Then the sequence of pairs
$(a,b): (a_1, b_1), (a_2, b_2),\ldots, (a_m, b_m)$ is called a \textit{saturated $k$-Bruhat chain} from $\omega$ to
$\omega'$, written $(a,b): \omega\to\omega'$.

A saturated $k$-Bruhat chain $(a,b): \omega\to \omega'$ is called a \textit{Sottile $r[k,m]$-sequence} (or \textit{Sottile r-
sequence}) if $|\{ b_1,\ldots, b_m \}| = m$. Similarly, a saturated  $k$-Bruhat chain $(a,b): \omega\to \omega'$ is called a
\textit{Sottile $c[k,m]$-sequence} (or \textit{Sottile c-sequence}) if $|\{ a_1,\ldots, a_m \}| = m$. With this notation, Pieri's rule can be stated as follows.

\begin{thm}[Sottile \cite{F}]
Let $\omega, \nu$ be permutations in $S_n$.
\begin{enumerate}
\item Let $m\in\Nn$ such that $k+m \le n$. Also let $r[k,m]$ denote the $(m+1)$-cycle $(k+m\, k+m-1\, \ldots\, k+1\, k)$.
Then the Littlewood-Richardson coefficient $c_{\omega, r[k,m]}^\nu$ is non-zero if and only if there exists a Sottile
$r[k,m]$-sequence  $(a,b):\omega\to \nu$.
\item Let $m\in\Nn$ such that $m \le k\le  n$. Also let $c[k,m]$ denote the $(m+1)$-cycle
$(k-m+1\, k-m+2\, \ldots\, k\, k+1)$.  Then the Littlewood-Richardson coefficient $c_{\omega, c[k,m]}^\nu$ is non-zero if
and only if there exists a Sottile $c[k,m]$-sequence  $(a,b):\omega\to \nu$.
\end{enumerate}
In both cases, $c_{\omega, r[k,m]}^\nu = 1$.
\end{thm}

Note that, when exists, a Sottile r-sequence (or a Sottile c-sequence) between two permutations is not necessarily unique. Thus, the above theorem is not bijective.
Our aim is to relate the above theorem of Sottile to our Monk's algorithm and state a bijective version using tower diagrams. Our main tool will again be Corollary \ref{cor: inc-dec}.

We describe the algorithm for the first part of Sottile's Theorem. The other one is similar. Given permutations $\omega$ and
$r[k,m]$ as in Sottile's Theorem. We want to
evaluate the product $\mathfrak S_\omega \mathfrak S_{r[k,m]}$ of the corresponding Schubert polynomials by applying Monk's rule
successively. The process goes as follows: We start with the tower diagram $\mathcal T_\omega$, and apply
Monk's algorithm from the previous section with $k$. This produces a set $k\cdot \mathcal T_\omega$ of tower diagrams.
Each diagram in $k\cdot \mathcal T_\omega$ is obtained by modifying only two towers
of $\mathcal T_\omega$ and by appending a new cell $e_0$. In each diagram, we label $e_0$ by $(1, a_1, b_1)$
where $(a_1, b_1)$ is the $\omega$-indexes of the modified towers. Notice that the flight number of $e_0$ is $a_1$. Then
for each diagram $\mathcal T'\in k\cdot \mathcal T_\omega$, we apply Monk's algorithm once more with $k$ to obtain a
new set of tower diagrams, denoted by $k_{(2)}\cdot \mathcal T$. In this case,
the new cells are labeled by $(2, a_2, b_2)$ in the same way. We apply this process $m$-times and obtain a set of
tower diagrams $k_{(m)}\cdot \mathcal T_\omega$. Note that each tower diagram in $k_{(m)}\cdot \mathcal T_\omega$
is partially labeled by the sequence $(1, a_1, b_1), (2, a_2, b_2),\ldots, (m, a_m, b_m)$.

In this
case the sequence $(a_1, b_1),\ldots,(a_m, b_m)$ gives rise to a saturated $k$-Bruhat chain. Moreover, it is easy to see that any saturated $k$-Bruhat chain
arises in this way. In particular, all Sottile r-sequences appear as labels of certain tower diagrams in this set.
Notice that, given a tower diagram $\mathcal T\in k_{(m)}\cdot \mathcal T_\omega$ with labels $(1, a_1, b_1),\ldots, (m, a_m, b_m)$, the sequence $( a_1, b_1),
\ldots, (a_m, b_m)$ is a Sottile r-sequence if and only if during the above process, if a tower $\tau$ is raised at some step $i$, then it is not modified at any later
step. Indeed, by Corollary \ref{cor: inc-dec}, the $\omega^{(i)}$-index of $\tau$ must be $b_i$, and at a later step, say $j$, the modified towers have
$\omega^{(j)}$-indexes $a_j$ and $b_j$. Since, $a_j\neq b_i\neq b_j$, the tower $\tau$ is not modified at this step. Thus, by imposing this as a condition in the
above process of successive applications of Monk's algorithm, it is possible to construct a set of tower diagrams partially labelled by Sottile r-sequences.
However, our aim is to introduce some conditions on the above process so that it only produces one copy of these tower diagrams.

For this aim, let $r[k, m]\cdot \mathcal T_\omega$ be the subset of all tower diagrams $\mathcal T$ in $k_{(m)}\cdot \mathcal T_\omega$ with labels
$(1, a_1, b_1), (2, a_2, b_2),\ldots, (m, a_m, b_m)$ such that
\begin{enumerate}
\item $a_1\ge a_2\ge\ldots\ge a_m$ and
\item no two labelled cells are in the same tower of $\mathcal T$.
\end{enumerate}
With this notation, we can state the main result of this section.
\begin{thm}\label{thm:pieri}
Let $\omega$ be a permutation in $S_n$, $k, m\in\Nn$ such that $k+m \le n$. Also let $r[k,m]$ denote the
$(m+1)$-cycle $(k+m\, k+m-1\, \ldots k+1 \, k)$.
Then
\[
\mathfrak S_\omega \mathfrak S_{r[k,m]} = \sum_{\mathcal T\in r[k,m]\cdot \mathcal T_\omega}
 \mathfrak S_{\omega_\mathcal T}.
\]
\end{thm}

The proof of this theorem consists of several steps. The first step is to show that for any permutation $\nu$, if there is a Sottile
r-sequence $\omega\to \nu$, then there is one with specific properties.

\begin{pro}
With the above notation, suppose $\mathcal T$ is a tower diagram in $k_{(m)}\cdot \mathcal T_\omega$ with labels
$(1,a_1, b_1), (2, a_2, b_2),\ldots, (m, a_m, b_m)$ such that $b_i$'s are distinct. Then we can rearrange
the sequence $(a_1, b_1),\ldots, (a_m, b_m)$ as $(a_1', b_1'),\ldots, (a_m', b_m')$ in such a way that
\begin{enumerate}
\item there is a tower diagram $\mathcal T'\in k_{(m)}\cdot \mathcal T_\omega$ with labels $(1, a_1', b_1'), (2, a_2', b_2'),\ldots,(m, a_m', b_m')$,
\item $\omega_\mathcal T = \omega_{\mathcal T'}$ and
\item $a_1\ge a_2\ge \ldots \ge a_m$.
\end{enumerate}
\end{pro}

\begin{proof}
Suppose $a_1\ge a_2\ge \ldots\ge a_{i-1}$ and $a_i>a_{i-1}$. In this case, since the numbers $a_{i-1}, b_{i-1}, a_i, b_i$
are distinct, we have $t_{a_{i-1}, b_{i-1}}t_{a_{i}, b_{i}} = t_{a_{i}, b_{i}}t_{a_{i-1}, b_{i-1}}$. Clearly, this exchange
respects the lengths of permutations. Thus we can interchange the $i$-th and $(i-1)$-st terms of the sequence and have the
first two properties hold. Now we apply the same procedure until we obtain a sequence $a_1'\ge \ldots\ge a_i'$ and
continue with the next term until we get a sequence as in $(3)$. As the first two properties hold at any step, they also
hold at the end.
\end{proof}

With this proposition, we can consider the set $\widetilde{k_{(m)}\cdot \mathcal T_\omega}$ consisting of the tower
diagrams with labels $(1, a_1, b_1), (2, a_2, b_2),\ldots, (m, a_m, b_m)$ such that the number $a_i$ are non-increasing
and the numbers $b_i$ are distinct. Still we may have some repetitions, that is, we may have the same tower
diagram with different labels. The next two lemmas show that this is not the case. The first one is a technical lemma about flight numbers of the new cells, and the
second one proves the uniqueness of the shapes.
\begin{lem}
Let $\mathcal T\in k_{(m)}\cdot \mathcal T_\omega$ with new cells $c_1, \ldots, c_m$ and labels $(1, a_1, b_1),\ldots, (m, a_m, b_m)$. Then
$\fn(\mathcal T, c_i) = a_i$.
\end{lem}
\begin{proof}
Recall that for each $1\le i\le m$, the label $a_i$ is the flight number of the cell $c_i$ in the tower diagram
$\mathcal T^{(i)} :=\mathcal T_{\omega^{(i)}}$.  We claim that it is still the flight number of the cell $c_i$ in the tower diagram $\mathcal T$. By definition of the
flight number of a cell, it is the $y$-coordinate of the last cell in the flight path of the cell. Thus, this number changes only if there is a change in the flight path.
To prove the lemma, we argue to prove that possible changes in the flight path of a new cell during repeated applications of the Monk's rule does not effect the flight
number.

Suppose for contradiction, that the label of $c_i$ changes at the $j$-th step for some $1\le i<j\le m$ and this is the first occurrence of a change in flight numbers.
Let $\tau$ be the tower containing $c_i$. In this case, the tower $\upsilon$ with $\omega^{(j)}$-index $a_j$ must be on the right of $\tau$ since otherwise it cannot
effect the flight path of $c_i$. There are several cases to consider.

First, the $j$-th step may append the cell $c_j$ to $\tau$ and does not modify any other tower. Then since
the flight path of $c_i$ changes, the cell $c_j$ must lie just under the flight path of $c_i$. Moreover, $c_j$ must be contained in the Schubert path $\mathcal P^{j-1}$
of $k$ in $\mathcal T^{(j-1)}$. In particular, the Schubert path $\mathcal P^{j-1}$ passes under the cell $c_i$. But this is not possible since necessarily the Schubert
path $\mathcal P^{i}$ passes above the cell $c_i$ and it is clear that if $s<t$ then no cell in $\mathcal P^{t}$ can be below the cell in $\mathcal P^{s}$ contained in
the same vertical strip. Thus, this case cannot occur.

Now suppose that in the $j$-th step, the cells $e_0, e_1,\ldots, e_t$ are appended to the tower $\upsilon$ and $t\ge 1$. Also let $\phi$ be the tower with
$\omega^{j-1}$-index $b_j$, so it is the tower which is shortened at this step. Then there are three more cases to consider. The tower $\phi$ can be equal to $\tau$, on the left, or on the right of $\tau$.

In the first case, if $\tau = \phi$, then the cell $c_i$ is moved to the tower $\upsilon$ following its flight path. Thus, its flight path is shorter but remains the same on
the left of the tower $\upsilon$. Hence its flight number does not change.

Secondly, if $\phi$ is on the left of $\tau$, then the flight path of $c_i$ necessarily passes from the deleted cells in this tower. Thus, although the flight path of $c_i$
from $\phi$ to $\upsilon$ is modified, it remains the same on the left of $\upsilon$. Hence the flight number of $c_i$ does not change.

Finally, $\phi$ can be on the right of $\tau$. We claim that this case is not possible. To see this, let $e_1', \ldots, e_t'$ be the cells deleted from $\phi$ at the $j$-th
step. Therefore, each of these cells must be able to fly from $\phi$ to $\upsilon$. But the tower $\tau$ is between these two towers, so these cells should be able to
pass the tower $\tau$, all with a direct pass or all with a zigzag pass. In order the flight number of $c_i$ to change, the flight path of one of the cells $e_1',
\ldots, e_t'$ should pass from the cell $c_i$. But this means that the height of $\tau$ is more than that of $\phi$. Now this means that, if we slide the cells in the
tower $\tau$ from top to bottom till $c_i$, one of them will erase the top cell of $\phi$. But this is not possible since by the construction, these cells must sit on top
of some tower. Therefore this case cannot occur, as required.
\end{proof}

\begin{lem}
In the set $\widetilde{k_{(m)}\cdot \mathcal T_\omega}$, each shape appears only once.
\end{lem}

\begin{proof}
Let $\mathcal T$ be a tower diagram in $\widetilde{k_{(m)}\cdot \mathcal T_\omega}$ with labels
$(1, a_1, b_1), (2, a_2, b_2),\ldots, (m, a_m, b_m)$. By the previous lemma, $\fn(\mathcal T, c_i) = a_i$, and since
the flight number of a cell only depends on the shape of the tower diagram, the multiset of labels $a_i$ are uniquely determined, up to reordering of
the repeated terms. Also, if a label $a$ appears only once, then it uniquely determines the corresponding triple
$(i,a,b_i)$. Indeed, the labels $a$ and $b_i$ are determined by the generalized flight number of the cell and $i$ is
the order that the cell appears. Thus, the only way to obtain a different label is to reorder the cells with equal flight
number. We claim that, even in this case, there is only one choice, that is, to order such cells from left to right.

To prove this, suppose $i$ is the index that the last repetition occurs, so that $a_{i-r-1} > a_{i-r} = \ldots = a_{i}> a_{i+1}$
for some $r$. Without loss of generality, we can assume that $i=m$ (so $a_i$ is the largest index). First assume
that the towers with label $a_i, \ldots, a_{i-r}$ are ordered as above, from right to left. Then it is easy to see that the
tower diagram $\mathcal T$ is the diagram of $$\omega^{i-r-1}t_{a_{i-r},b_{i-r}}t_{a_{i-r}, b_{i-r+1}}\ldots t_{a_{i-r},b_i}.$$
Notice that $a_{i-r}$ is common for each of the transpositions.

Now it is clear that there is no other way to obtain the same labelled diagram because the equality
$$t_{\alpha,\beta}t_{\alpha, \gamma} = t_{\alpha, \gamma}t_{\beta,\gamma}$$
is the only way to reorder the above sequence of transpositions and in this case the resulting chain is not $k$-Bruhat.

\end{proof}

With this result, we see that the set $\widetilde{k_{(m)}\cdot \mathcal T_\omega}$ is in bijection with the set of
permutations that appear in the product $\mathfrak S_\omega \mathfrak S_{r[k,m]}$. Notice that labels of a given
diagram in $\widetilde{k_{(m)}\cdot \mathcal T_\omega}$ has the property that the flight numbers are non-increasing.
Next we show that we can change the condition on the last components of the labels.

\begin{lem}
Let $\mathcal T\in k_{(m)}\cdot \mathcal T_\omega$ with labels $(1, a_1, b_1), (2, a_2, b_2),\ldots,
(m, a_m, b_m)$ that satisfies $a_1\ge \ldots \ge a_m$. Then $\mathcal T\in \widetilde{k_{(m)}\cdot \mathcal T_\omega}$
if and only if no two labelled cells in $\mathcal T$ are on the same tower.
\end{lem}

\begin{proof}
Suppose $\mathcal T\in \widetilde{k_{(m)}\cdot \mathcal T_\omega}$. Then by definition, the numbers $b_i$ are distinct.
Thus by Corollary \ref{cor: inc-dec}, a tower whose height increases at some step will not decrease again. In particular,
the labelled cells do not move during the construction. Hence no two labelled cells can lie on the same tower.

Conversely, suppose $\mathcal T\in k_{(m)}\cdot \mathcal T_\omega$ with labels $(1, a_1, b_1), (2, a_2, b_2),
\ldots, (m, a_m, b_m)$ such that $a_1\ge \ldots \ge a_m$ and that no tower contains more than one labelled cell. We
have to prove that the numbers $b_i$ are distinct. By Corollary \ref{cor: inc-dec}, it is sufficient to prove that a raised
tower is not modified at a later step. Suppose for contradiction, that this is the case, and $b_i = b_j$ for some indexes $j<i$.
In this case, at the $i$-th step, the towers with $\omega^{(i-1)}$-indexes $a_i$ and $b_i$ are modified. Note that the tower with
$\omega^{(i-1)}$-index $a_i$ is to be raised and the $\omega^{(i)}$-index of this tower is $b_i$. Thus at this step, the
other modified tower $\tau$ has $\omega^{(i-1)}$-index $b_i$ and by our assumption, it must be the raised tower at the $j$-th step.

Notice that the tower $\tau$ has a cell labeled with $(1, a_j, b_j)$. Moreover, the $\omega^{(i)}$-index of $\tau$ is $a_i$
and since no tower contains two labelled cells, the tower $\tau$ still contains the cell with the above label. Now the bottom empty cell over
$\tau$ has flight number $a_i$. Thus we must have $a_i > a_j$. But by our assumption, we must have $a_j\ge a_i$.
\end{proof}

With this lemma, we can identify the set $\widetilde{k_{(m)}\cdot \mathcal T_\omega}$ with the set
$r[k,m]\cdot \mathcal T_\omega$ and complete the proof of Theorem \ref{thm:pieri}.

We summarize the process of obtaining the set $r[k,m]\cdot \mathcal T_\omega$ as an algorithm. Given a permutation $\omega$ with the tower diagram $\Tt$ and
a row permutation $\fr_{k,m}= r[k,m] = s_{k+m}s_{k+m-1} \cdots s_k$.

\smallskip
\noindent\textsc{Pieri's Row Algorithm}

\smallskip
\noindent\textsc{Step 0}: Set $\alpha = m-1$.

\smallskip
\noindent\textsc{STEP 1}: Apply Monk's Algorithm to the pair $(\Tt, k)$ to construct the set $k\cdot \Tt = \{ \Tt^1, \Tt^2,
\ldots, \Tt^m\}$. For each $j$, label the new cell $c_j$ in the diagram $\Tt^j$ by $(1, a_1, b_1)$ where $(a_1, b_1)$ is the
pair of $\omega$-indexes of the modified towers in $\mathcal T$.

\smallskip
\noindent\textsc{STEP 2}: Apply Monk's Algorithm to each pair $(\Tt^j, k)$ for $\Tt^j\in k\cdot \Tt$ to construct the union
set $k_{(m-\alpha+1)}\cdot \mathcal T =\bigcup\{ k\cdot \Tt^j | 1\le j\le m \}$, and label the new cells as in Step 1.

\smallskip
\noindent\textsc{STEP 3}: Construct the set $r[k, m-\alpha +1]\cdot \Tt$ by appending a diagram $\mathcal U$ in $k_{(m-\alpha+1)}\cdot\mathcal T$ if and only if
\begin{enumerate}
\item no two labelled cells in $\mathcal U$ are not in the same tower of $\mathcal U$ and
\item the labels of $\mathcal U$ satisfies $a_1\ge a_2\ge \ldots\ge a_{(m-\alpha+1)}$.
\end{enumerate}

\smallskip
\noindent\textsc{STEP 4}: Set $\alpha = \alpha - 1$. If $\alpha > 0$, set $k\cdot \Tt = r[k, m-\alpha]\cdot \Tt$ and apply
Step 2 with the set $k\cdot \Tt$.

\begin{ex} We illustrate the above algorithm with an example. Let $\omega = s_5 s_6 s_4 s_3 s_4\in S_7$. Then the tower
diagram of $\omega$ with the $\omega$-indexes is given as follows.
\begin{figure}[h]\setlength{\unitlength}{0.5mm}
	\begin{center}
		\begin{picture}(100,90)
		\multiput(0,0)(0,0){1}{\line(1,0){80}}
		\multiput(0,0)(0,0){1}{\line(0,1){60}}
		\multiput(0,10)(2,0){40}{\line(0,1){.1}}
		\multiput(0,20)(2,0){40}{\line(0,1){.1}}
		\multiput(0,30)(2,0){40}{\line(0,1){.1}}
		\multiput(0,40)(2,0){40}{\line(0,1){.1}}
		\multiput(0,50)(2,0){40}{\line(0,1){.1}}
		\multiput(0,60)(2,0){40}{\line(0,1){.1}}
		\multiput(10,0)(0,2){30}{\line(1,0){.1}}
		\multiput(20,0)(0,2){30}{\line(1,0){.1}}
		\multiput(30,0)(0,2){30}{\line(1,0){.1}}
		\multiput(40,0)(0,2){30}{\line(1,0){.1}}
		\multiput(50,0)(0,2){30}{\line(1,0){.1}}
		\multiput(60,0)(0,2){30}{\line(1,0){.1}}
		\multiput(70,0)(0,2){30}{\line(1,0){.1}}
		\multiput(80,0)(0,2){30}{\line(1,0){.1}}
		\put(20,10){\tableaux{{}\\{}}}
		\put(30,00){\tableaux{{}}}
		\put(40,10){\tableaux{{}\\{}}}
		\put(3,-11){$1$}\put(13,-11){$2$}\put(23,-11){$5$}\put(33,-11){$4$}\put(43,-11){$7$}\put(53,-11){$3$}\put(63,-11){$6$}
		\end{picture}
	\end{center}
	\caption{}
\end{figure}

We evaluate the product $\mathfrak S_\omega \mathfrak S_{r[3,2]}$ using the above algorithm. We skip details. First by applying Monk's Algorithm, we
obtain the three diagrams shown in Figure \ref{monk1}. We denote these diagrams by $\mathcal T^1, \mathcal T^2, \mathcal T^3$ form left to right. To fit in the cells,
we write the label $(1, a_1, b_1)$ as $1^{a_1}_{b_1}$

\begin{figure}[h]	\setlength{\unitlength}{0.5mm}
	\begin{center}
		\begin{picture}(80,60)
		\multiput(0,0)(0,0){1}{\line(1,0){80}}
		\multiput(0,0)(0,0){1}{\line(0,1){60}}
		\multiput(0,10)(2,0){40}{\line(0,1){.1}}
		\multiput(0,20)(2,0){40}{\line(0,1){.1}}
		\multiput(0,30)(2,0){40}{\line(0,1){.1}}
		\multiput(0,40)(2,0){40}{\line(0,1){.1}}
		\multiput(0,50)(2,0){40}{\line(0,1){.1}}
		\multiput(0,60)(2,0){40}{\line(0,1){.1}}
		\multiput(10,0)(0,2){30}{\line(1,0){.1}}
		\multiput(20,0)(0,2){30}{\line(1,0){.1}}
		\multiput(30,0)(0,2){30}{\line(1,0){.1}}
		\multiput(40,0)(0,2){30}{\line(1,0){.1}}
		\multiput(50,0)(0,2){30}{\line(1,0){.1}}
		\multiput(60,0)(0,2){30}{\line(1,0){.1}}
		\multiput(70,0)(0,2){30}{\line(1,0){.1}}
		\multiput(80,0)(0,2){30}{\line(1,0){.1}}
		\put(10,20){\tableaux{{}\\{}\\{1^2_5}}}
		\put(30,00){\tableaux{{}}}
		\put(40,10){\tableaux{{}\\{}}}
		\put(3,-11){$1$}\put(13,-11){$5$}\put(23,-11){$2$}\put(33,-11){$4$}\put(43,-11){$7$}\put(53,-11){$3$}\put(63,-11){$6$}
		\end{picture}\hspace{.2in}
			\begin{picture}(80,60)
			\multiput(0,0)(0,0){1}{\line(1,0){80}}
			\multiput(0,0)(0,0){1}{\line(0,1){60}}
			\multiput(0,10)(2,0){40}{\line(0,1){.1}}
			\multiput(0,20)(2,0){40}{\line(0,1){.1}}
			\multiput(0,30)(2,0){40}{\line(0,1){.1}}
			\multiput(0,40)(2,0){40}{\line(0,1){.1}}
			\multiput(0,50)(2,0){40}{\line(0,1){.1}}
			\multiput(0,60)(2,0){40}{\line(0,1){.1}}
			\multiput(10,0)(0,2){30}{\line(1,0){.1}}
			\multiput(20,0)(0,2){30}{\line(1,0){.1}}
			\multiput(30,0)(0,2){30}{\line(1,0){.1}}
			\multiput(40,0)(0,2){30}{\line(1,0){.1}}
			\multiput(50,0)(0,2){30}{\line(1,0){.1}}
			\multiput(60,0)(0,2){30}{\line(1,0){.1}}
			\multiput(70,0)(0,2){30}{\line(1,0){.1}}
			\multiput(80,0)(0,2){30}{\line(1,0){.1}}
			\put(10,10){\tableaux{{}\\{1^2_4}}}
			\put(20,10){\tableaux{{}\\{}}}
			\put(40,10){\tableaux{{}\\{}}}
			\put(3,-11){$1$}\put(13,-11){$4$}\put(23,-11){$5$}\put(33,-11){$2$}\put(43,-11){$7$}\put(53,-11){$3$}\put(63,-11){$6$}
			\end{picture}\hspace{.2in}
		\begin{picture}(80,60)
		\multiput(0,0)(0,0){1}{\line(1,0){80}}
		\multiput(0,0)(0,0){1}{\line(0,1){60}}
		\multiput(0,10)(2,0){40}{\line(0,1){.1}}
		\multiput(0,20)(2,0){40}{\line(0,1){.1}}
		\multiput(0,30)(2,0){40}{\line(0,1){.1}}
		\multiput(0,40)(2,0){40}{\line(0,1){.1}}
		\multiput(0,50)(2,0){40}{\line(0,1){.1}}
		\multiput(0,60)(2,0){40}{\line(0,1){.1}}
		\multiput(10,0)(0,2){30}{\line(1,0){.1}}
		\multiput(20,0)(0,2){30}{\line(1,0){.1}}
		\multiput(30,0)(0,2){30}{\line(1,0){.1}}
		\multiput(40,0)(0,2){30}{\line(1,0){.1}}
		\multiput(50,0)(0,2){30}{\line(1,0){.1}}
		\multiput(60,0)(0,2){30}{\line(1,0){.1}}
		\multiput(70,0)(0,2){30}{\line(1,0){.1}}
		\multiput(80,0)(0,2){30}{\line(1,0){.1}}
		\put(20,10){\tableaux{{}\\{}}}
		\put(30,00){\tableaux{{}}}
		\put(40,10){\tableaux{{}\\{}}}
		\put(50,0){\tableaux{{1^3_6}}}
		\put(3,-11){$1$}\put(13,-11){$2$}\put(23,-11){$5$}\put(33,-11){$4$}\put(43,-11){$7$}\put(53,-11){$6$}\put(63,-11){$3$}
		\end{picture}
	\end{center}
	\caption{}\label{monk1}
\end{figure}

We apply Monk's algorithm to these tower diagrams an obtain the following list of diagrams. Each of the diagrams is denoted using the above naming, so the first
tower diagram from left in Figure \ref{monk11}, is denoted by $\mathcal T^{1,1}$. Now we see that among these 9 tower diagrams, only two of them,
$\mathcal T^{1,1}$ and $\mathcal T^{2,1}$ do not induce a Sottile r-sequence, and each contains a tower with two labelled cells. With respect to their shapes, the
remaining tower diagrams are partitioned as $\{\mathcal T^{1,2}\}, \{\mathcal T^{1,3}, \mathcal T^{3,1}\}, \{\mathcal T^{2,2}\}, \{\mathcal T^{2,3}, \mathcal T^{3,2}\},
\{\mathcal T^{3,3}\}$. Now it is clear that we can choose one tower diagram from each part so that we have $a_1\ge a_2$. Thus, putting
$\omega^{i,j} =\omega_{\mathcal T^{i,j}}$, we get
\[
\mathfrak S_\omega \mathfrak S_{r[3,2]} = \mathfrak S_{\omega^{1,2}} + \mathfrak S_{\omega^{3,1}} + \mathfrak S_{\omega^{2,2}}+ \mathfrak S_{\omega^{3,2}}
+ \mathfrak S_{\omega^{3,3}}.
\]

\begin{figure}[h]\setlength{\unitlength}{0.5mm}
	\begin{center}
		\begin{picture}(80,60)
		\multiput(0,0)(0,0){1}{\line(1,0){80}}
		\multiput(0,0)(0,0){1}{\line(0,1){60}}
		\multiput(0,10)(2,0){40}{\line(0,1){.1}}
		\multiput(0,20)(2,0){40}{\line(0,1){.1}}
		\multiput(0,30)(2,0){40}{\line(0,1){.1}}
		\multiput(0,40)(2,0){40}{\line(0,1){.1}}
		\multiput(0,50)(2,0){40}{\line(0,1){.1}}
		\multiput(0,60)(2,0){40}{\line(0,1){.1}}
		\multiput(10,0)(0,2){30}{\line(1,0){.1}}
		\multiput(20,0)(0,2){30}{\line(1,0){.1}}
		\multiput(30,0)(0,2){30}{\line(1,0){.1}}
		\multiput(40,0)(0,2){30}{\line(1,0){.1}}
		\multiput(50,0)(0,2){30}{\line(1,0){.1}}
		\multiput(60,0)(0,2){30}{\line(1,0){.1}}
		\multiput(70,0)(0,2){30}{\line(1,0){.1}}
		\multiput(80,0)(0,2){30}{\line(1,0){.1}}
		\put(0,30){\tableaux{{}\\{}\\{1^1_5}\\{2^1_5}}}
		\put(30,00){\tableaux{{}}}
		\put(40,10){\tableaux{{}\\{}}}
		\put(3,-11){$5$}\put(13,-11){$1$}\put(23,-11){$2$}\put(33,-11){$4$}\put(43,-11){$7$}\put(53,-11){$3$}\put(63,-11){$6$}
		\end{picture}\hspace{.2in}
		\begin{picture}(80,60)
		\multiput(0,0)(0,0){1}{\line(1,0){80}}
		\multiput(0,0)(0,0){1}{\line(0,1){60}}
		\multiput(0,10)(2,0){40}{\line(0,1){.1}}
		\multiput(0,20)(2,0){40}{\line(0,1){.1}}
		\multiput(0,30)(2,0){40}{\line(0,1){.1}}
		\multiput(0,40)(2,0){40}{\line(0,1){.1}}
		\multiput(0,50)(2,0){40}{\line(0,1){.1}}
		\multiput(0,60)(2,0){40}{\line(0,1){.1}}
		\multiput(10,0)(0,2){30}{\line(1,0){.1}}
		\multiput(20,0)(0,2){30}{\line(1,0){.1}}
		\multiput(30,0)(0,2){30}{\line(1,0){.1}}
		\multiput(40,0)(0,2){30}{\line(1,0){.1}}
		\multiput(50,0)(0,2){30}{\line(1,0){.1}}
		\multiput(60,0)(0,2){30}{\line(1,0){.1}}
		\multiput(70,0)(0,2){30}{\line(1,0){.1}}
		\multiput(80,0)(0,2){30}{\line(1,0){.1}}
		\put(10,20){\tableaux{{}\\{}\\{1^2_5}}}
		\put(20,10){\tableaux{{}\\{2^2_4}}}
		\put(40,10){\tableaux{{}\\{}}}
		\put(3,-11){$1$}\put(13,-11){$5$}\put(23,-11){$4$}\put(33,-11){$2$}\put(43,-11){$7$}\put(53,-11){$3$}\put(63,-11){$6$}
		\end{picture}\hspace{.2in}
		\begin{picture}(80,60)
		\multiput(0,0)(0,0){1}{\line(1,0){80}}
		\multiput(0,0)(0,0){1}{\line(0,1){60}}
		\multiput(0,10)(2,0){40}{\line(0,1){.1}}
		\multiput(0,20)(2,0){40}{\line(0,1){.1}}
		\multiput(0,30)(2,0){40}{\line(0,1){.1}}
		\multiput(0,40)(2,0){40}{\line(0,1){.1}}
		\multiput(0,50)(2,0){40}{\line(0,1){.1}}
		\multiput(0,60)(2,0){40}{\line(0,1){.1}}
		\multiput(10,0)(0,2){30}{\line(1,0){.1}}
		\multiput(20,0)(0,2){30}{\line(1,0){.1}}
		\multiput(30,0)(0,2){30}{\line(1,0){.1}}
		\multiput(40,0)(0,2){30}{\line(1,0){.1}}
		\multiput(50,0)(0,2){30}{\line(1,0){.1}}
		\multiput(60,0)(0,2){30}{\line(1,0){.1}}
		\multiput(70,0)(0,2){30}{\line(1,0){.1}}
		\multiput(80,0)(0,2){30}{\line(1,0){.1}}
		\put(10,20){\tableaux{{}\\{}\\{1^2_5}}}
		\put(30,00){\tableaux{{}}}
		\put(40,10){\tableaux{{}\\{}}}
		\put(50,0){\tableaux{{2^3_6}}}
		\put(3,-11){$1$}\put(13,-11){$5$}\put(23,-11){$2$}\put(33,-11){$4$}\put(43,-11){$7$}\put(53,-11){$6$}\put(63,-11){$3$}
		\end{picture}
	\end{center}
	\caption{}\label{monk11}
\end{figure}

\begin{figure}[h]\setlength{\unitlength}{0.5mm}
	\begin{center}
		\begin{picture}(80,60)
		\multiput(0,0)(0,0){1}{\line(1,0){80}}
		\multiput(0,0)(0,0){1}{\line(0,1){60}}
		\multiput(0,10)(2,0){40}{\line(0,1){.1}}
		\multiput(0,20)(2,0){40}{\line(0,1){.1}}
		\multiput(0,30)(2,0){40}{\line(0,1){.1}}
		\multiput(0,40)(2,0){40}{\line(0,1){.1}}
		\multiput(0,50)(2,0){40}{\line(0,1){.1}}
		\multiput(0,60)(2,0){40}{\line(0,1){.1}}
		\multiput(10,0)(0,2){30}{\line(1,0){.1}}
		\multiput(20,0)(0,2){30}{\line(1,0){.1}}
		\multiput(30,0)(0,2){30}{\line(1,0){.1}}
		\multiput(40,0)(0,2){30}{\line(1,0){.1}}
		\multiput(50,0)(0,2){30}{\line(1,0){.1}}
		\multiput(60,0)(0,2){30}{\line(1,0){.1}}
		\multiput(70,0)(0,2){30}{\line(1,0){.1}}
		\multiput(80,0)(0,2){30}{\line(1,0){.1}}
		\put(0,20){\tableaux{{}\\{1^2_4}\\{2^1_4}}}
		\put(20,10){\tableaux{{}\\{}}}
		\put(40,10){\tableaux{{}\\{}}}
		\put(3,-11){$4$}\put(13,-11){$1$}\put(23,-11){$5$}\put(33,-11){$2$}\put(43,-11){$7$}\put(53,-11){$3$}\put(63,-11){$6$}
		\end{picture}\hspace{.2in}
		\begin{picture}(80,60)
		\multiput(0,0)(0,0){1}{\line(1,0){80}}
		\multiput(0,0)(0,0){1}{\line(0,1){60}}
		\multiput(0,10)(2,0){40}{\line(0,1){.1}}
		\multiput(0,20)(2,0){40}{\line(0,1){.1}}
		\multiput(0,30)(2,0){40}{\line(0,1){.1}}
		\multiput(0,40)(2,0){40}{\line(0,1){.1}}
		\multiput(0,50)(2,0){40}{\line(0,1){.1}}
		\multiput(0,60)(2,0){40}{\line(0,1){.1}}
		\multiput(10,0)(0,2){30}{\line(1,0){.1}}
		\multiput(20,0)(0,2){30}{\line(1,0){.1}}
		\multiput(30,0)(0,2){30}{\line(1,0){.1}}
		\multiput(40,0)(0,2){30}{\line(1,0){.1}}
		\multiput(50,0)(0,2){30}{\line(1,0){.1}}
		\multiput(60,0)(0,2){30}{\line(1,0){.1}}
		\multiput(70,0)(0,2){30}{\line(1,0){.1}}
		\multiput(80,0)(0,2){30}{\line(1,0){.1}}
		\put(10,10){\tableaux{{}\\{1^2_4}}}
		\put(20,10){\tableaux{{}\\{}}}
		\put(30,20){\tableaux{{}\\{}\\{2^2_7}}}
		\put(3,-11){$1$}\put(13,-11){$4$}\put(23,-11){$5$}\put(33,-11){$7$}\put(43,-11){$2$}\put(53,-11){$3$}\put(63,-11){$6$}
		\end{picture}\hspace{.2in}
		\begin{picture}(80,60)
		\multiput(0,0)(0,0){1}{\line(1,0){80}}
		\multiput(0,0)(0,0){1}{\line(0,1){60}}
		\multiput(0,10)(2,0){40}{\line(0,1){.1}}
		\multiput(0,20)(2,0){40}{\line(0,1){.1}}
		\multiput(0,30)(2,0){40}{\line(0,1){.1}}
		\multiput(0,40)(2,0){40}{\line(0,1){.1}}
		\multiput(0,50)(2,0){40}{\line(0,1){.1}}
		\multiput(0,60)(2,0){40}{\line(0,1){.1}}
		\multiput(10,0)(0,2){30}{\line(1,0){.1}}
		\multiput(20,0)(0,2){30}{\line(1,0){.1}}
		\multiput(30,0)(0,2){30}{\line(1,0){.1}}
		\multiput(40,0)(0,2){30}{\line(1,0){.1}}
		\multiput(50,0)(0,2){30}{\line(1,0){.1}}
		\multiput(60,0)(0,2){30}{\line(1,0){.1}}
		\multiput(70,0)(0,2){30}{\line(1,0){.1}}
		\multiput(80,0)(0,2){30}{\line(1,0){.1}}
		\put(10,10){\tableaux{{}\\{1^2_4}}}
		\put(20,10){\tableaux{{}\\{}}}
		\put(40,10){\tableaux{{}\\{}}}
		\put(50,0){\tableaux{{2^3_6}}}
		\put(3,-11){$1$}\put(13,-11){$4$}\put(23,-11){$5$}\put(33,-11){$2$}\put(43,-11){$7$}\put(53,-11){$6$}\put(63,-11){$3$}
		\end{picture}
	\end{center}
	\caption{}
\end{figure}

\begin{figure}[h]\setlength{\unitlength}{0.5mm}
	\begin{center}
		\begin{picture}(80,60)
		\multiput(0,0)(0,0){1}{\line(1,0){80}}
		\multiput(0,0)(0,0){1}{\line(0,1){60}}
		\multiput(0,10)(2,0){40}{\line(0,1){.1}}
		\multiput(0,20)(2,0){40}{\line(0,1){.1}}
		\multiput(0,30)(2,0){40}{\line(0,1){.1}}
		\multiput(0,40)(2,0){40}{\line(0,1){.1}}
		\multiput(0,50)(2,0){40}{\line(0,1){.1}}
		\multiput(0,60)(2,0){40}{\line(0,1){.1}}
		\multiput(10,0)(0,2){30}{\line(1,0){.1}}
		\multiput(20,0)(0,2){30}{\line(1,0){.1}}
		\multiput(30,0)(0,2){30}{\line(1,0){.1}}
		\multiput(40,0)(0,2){30}{\line(1,0){.1}}
		\multiput(50,0)(0,2){30}{\line(1,0){.1}}
		\multiput(60,0)(0,2){30}{\line(1,0){.1}}
		\multiput(70,0)(0,2){30}{\line(1,0){.1}}
		\multiput(80,0)(0,2){30}{\line(1,0){.1}}
		\put(10,20){\tableaux{{}\\{}\\{2^2_5}}}
		\put(30,0){\tableaux{{}}}
		\put(40,10){\tableaux{{}\\{}}}
		\put(50,00){\tableaux{{1^3_6}}}
		\put(3,-11){$1$}\put(13,-11){$5$}\put(23,-11){$2$}\put(33,-11){$4$}\put(43,-11){$7$}\put(53,-11){$6$}\put(63,-11){$3$}
		\end{picture}\hspace{.2in}
		\begin{picture}(80,60)
		\multiput(0,0)(0,0){1}{\line(1,0){80}}
		\multiput(0,0)(0,0){1}{\line(0,1){60}}
		\multiput(0,10)(2,0){40}{\line(0,1){.1}}
		\multiput(0,20)(2,0){40}{\line(0,1){.1}}
		\multiput(0,30)(2,0){40}{\line(0,1){.1}}
		\multiput(0,40)(2,0){40}{\line(0,1){.1}}
		\multiput(0,50)(2,0){40}{\line(0,1){.1}}
		\multiput(0,60)(2,0){40}{\line(0,1){.1}}
		\multiput(10,0)(0,2){30}{\line(1,0){.1}}
		\multiput(20,0)(0,2){30}{\line(1,0){.1}}
		\multiput(30,0)(0,2){30}{\line(1,0){.1}}
		\multiput(40,0)(0,2){30}{\line(1,0){.1}}
		\multiput(50,0)(0,2){30}{\line(1,0){.1}}
		\multiput(60,0)(0,2){30}{\line(1,0){.1}}
		\multiput(70,0)(0,2){30}{\line(1,0){.1}}
		\multiput(80,0)(0,2){30}{\line(1,0){.1}}
		\put(10,10){\tableaux{{}\\{2^2_4}}}
		\put(20,10){\tableaux{{}\\{}}}
		\put(40,10){\tableaux{{}\\{}}}
		\put(50,00){\tableaux{{1^3_6}}}
		\put(3,-11){$1$}\put(13,-11){$4$}\put(23,-11){$5$}\put(33,-11){$2$}\put(43,-11){$7$}\put(53,-11){$6$}\put(63,-11){$3$}
		\end{picture}\hspace{.2in}
		\begin{picture}(80,60)
		\multiput(0,0)(0,0){1}{\line(1,0){80}}
		\multiput(0,0)(0,0){1}{\line(0,1){60}}
		\multiput(0,10)(2,0){40}{\line(0,1){.1}}
		\multiput(0,20)(2,0){40}{\line(0,1){.1}}
		\multiput(0,30)(2,0){40}{\line(0,1){.1}}
		\multiput(0,40)(2,0){40}{\line(0,1){.1}}
		\multiput(0,50)(2,0){40}{\line(0,1){.1}}
		\multiput(0,60)(2,0){40}{\line(0,1){.1}}
		\multiput(10,0)(0,2){30}{\line(1,0){.1}}
		\multiput(20,0)(0,2){30}{\line(1,0){.1}}
		\multiput(30,0)(0,2){30}{\line(1,0){.1}}
		\multiput(40,0)(0,2){30}{\line(1,0){.1}}
		\multiput(50,0)(0,2){30}{\line(1,0){.1}}
		\multiput(60,0)(0,2){30}{\line(1,0){.1}}
		\multiput(70,0)(0,2){30}{\line(1,0){.1}}
		\multiput(80,0)(0,2){30}{\line(1,0){.1}}
		\put(20,10){\tableaux{{}\\{}}}
		\put(30,00){\tableaux{{}}}
		\put(40,10){\tableaux{{}\\{}}}
		\put(50,0){\tableaux{{1^3_6}}}
		\put(60,00){\tableaux{{2^3_8}}}
		\put(3,-11){$1$}\put(13,-11){$2$}\put(23,-11){$5$}\put(33,-11){$4$}\put(43,-11){$7$}\put(53,-11){$6$}\put(63,-11){$8$}\put(73,-11){$3$}
		\end{picture}
	\end{center}
	\caption{}
\end{figure}
\end{ex}

\begin{rem}
For each permutation $\omega_\mathcal T$ with $\mathcal T\in r[k,m]\cdot \mathcal T_\omega$, it is also possible to
choose another diagram $\mathcal T'\in k_{(m)}\cdot \mathcal T_\omega$, in a unique way, such that the corresponding
sequence $a_1,\ldots, a_m$ is non-decreasing. This version coincides with Kogan-Kumar version of Sottile's Theorem
given in \cite{KK}. Note that their version has an extra condition which already holds in our version. Also the version originally conjectured by Bergeron and
Billey in \cite{NS} and proved by Sottile in \cite{F} can be obtained using the techniques of this section.
\end{rem}

We can modify the above arguments to obtain a bijective version of the second part of Sottile's Theorem. We only state the result and leave the proof
as an easy exercise.

\begin{thm}
Let $\omega$ be a permutation in $S_n$, $k, m\in\Nn$ such that $m\le k \le n$. Also let $c[k,m]$ denote the
$(m+1)$-cycle $(k-m+1\, \ldots \, k+1)$. Let $c[k,m]\cdot \mathcal T_\omega$ be the subset of $k_{(m)}\cdot \mathcal T_\omega$
consisting of tower diagrams $\mathcal T$ with labels $(1, a_1, b_1),\ldots, (m, a_m, b_m)$ such that $b_1\ge b_2\ge \ldots\ge b_m$ and no two labelled
cell have the same flight number (equivalently, with distinct numbers $a_i$).
Then
\[
\mathfrak S_\omega \mathfrak S_{c[k,m]} = \sum_{\mathcal T\in c[k,m]\cdot \mathcal T_\omega}
 \mathfrak S_{\omega_\mathcal T}.
\]

\end{thm}

\begin{ex}
In the above example, we applied Monk's algorithm twice with $k=3$. The same set of tower diagrams can be used to evaluate the product
$\mathfrak S_{\omega}\mathfrak S_{c[3,2]}$. Indeed, in this case, we look for tower diagrams with labels determined as in the above theorem. One gets
\[
\mathfrak S_\omega \mathfrak S_{c[3,2]} = \mathfrak S_{\omega^{1,1}} + \mathfrak S_{\omega^{2,1}} + \mathfrak S_{\omega^{3,1}}+ \mathfrak S_{\omega^{3,2}}.
\]

\end{ex}

\end{document}